\documentclass[3p]{elsarticle}
\usepackage{amssymb,amsmath,amsthm}
\usepackage{esint}
\usepackage{mathrsfs}
\usepackage{comment}
\usepackage{bbm,dsfont}
\usepackage{enumerate}
\usepackage[shortlabels]{enumitem}
\usepackage{epsfig}
\newtheorem{theorem}{Theorem}
\newtheorem*{theorem*}{Theorem}
\newtheorem{lemma}[theorem]{Lemma}
\newtheorem{corollary}[theorem]{Corollary}
\newtheorem*{corollary*}{Corollary}

\newtheorem{definition}[theorem]{Definition}
\newtheorem{proposition}[theorem]{Proposition}
\newtheorem{remark}[theorem]{Remark}
\newtheorem{example}[theorem]{Example}
\newenvironment{claim}[1]{\par\noindent\textbf{Claim.}\space#1}{}
\newenvironment{claimproof}[1]{\par\noindent\textit{Proof of the claim.}\space#1}{\hfill $\blacksquare$}

\makeatletter
\def\subsubsection{\@startsection{subsubsection}{3}%
  \z@{1em\@plus 7em}{-.5em}%
  {\normalfont\itshape}}
\makeatother

\usepackage{graphicx}
\usepackage{mathtools}
\usepackage{hyperref}                   
\hypersetup{                            
		colorlinks=true,                    
		citecolor=blue,                     
		urlcolor={blue!80!black}                       
}

\newcommand{\hsobdual}{\dot H^{-m}(\BBR^n)}
\newcommand{\sob}{{H}^m(\BBR^n)}
\newcommand{\hlions}{\dot{W}^m(0,\infty)}

\newcommand{\testfunctions}{\mathscr C_c^\infty(\BBR^n)}
\newcommand{\nontan}{ {\mathcal N}_m}
\newcommand{\nontanbeta}{ {\mathcal N}_{m,\beta}}
\newcommand{\tentspace}{T^{p,2}_{m}}
\newcommand{\pol}{\mathcal{P}_{m-1}}
\newcommand{\Mmaxreg}{\tilde{\mathcal M}_L}
\DeclareMathOperator*{\esssup}{ess\,sup}

\def\BBN {{\mathbb N}}
\def\BBZ {{\mathbb Z}}

\def\BBR {{\mathbb R}}
\def\BBC {{\mathbb C}}

\begin{document}
 	
\begin{frontmatter}
\title{Tent space well-posedness for parabolic Cauchy problems\\ with rough coefficients\tnoteref{t1}}
\tnotetext[t1]{This work contains parts of the author's Master Thesis from the University of Bonn. I gratefully acknowledge the financial support of the DAAD and the Studienstiftung des Deutschen Volkes.}
\author{Wiktoria Zato\'n}
\ead{wiktoria.zaton@math.uzh.ch}

\address{Universit\"at Z\"urich, Institut f\"ur Mathematik, Winterthurerstrasse 190, 8057 Z\"urich, Switzerland.}

\begin{abstract}
	We study the well-posedness of Cauchy problems on the upper half space $\BBR^{n+1}_+$ associated to higher order systems $\partial_t u =(-1)^{m+1}\mbox{div}_m A\nabla ^m u$ with bounded measurable and uniformly elliptic coefficients. We address initial data lying in $L^p$ ($1<p<\infty$) and $BMO$ ($p=\infty$) spaces and work with weak solutions. Our main result is the identification of a new well-posedeness class, given for $p\in(1,\infty]$ by distributions satisfying $\nabla^m u \in T^{p,2}_m$, where $\tentspace$ is a parabolic version of the tent space of Coiffman--Meyer--Stein. In the range $p\in [2,\infty]$, this holds without any further constraints on the operator and for $p=\infty$ it provides a Carleson measure characterization of $BMO$ with non-autonomous operators. We also prove higher order $L^p$ well-posedness, previously only known
for the case $m = 1$. The uniform $L^p$ boundedness of propagators of energy solutions plays an important role in the well-podesness theory and we discover that such bounds hold for $p$ close to $2$. This is a consequence of local weak solutions being locally H\"older continuous with values in spatial $L^p_{loc}$ for some $p>2$, what is also new for the case $m>1$.
\end{abstract}

\end{frontmatter}
\setcounter{section}{0}
\tableofcontents        
\section{Introduction}

	\subsection{Setup and main result} For fixed positive integers $N$ and $m$, consider a homogeneous divergence form elliptic operator $L$ of order $2m$ with bounded measurable complex coefficients, that is 
	\begin{equation*}
	(Lu)_i(t,x)= (-1)^m \sum_{\substack{|\alpha|=|\beta|=m \\ 1\leq j\leq N}} \partial^\alpha (a_{\alpha,\beta}^{i,j}(t,x) \partial^\beta u_j)(t,x) \quad\mbox{for } (t,x)\in\BBR^{n+1}_+ \mbox{ and } i=1,\dots,N, 
	\end{equation*}
	where $u=(u_1,\dots,u_N)$. We assume that the ellipticity estimates in the sense of the G{\aa}rding inequality hold uniformly in $t>0$ (cf.\ Section \ref{sec:ellipticity} for the precise definition and a discussion of related ellipticity assumptions). A prototype example is the polyharmonic operator $L=(-1)^m\Delta^m$.
	We study the associated parabolic equation (or system if $N>1$)
	\begin{equation}
	\label{eq:the parabolic equation}
	\partial_t u_i(t,x) = - (Lu)_i(t,x) \quad\mbox{for } (t,x)\in\BBR^{n+1} \mbox{ and } i=1,\dots,N
	\end{equation}
	interpreted in the weak sense (this notion is recall in Definition \ref{def:weak_sol}) and are particularly interested in the well-posedness of the associated Cauchy problem  
	\begin{equation} \label{eq:Cauchy problem}
	u \in X \mbox{ is a global weak solution to }  \eqref{eq:the parabolic equation} \quad\mbox{and}\quad u(0,\cdot) = u_0 \in Y
	\end{equation}	
	for an initial data space $Y\subseteq  L^1_{loc}(\BBR^n)$ and some solution space $X$. Typical choices for $Y$ are the $L^p$ spaces and the space of bounded mean oscillations $BMO(\BBR^n)$, the latter of which allows rough data, see Section \ref{sec:Campanato} for its definition. 
	
	We say that \eqref{eq:Cauchy problem} is well-posed for the pair of semi-normed spaces $(Y,X)$ if the following holds
	\begin{enumerate}[label=\upshape{(\roman{*})}]
		\item For every weak solution $u\in X$ of \eqref{eq:the parabolic equation} there exists $u_0\in L^1_{loc}(\BBR^n)$, such that for any compact set $K\subseteq \BBR^n$
		\[\lim_{t\to 0}\|u(t,\cdot)-u_0\|_{L^1(K)}=0.\]
		This $u_0\in L^1_{loc}(\BBR^n)$ is necessarily unique and we call it the \textit{trace} of $u$ at $t=0$ or the \textit{initial datum}. Note that it suffices to find the trace in some $L^p_{loc}(\BBR^n)$ for $p\ge1$.
		\item Given $u_0 \in Y$, the Cauchy problem \eqref{eq:Cauchy problem} can be solved uniquely in $X$.
		\item There exists a constant $C>0$ with $\|u\|_X\leq C\|u_0\|_Y$ for any initial data $u_0\in Y$ and the corresponding solution $u\in X$.
	\end{enumerate}
	
	In their simplest form our well-posedness results can be summarized as follows, cf.\ Theorems \ref{thm:global_well}, \ref{thm:wp pol}, \ref{thm:wp p big}, \ref{thm:tent space wp small p}, \ref{thm:uniform_bdd_p_close_to_2}.
	\begin{theorem}\label{thm:main thm}
		There exists $\varepsilon>0$ depending only on the ellipticity constants, $m$, $N$ and the dimension, such that the Cauchy problem \eqref{eq:Cauchy problem} is well-posed for 
		\begin{enumerate}[label={\upshape(\roman*)}]
			\item $p\in (2-\varepsilon, 2+\varepsilon)$ and $Y=L^p(\BBR^n)$ with $X=L^\infty(0,\infty;L^p(\BBR^n))$.
			\item $p\in[2,\infty)$ and $Y=L^p(\BBR^n)$ with $X=\{u\in \mathscr D'(\BBR^{n+1}_+)\mid\nabla^mu\in \tentspace\}$.
			\item $p=\infty$ and $Y=BMO(\BBR^n)$ with $X=\{u\in \mathscr D'(\BBR^{n+1}_+)\mid\nabla^mu\in \tentspace\}$.
		\end{enumerate} 
		Here, $\tentspace$ denotes a prabolic version of the tent space of Coifman--Meyer--Stein, see Definition \ref{def:tent}. 
		In (i) the unique solution satisfies $u\in \mathscr C_0([0,\infty); L^p(\BBR^n))$. Furthermore, (ii) holds also for $p\in (2-\varepsilon,2)$ if $L$ is pointwise elliptic, that is the condition 
		\begin{equation*}
		\text{Re} \left(\sum\nolimits_{\substack{|\alpha|=|\beta|=m \\ 1\leq i, j\leq N}} a_{\alpha,\beta}^{i,j}(t,x) \xi_\beta^j\overline{\xi_\alpha^i}\right)\ge \lambda \|\xi\|^2\end{equation*} 
		is satisfied for some $\lambda >0$, any vector $(\xi_\alpha)_{|\alpha|=m}$ with entries in $\BBC^{N}$ and almost every $(t,x)\in \BBR^{n+1}_+$.
	\end{theorem}
We refer to (ii) and (iii) of Theorem \ref{thm:main thm} as the tent space well-posedness. Below, we survey known well-posedness results for both second and higher order autonomous problems and then proceed with a thorough discussion of Theorem \ref{thm:main thm} and an explanation of our methods. 
	
		\subsection{Previous results on non-autonomous Cauchy problems}
	For non-autonomous second order operators, the issue of existence and uniqueness of solutions with $L^p$ initial data was studied in detail by Auscher, Monniaux, and Portal \cite{AMP15}. Let us first recall from \cite[\S1]{AMP15}	that for the heat equation $X=L^\infty(0,\infty;L^p(\BBR^n))$ and $X=\{u\mid u^* \in L^p(\BBR^n)\}$ are well-posedness classes for $L^p(\BBR^n)$ initial data with $p\in(1,\infty)$, where $u^*$ denotes the non-tangential maximal function $$u^*\colon x\mapsto \displaystyle{\sup_{(t,y)\colon |y-x|<\sqrt t}|u(t,y)|}.$$ In the latter case, existence follows from the maximal function characterization of Hardy spaces (cf.\ \cite[Chap. III]{S93} and \cite{FS72}), while the uniqueness is a consequence of the maximum principle. The challenge in the case of complex, merely bounded measurable coefficients or systems, is that many of the well-known classical methods, like the maximum principles, break down and different strategies are needed to approach the problem of unique solvability.
	
	In \cite{AMP15} the authors presented novel techniques, which lead to new well-posedness classes also for the real equation. First, they settled the energy well-posedness in the space of distributions $u$ satisfying $\| \nabla u \|_{ L^2(\BBR^{n+1}_+)}<\infty$ and used the established uniqueness of energy solutions to define the family of propagators 
			\begin{equation*}
			\{\Gamma(t,s) \mid 0\leq s\leq t<\infty\}\subseteq \mathscr L (L^2(\BBR^n)),\end{equation*} 
			corresponding to the semigroup $(e^{-(t-s)L})_{0\leq s\leq t<\infty}$ for autonomous operators $L$, see \cite[\S 3]{AMP15}. The propagators were further used to construct solutions with $L^p(\BBR^n)$ initial data, while the main step towards the uniqueness of solutions relied on showing that the evolution of solutions with controlled growth in the spatial variable is governed by the propagators. The precise condition is 
				\begin{equation*}
		\int_{\BBR^n}\left(\int_a^b\int_{B(x,\sqrt{b})} |u(t,y)|^2 dy dt \right) ^{1/2}e^{-\gamma|x|^{2}}dx< \infty,
		\end{equation*} 
		where $\gamma< c/(b-a)$ and the constant $c$ is determined by the ellipticity constants, cf.\ \cite[Theorem 5.1]{AMP15}. With these methods, the authors derived analogous results to case of the heat equation, where the non-tangential function was replaced by a parabolic version of the Kenig--Pipher maximal function	 \[\tilde {\mathcal N} u(x)\coloneqq \sup_{\delta>0}\left(\fint_{\delta/2}^\delta\fint_{B(x,\sqrt{\delta})}|u(t,y)|^2dydt\right)^{1/2}\] introduced in \cite{KP93} in the context of elliptic equations, cf.\ \cite[Theorem 5.9, Proposition 5.11]{AMP15} and \cite[Theorem 5.4, Corollary 5.10]{AMP15}. For the well-posedness in the class $X=L^\infty(0,\infty;L^p(\BBR^n))$ or the non-tangential space if $p<2$, the uniform $L^p(\BBR^n)$ boundedness of the propagators plays a crucial role.  
			The uniform boundedness of propagators is known to hold for instance for coefficients with bounded variation in time, or small perturbations from the autonomous case, see Section 6 in loc.\ cit.\ In this work we provide unconditional bounds of the propagators in a range of exponents around $p=2$. We address this result in the final part of the introduction.
	
The only available well-posedness results for higher order complex parabolic systems assume further regularity of the coefficients. Systems in non-divergence form were extensively studied by Solonnikov, who presented the main developments on this subject in his monograph \cite{S65}. In particular, Solonnikov established unique solvability of the Cauchy problem on $[0,T]\times\BBR^n$ in certain time-space H\"older and $p$-Energy classes for systems with coefficients with H\"older continuous derivatives, cf. Theorems 4.10 and 5.5 in loc.\ cit. The methods used relied on the well-known technique of freezing the coefficients, dealing with the constant coefficients case first and using the continuity (of the derivatives) of the coefficients to bound the error term. This strategy allows not only mixed time-space derivatives, but also domains that are not necessarily bounded or cylindrical. 
	
		 We also mention the work of Dong and Kim \cite{DK11}, who studied the $L^p$ solvability of parabolic systems in both divergence and non-divergence form under weaker ellipticity assumptions, but additional (small spatial $BMO$ norm) regularity on the coefficients. 
	
		\subsection{Ansatz} 
	The starting point of this paper was to observe that it is possible to apply the novel approach from \cite{AMP15} to the higher order case, as the methods exploit energy estimates. This allows us to treat the case of merely bounded coefficients. 
	With begin with an alternative approach to the energy well-posedness in the class of distributions satisfying $\nabla^m u\in L^2(\BBR^{n+1}_+)$ and define the family of propagators corresponding to our equation. By careful generalizations of the energy estimates and the $L^2$ off-diagonal decay of the propagators from \cite{AMP15}, in this work we give the tools necessary to extend the results of loc.\ cit.\ to the higher order case, which produces new results even for the polyharmonic operator.
	
	Combining the aforementioned estimates, we establish the conservation property for propagators
	\[\Gamma(t,s)P=P \quad\mbox{in}\quad L^2_{loc}(\BBR^n) \quad\mbox{for any  } P \in \mathcal  P_{m-1}.\]
	Here $\pol$ denotes the set of polynomials on $\BBR^n$ of degree less than $m$. This is new under the weak ellipticity assumption and gives an alternative proof for the second order case, where it was deduced from the well-posedness results in the non-tangential spaces. 
	
	A crucial observation is that under mild growth assumptions the evolution of weak solutions to \eqref{eq:the parabolic equation} at positive times is indeed governed by the propagators and can be proven verbatim along the proof for the second order case known from \cite{AMP15}. For any open bounded interval $(a,b)\subseteq (0,\infty)$ supposing
	\[ \int_a^b \int_{\BBR^n}|u(t,y)|e^{-\overline \gamma|y|^{2m/(2m-1)}}dydt<\infty\] 
 with some constant $\overline \gamma >0$ allows to obtain the identity $$u(t,\cdot)=\Gamma(t,s)u(s,\cdot)\coloneqq u(s,\cdot)\circ \Gamma(t,s)^*$$ in the sense of distributions for all $a<s\leq t <b$,
see Theorem \ref{thm:int_repr} and Remark \ref{rem:equivalent 5.1}.  In particular, in the prototype example of the polyharmonic operator $L=(-1)^m \Delta^m$, we find that the propagation of global solutions satisfying \begin{equation}
	\label{eq:ptw bound}
	|u(t,x)|\lesssim e^{c|x|^{2m/(2m-1)}} \quad \mbox{for all }x\in \BBR^n
	\end{equation} locally in time for some (varying) $c>0$, is governed by the semigroup at times $t>0$. Precisely, we require that for any $t_0>0$ there exist a constant $c>0$ and a $\delta$-neighborhood of $t_0$ in $\BBR_+$, such that \eqref{eq:ptw bound} holds for all $t\in (t_0-\delta,t_0+\delta)$.
	This exponential growth assumption does not yet imply the full uniqueness withing the class of solutions satisfying such a bound. We will need to impose stronger conditions, i.e.\ restrict ourselves to certain function spaces, in order to uniquely identify any potential solution belonging to this class via its trace at $t=0$.
\subsection{The tent space well-posedness}
As the parabolic tent spaces are an object central to this work, we begin with a precise definition. For the classical introduction to tent spaces, see Coifman, Meyer and Stein \cite{CMS85}.
\begin{definition}
	\label{def:tent}
	Let $p\in (0,\infty)$. The parabolic tent space $\tentspace$ consists of measurable functions $f\colon \BBR^{n+1}_+\to \BBC$, for which the square function 
	\[x\mapsto \left(\int_0^\infty \fint _{B(x,\sqrt[2m]{t})}|f(t,y)|^2 dy dt\right)^{1/2}\] belongs to $L^p(\BBR^n)$. For $p=\infty$ we define the space $T_m^{\infty,2}$ accordingly via the usual modified condition $$\sup\limits_{B\colon x\in B}\left(\int_0^{r_B^{2m}} \fint _{B}|f(t,y)|^2 dy dt\right)^{1/2} \in L^\infty(\BBR^n).$$
\end{definition}
\noindent If we equip $\tentspace$ with the $L^p(\BBR^n)$ norm of the corresponding object arising in the above definition, we obtain a Banach space if $1<p\leq \infty$. To explain the use of the notion \textit{tent} spaces, we recall that in the classical definition from \cite{CMS85}, the domain of integration is the cone $C_x=\{(t,y)\in \BBR^{n+1}_+\mid |x-z|<t\}$. For any closed set $F\in \BBR^n$, the complement of the union $\cup_{x\in F} C_x$, resembles a tent over $\BBR^n\setminus F$. Further, estimates involving tent spaces are often referred to as square function estimates (see the discussion after Theorem 6 in the original paper \cite{CMS85}). We return to this point of view later on and refer the interested reader to the proof of Proposition \ref{prop:tent_kenig_comp_smallp}, where we highlight this connection.

By Fubini's Theorem, for all $m\in \BBN_+$ it holds $T^{2,2}_m = L^2(\BBR^{n+1}_+)$. Hence, the tent space well-posedness for $p=2$ boils down to the energy case and was proven for second order equations in \cite{AMP15}. In Lemma \ref{lem:trace integrability}, we show that any global weak solution $u$ to \eqref{eq:the parabolic equation} satisfying $\nabla^m u\in L^2(\BBR^{n+1}_+)$ can be written as $u=v+P$ for unique polynomial $P\in \pol$ and $v\in \mathscr C_0([0,\infty);L^2(\BBR^n))$. Thus, we obtain that  $\nabla^m u\in L^2(\BBR^{n+1}_+)$ together with the weak decay assumption $u(t,\cdot)\in L^2(\BBR^n)$ for some $t>0$ implies uniqueness of solutions of the Cauchy problem with $L^2(\BBR^n)$ initial data. 

It is natural to ask, whether there is a similar condition we can impose on $\nabla^m u$ in order to capture the $L^p(\BBR^n)$ initial data, see also the question risen in the case of the heat equation in \cite[\S1]{AMP15}. In this work we give an affirmative answer to this question in the range $p\in(1,\infty)$ and cover also the case $p=\infty$, which turns out to provide the right condition for $BMO(\BBR^n)$ initial data.

It was to be expected that solutions to \eqref{eq:the parabolic equation} admitting $L^p(\BBR^n)$ initial data satisfy $\nabla^m u\in \tentspace$. Indeed, for the heat equation different characterizations of Hardy spaces $\mathcal H^p$ enable to compare $\|u^*\|_{L^p}$ and $\|\nabla u \|_{T^{p,2}_1}$ for solutions of the form $u(t, \cdot)=e^{t\Delta}f$ and $f\in L^2(\BBR^n)$. As mentioned above, this is equivalent to finding $\mathcal H^p$ to $L^p$ estimates for the square function
	\[f\mapsto \left( x\mapsto \left(\int_0^\infty \fint _{B(x,\sqrt{t})}|\nabla e^{t\Delta}f(t,y)|^2 dy dt\right)^{1/2}\right).\]
In the non-autonomous setting analogous estimates hold for global weak solutions $u$ to \eqref{eq:the parabolic equation} of the form $u_f(t,\cdot)=\Gamma(t,0)f$ with $f\in L^2(\BBR^n)$. In particular, following \cite{AMP15}, we derive
\begin{equation}\label{eq:intro1}
\|u_f\|_{X^p_m}\coloneqq \|\nontan u_f\|_{L^p(\BBR^n)}\lesssim \|\nabla^m u_f \|_{\tentspace},
\end{equation} where $p\in(\frac{n}{n+m},\infty)$ and $u\mapsto\nontan u$ is the natural adaptation of the non-tangential maximal function $\tilde{\mathcal{N}}u$ above to the homogeneity of the equation (see Section \ref{sec:tent} for the precise definitions). The converse inequality to \eqref{eq:intro1} holds for any global weak solution to \eqref{eq:the parabolic equation} if $p\in[1,2)$ and $L$ is strongly elliptic (cf.\ \eqref{eq:strong lower ellipticity bounds}) or $p\in [2,\infty)$. In general, this cannot be true for inequality \eqref{eq:intro1}, as it fails for $u= P\in \pol$. However, for $p>1$ we can show that this is essentially the only counterexample. 

\subsubsection*{The case $p\in [2,\infty]$.} 

Our main result for $p\in[2,\infty]$ is the following, cf.\ Theorem \ref{thm:wp pol} for $p>2$ and Theorem \ref{thm:global_well} together with Lemma \ref{lem:trace integrability} for $p=2$.
\begin{theorem}
\label{intr_thm:wp pol}
			Let $p\in[2,\infty]$. For a distribution $u\in \mathscr D'(\BBR^{n+1}_+)$ it is equivalent
			\begin{enumerate}[label={\upshape(\roman*)}]
				\item $u$ is a global weak solution of \eqref{eq:the parabolic equation} and $\nabla^mu\in \tentspace$.
				\item There are unique $ f\in Y$ and $P\in \pol$ (unique up to a constant if $p=\infty$), such that $u(t,\cdot)-P=\Gamma(t,0)f$ in $L^2_{loc}(\BBR^n)$ for $t>0$,
			\end{enumerate} where $Y=BMO(\BBR^n)$ if $p=\infty$ or $Y=L^p(\BBR^n)$ if $p\in [ 2,\infty)$.
			Moreover, it holds
			$$
			\|\nabla^m u\|_{\tentspace}\sim\|f\|_{Y}.
			$$
\end{theorem}
		
		 We can treat the cases $p\in (2,\infty)$ and $p=\infty$ simultaneously, as we exploit the fact that only in this range of exponents $\tentspace$ can be normed by
		 $$\tentspace \ni f \mapsto \Bigg\|\sup_{\footnotesize\substack{z\in \BBR^n,\,r>0\colon \\x\in B(z,r)}}\left(\int_0^{r^{2m}} \fint _{B(z,r)}|f(t,y)|^2 dy dt\right)^{1/2} \Bigg\|_{L^p(\BBR^n)}.$$
	 Now it is easy to see that if $\|\nabla^m u\|_{\tentspace}<\infty$ holds, then $\nabla^m u$ is square integrable over cylinders $[0,1]\times B(x_0,1)$, where $x_0\in \BBR^n$. We use this information, combined with the Poincar\'e inequality and the equation, to obtain $u\in L^2(0,1; B(x_0,1))$ and deduce the existence of a $L^2_{loc} (\BBR^n)$ trace $u_0$. As we can control the $L^2(B(x_0,R))$-averages of $u_0$ modified by some polynomial $P_{B(x_0,R)}$, it is convenient to first prove the well-posedness for some Campanato type spaces $Y$, which we call polynomial $L^p$ and polynomial $BMO$ spaces, see Definitions \ref{def:BMO higher order} and \ref{def:L^p higher order}. It is non-trivial to show that those spaces equal the usual $L^p(\BBR^n)$ and $BMO(\BBR^n)$ spaces up to polynomials, cf.\ Section \ref{sec:Campanato}. The existence of solutions with initial data in $Y$ is easy due to the $L^2$ off-diagonal decay of the propagator and the bound $\|\nabla^m u_f\|_{\tentspace}\lesssim \|f\|_Y$ is proven analogously to the classical estimates of C. Fefferman and Stein, cf.\ \cite[Chap.\ VI, \S4.3]{S93}. This technique relies on the conservation property for polynomials mentioned above. 
	
	Recalling that a Borel measure $\mu$ on $\mathcal{B}(\BBR^{n+1}_+)$ satisfying
			$$\sup_{x\in \BBR^n} \sup_{r>0}\frac{1}{r^n}\mu \left((0,r)\times B(x,r)\right)<\infty$$ is called Carleson, we obtain after rephrasing Theorem \ref{intr_thm:wp pol} a Carleson measure characterization of $BMO(\BBR^n)$.
\begin{corollary}
	 			\label{intr_cor:Carleson measure characterisation}
	 			For $f\in L^2_{loc}(\BBR^n)$ it is equivalent
	 			\begin{enumerate}[label={\upshape(\roman*)}]
	 				\item There exists a global weak solution $u$ to \eqref{eq:the parabolic equation}, for which $$ d\mu(x,t)=|t^m\nabla^m u(t^{2m},x)|^2\frac{dxdt}{t}$$ is a Carleson measure and the $L^2_{loc}(\BBR^n)$ trace of $u$ is given by $f$.
	 				\item There exists a polynomial $P\in \pol$ such that $f-P\in BMO(\BBR^n)$.
	 			\end{enumerate}
	 			Moreover, $\|\nabla^m u\|_{T^{\infty,2}_m}\sim\|f-P\|_{BMO}.$
	 		\end{corollary}
	 	In particular, for $m=1$, this result complements the well-known Carleson measure characterization of $BMO(\BBR^n)$ by C. Fefferman and Stein (recalled in Proposition \ref{prop:carleson}). The variable $t^{2m}$ reminds of the scaling properties of our equation and will appear frequently throughout this work. 

	According to Corollary \ref{intr_cor:Carleson measure characterisation}, any operator $L$ as above leads to an equivalent characterization of $BMO(\BBR^n) + \mathcal P_{m-1}$. For $m=1$ and some class of coefficients, this was addressed in the survey article \cite{KL13} and similar ideas for the bi-Laplacian appeared earlier in \cite{KL12}.\ The $BMO(\BBR^n)$ well-posedness of parabolic equations is desired, as this allows rough initial data, cf.\ the famous result for the Navier--Stokes equations of Koch--Tataru \cite{KT01}. 
			
\subsubsection*{The case $p\in (1,2]$.}	For $p\in(1,2]$ and any global weak solution $u$ to \eqref{eq:the parabolic equation} with $\|\nabla^m u \|_{\tentspace} <\infty$, we prove the existence of a unique distributional trace $u_0$ and a unique polynomial $P\in \pol$, for which it holds
			\begin{equation}\label{eq:06}
			\sup_{t\ge 0}\|u(t)-P\|_{L^p(\BBR^n)}\lesssim \|\nabla^m u \|_{\tentspace}.
			\end{equation}

Thus, by combining the well-posedness result for the class $L^\infty(0,\infty;L^p(\BBR^n))$ together with the bound $\|\nabla^m u \|_{\tentspace} \lesssim \|u\|_{X^p_m}$ (if $L$ is strongly elliptic) we obtain the following, cf.\ Theorem \ref{thm:tent space wp small p}.
	\begin{theorem}
		Let $1\leq p<r\leq 2$. Assume the pointwise ellipticity bounds \eqref{eq:strong lower ellipticity bounds} and that the propagators are uniformly bounded on $L^p(\BBR^n)$. Then for a distribution $u\in \mathscr D'(\BBR^{n+1}_+)$ it is equivalent
		\begin{enumerate}[label={\upshape(\roman*)}]
			\item It holds $\nabla^m u\in T^{r,2}_m$ and $u$ is a global weak solution of \eqref{eq:the parabolic equation}.
			\item There are unique $ f\in L^r(\BBR^n)$ and $P\in\pol$, such that $u(t,\cdot)-P=\Gamma(t,0)f$ in $L^r(\BBR^n)$ for $t>0$.
		\end{enumerate} In this case, we have $u-P\in \mathscr  C_0([0,\infty);L^r(\BBR^n))$ and
		\[\|f\|_{L^r}\sim \|u-P\|_{L^\infty(L^r)}\sim\|u-P\|_{X^r_m}\sim\|\nabla^m u \|_{T^{r,2}_m}.\] 
	\end{theorem}
We could deduce the continuity of the solution with values in $L^r(\BBR^n)$ as we are working in the open interval $r\in (p,2)$. For the same conclusion in the range $p>2$, see Remark \ref{rem:continuity big p}.

	We stress that distinct proofs are needed to treat the cases $p\in(1,2]$ and $p\in (2,\infty]$. For $p\in(1,2]$ the condition $\|\nabla^m u \|_{\tentspace} <\infty$ implies that $\nabla^m u \in L^2([\varepsilon,\infty]\times \BBR^n) $ for any $\varepsilon >0$, whence it can be deduced $u\in \mathscr C((0, \infty); L^2(\BBR^n))$ and we use the weak formulation of the equation \eqref{eq:the parabolic equation} to find a distributional limit of $\{u(t)\}_{t>0}$ as $t\to 0$. As $p\leq2$, to prove \eqref{eq:06}, it is enough to bound the $L^p$ norm of weighted $L^2$-averages of $u(t)-P$ for some fixed polynomial $P$ and all $t>0$, which can be done similarly as in the previous case.
	
	\subsection{Uniform boundedness of propagators} Finally, we study the regularity question of weak solutions to \eqref{eq:the parabolic equation}, from which we derive new (even in the second order case) results on the boundedness of the propagators. As mentioned above, for coefficients satisfying certain growth assumptions in time, such bounds were established in \cite{AMP15} for $p<2$. The case $p>2$ was left open. However, in the autonomous case it is well known that the semigroup $(e^{-(t-s)L})_{0\leq s<t<\infty}$ satisfies uniform $ L^p(\BBR^n)$-bounds for each exponent $p$ in some (maximal) open interval $(p_-(L),p_+(L))$ (cf.\ \cite{A07}). Moreover, we have $$p_-(L)< q_-<2<q_+< p_+(L)$$ for some exponents $q_-$ and $q_+$, which depend on $L$ only through the order $m$ and the dimension $n$. In the case of propagators associated to non-autonomous operators, one cannot use semigroup theory methods to investigate this problem, but we were able to combine our methods with the strategy from the recent work of Auscher, Bortz, Egert, and Saari \cite{ABES18} to establish the following result, cf.\ Theorem \ref{thm:uniform_bdd_p_close_to_2}.
	\begin{theorem} \label{thm:i4}
		There exists $\varepsilon >0$ depending only on the ellipticity constants, $m$, $N$ and the dimension, such that the family of propagators $\{\Gamma(t,s)| \; 0\leq s\leq t<\infty\}$ is uniformly bounded on $L^p(\BBR^n)$ for all $p\in [1,\infty]$ with 
		$p\in(2-\varepsilon, 2+\varepsilon)$.
	\end{theorem}
	
	The main result of \cite{ABES18} concerns the regularity of local weak solutions to \eqref{eq:the parabolic equation} for $m=1$ and states that they are locally $(1/2-1/p)$-H\"older continuous with values in spatial $L^p_{loc}(\BBR^n)$ for any $2<p<q$ and some exponent $q$ depending on the ellipticity constants and the dimensions in a non-explicit way. We show in Section \ref{sec:bound reduction} that pointwise bounds of $L^p_{loc}(\BBR^n)$ valued solutions by their $L^2$ norms on certain cylinders imply uniform $L^p(\BBR^n)$ bounds on the propagators. This is due to a result from \cite{BK05} about the boundedness of general linear operators on $L^2(\BBR^n)$.
	
	We mention that unless the coefficients are real and $N=1$, we cannot hope to obtain H\"older continuity of local weak solutions with respect to the parabolic distance as in the Moser--Nash regularity theory \cite{M64, N57}.\ We refer to \cite[\S1]{ABES18} for a discussion of counterexamples and an overview over the subject for $m=1$ and rough coefficients. We remark that if $N,\,m\ge 1$ are arbitrary, but we impose higher H\"older regularity of the coefficients, equation \eqref{eq:the parabolic equation} has a unique solution in some H\"older class by \cite[Theorem 4.10]{S65}. 
	
	The key idea in \cite{ABES18} was to pass to the global setup, that is to a weak solution $v$ of an inhomogeneous equation on the full space by multiplying the investigated solution with a cut-off function. This allows to make sense of the half-time derivative $D_t^{1/2}v$, whose $L^2(\BBR^{n+1})$-norm can be then controlled due to the hidden coercivity of the equation on appropriate energy space. Then, an abstract interpolation result by \v{S}ne\u{\i}berg \cite{S74} is used to conclude the higher integrability for $p>2$. The desired H\"older continuity and the pointwise bounds in time follow then from the results on Campanato spaces. We adapt those methods to the higher order case with the difference that, instead of studying the inhomogeneous equation, as in \cite{ABES18}, we use the previously derived properties of weak solutions to \eqref{eq:the parabolic equation} and carry them over to the extension $v$. We obtain the following regularity result (cf.\ Theorem \ref{thm:higher integrability of solutions}).
	
	 \begin{theorem}\label{thm:i5}
		There exists $\varepsilon>0$ depending only on the ellipticity constants, $m$, $N$ and the dimension, so that any global weak solution $u\in L^2_{loc}(0,\infty; H^m_{loc}(\BBR^n))$  of \eqref{eq:the parabolic equation}  is locally bounded and H\"older continuous in time with values in spatial $L^p_{loc}$ for any $2<p<2+\varepsilon$. Moreover, $u\in L^p_{loc}(0,\infty; W^{m,p}_{loc}(\BBR^n))$.
	\end{theorem}
	
	We mention that the higher integrability of $\nabla^m u$ has already been obtained by Giaquinta and Struwe for $m=1$ \cite{GS82} and also by B\"ogelein \cite{B08} in the real-valued case, but general $m$. 
	
\subsection{Structure of the paper}
	In Section \ref{sec:prel} we fix the notation, state the ellipticity assumptions, introduce function spaces appearing in this work and review the semigroup theory. We continue with the derivation and consequences of the a priori energy estimates in Section \ref{sec:energy estimates}.
	
	In Section \ref{sec:traces tent space} we demonstrate that the tent space condition $\nabla^m u \in \tentspace$ is sufficient to find and control the distributional trace of a global weak solution $u$. In Section \ref{sec:energy solutions} we use obtained trace estimates to show the well-posedness of energy solutions and introduce the family of propagators. 
	
	In Section \ref{sec:existence results} we construct solutions with initial data in $L^p$ with $p>2$ or $BMO$ by exploiting the $L^2$ off-diagonal decay of the propagators. If $p<2$, we need to assume the uniform boundedness of propagators and follow closely \cite{AMP15}. 
	
	The uniqueness of solutions is addressed in Section \ref{sec:uniqueness results}.\ We begin with the analogous interior representation result as in \cite[\S5]{AMP15} and consequently establish the tent space well-posedness.
	
	The final section addresses the uniform boundedness assumption for the propagators. First, we provide some examples, which again follows loc.\ cit. Second, we prove that with no extra assumption on the coefficients, the propagators satisfy uniform $L^p(\BBR^n)$ bounds for exponents $p$ in some neighborhood of $2$, which depends on ellipticity and dimensions only. We also obtain the result on H\"older regularity in time mentioned above.
	
	The Appendix contains the proof of the estimate $\|u\|_{X^p_m}\lesssim \|\nabla^mu\|_{T^{p,2}}$, as well as its converse, if $p\ge 2$ or $p\in (1,2)$ and the strong ellipticity estimates hold.
	
	\subsection{Remarks} 
	The reader might have noticed that our well-posedness results implicitly exclude the case of $L^1(\BBR^n)$ initial data. Indeed, this case  is rather delicate, cf.\ \cite[\S9]{AMP15}.\ Analogous statements hold as well for the higher order case, but for the sake of brevity, we leave their proof as an exercise to the interested reader. The case $p<1$ is completely open and an optimal condition implying the uniform boundedness of propagators is also unknown.

	\section{Review}
	\label{sec:prel}
	\setcounter{theorem}{0} \setcounter{equation}{0}
		\subsection{Notation}
		
		We denote the positive integers by $\BBN_+$, $n\in \BBN_+$ is the spatial dimension and the parabolic half-space is given by $\BBR^{n+1}_+=\{(t,x)\in \BBR^{n+1}\mid t>0\}$. Also, $(a,b)\subseteq (0,\infty)$ and $\Omega\subseteq \BBR^n$ denote open sets and $B(x,r)\subseteq\BBR^n$ the open ball centered at $x\in \BBR^n$ with radius $r>0$. The Euclidean distance of $x\in \BBR^n$ to a closed set $E\subseteq \BBR^n$ is denoted by $d(x,E)$ and the distance of two closed sets $E,\, F\subseteq \BBR^n$ by $d(E,F)$. Further, let $M$ denote the number of multi-indices $\alpha \in \BBN^n$ of length $m$, that is $M=\binom{n+m-1}{m}$. We use standard notation for the partial spatial derivatives $\partial^\alpha$.
		
		In this work we deal with systems, but for readability we usually make no notational difference between the cases $N=1$ and $N>1$, meaning we will write 
		\[Lu= (-1)^m \sum_{|\alpha|=|\beta|=m } \partial^\alpha (a_{\alpha,\beta}(t,x) \partial^\beta u)\] and, in shorthand notation, $L=(-1)^m\mbox{div}_m A \nabla ^m$,
		while keeping in mind that the coefficient matrix
		$A(t,x)=(a_{\alpha,\beta}^{i,j}(t,x))_{|\alpha|=|\beta|=m,\,1\leq i,j\leq N}$ is a measurable function on $\BBR^{n+1}_+$ valued in $\BBC^{NM\times NM}$. We refer to the constants $n,\, N,\, m$ simply as \textit{dimensions}.
		
		We use standard notions of $L^p(\Omega)$, $1\leq p\leq\infty$, and Sobolev spaces $W^{m,p}(\Omega)$ of complex-valued functions on $\Omega$. For a Banach space $X(\Omega)$ of complex-valued functions, let $\mathscr L(X(\Omega))$ denote the space of bounded linear operators. 
		Further, we denote by $L^p(a,b;X(\Omega))$ (or $L^p(X)$ for $(a,b)=(0,\infty)$ and $\Omega=\BBR^n$) the Bochner space of $X(\Omega)$-valued $L^p$ functions on $(a,b)$. We write $f\in L^p_{loc}(a,b;X_{loc}(\Omega))$ if $f\in L^p(c,d;X(\omega))$ for any open cylinder $(c,d)\times \omega$ with $a<c<d<b$ and $\overline \omega \subseteq \Omega$. Spaces $\mathscr C(a,b;L^p(\Omega)), \, \mathscr C(a,b;L_{loc}^p(\Omega))$ are defined similarly and if the continuity holds up to the endpoints of the interval, we write $\mathscr C([a,b];L^p(\Omega))$. Additionally, $\mathscr C_0([0,\infty),L^p(\BBR^n))$ consists of those elements of $\mathscr C([0,\infty);L^p(\BBR^n))$, whose $L^p$-norm vanishes as $t$ tends to infinity. 
		
		For an open subset $U$ of $\BBR^n$ or $\BBR^{n+1}_+$, we denote the space of smooth compactly supported functions on $U$ by $\mathscr D(U)$ and by $\mathscr D'(U)$ the space of distributions. The spatial Fourier transform $\mathcal Fu$ defined on the Schwartz space $\mathscr S (\BBR^n)$, will be sometimes denoted by $\hat u$. Throughout the work, $\mathcal M_{HL}$ denotes the (uncentered) Hardy--Littlewood maximal operator.
		
		We let $\mathcal P_{m-1}$ denote the space of polynomials on $\BBR^n$ of degree less than $m$. 
		
		When referring to a solution to \eqref{eq:the parabolic equation} we always mean a weak solution according to the definition below.
		\begin{definition}
			\label{def:weak_sol}
			A (local) weak solution of \eqref{eq:the parabolic equation} on $(a,b)\times \Omega\subseteq \BBR^{n+1}_+$ is a complex-valued function $u\in L^2_{loc}(a,b; H^m_{loc}(\Omega))$ such that 
			\[\int_a^b\int_\Omega u(t,x)\overline{\partial_t \phi (t,x)}dxdt=\sum_{|\alpha|=|\beta|=m} \int_a^b\int_\Omega a_{\alpha,\beta}(t,x) \partial^\beta u(t,x)  \overline{\partial^\alpha \phi(t,x)} dxdt\] holds for all $\phi \in \mathscr C_c^\infty((a,b)\times \Omega)$. If $(a,b)=(0,\infty)$ and $\Omega =\BBR^n$, we call $u$ a global weak solution. 
		\end{definition}
		In this work the constant $C>0$ may vary from line to line. Unless stated otherwise, it depends only on the ellipticity and dimensions. We write $a\lesssim b$ for $a,b\in\BBR$ if there is a constant $C>0$ with $a\leq Cb$. Finally, $a\sim b$ if $a\lesssim b$ and $b\lesssim a$.  
	\subsection{Ellipticity}
		\label{sec:ellipticity}
		The coefficients $(a_{\alpha,\beta}^{i,j}(t,x))_{|\alpha|=|\beta|=m,\,1\leq i,j\leq N}$ of the operator $L=(-1)^m\mbox{div}_m A \nabla ^m$ are always assumed to belong to $L^\infty(\BBR^{n+1}_+;\BBC)$. In particular there exists a $\Lambda >0$ with 
		\begin{equation}
		\label{eq:upper ellipticity estimate}
		\||a_{\alpha,\beta}|\|_{L^\infty(\BBR^{n+1}_+)}\leq \Lambda.
		\end{equation}
		Unless otherwise stated, $L$ is elliptic in the sense of the strong G{\aa}rding inequality, meaning there exists $\lambda >0$ so that
		\begin{equation}
		\label{eq:lower Garding ellipticity estimate}
		\mbox{Re}\sum\nolimits_{\substack{|\alpha|=|\beta|=m \\ 1\leq i, j\leq N}} \int_{\BBR^n}a_{\alpha,\beta}^{i,j}(t,x)\partial^\beta f_j(x)\overline {\partial^\alpha f_i(x)}dx\ge \lambda \|\nabla^m f\|_{L^2(\BBR^{n})}^2
		\end{equation} holds for almost every $t>0$ and for any $f\in \mathscr D'(\BBR^n; \BBC^N)$ with $|\nabla^m f|\in L^2(\BBR^n)$. We will refer to conditions \eqref{eq:upper ellipticity estimate} and \eqref{eq:lower Garding ellipticity estimate} as the \textit{ellipticity estimates} for $L$ and to $\lambda, \,\Lambda$ as the \textit{ellipticity constants}.
		
		As we show in Proposition \ref{prop:comp_tent_kenig_bigp}, any global solution $u$ of \eqref{eq:the parabolic equation} satisfies $\|\nabla^m u \|_{T^{p,2}}\lesssim \|u\|_{X^p_m}$ for $p\ge2$. We prove this estimate also in the case $p\in[1,2)$ under the following ellipticity assumption. We say that $L$ satisfies the strong ellipticity condition if there exists $\lambda >0$ so that for any $\xi\in \BBC^{NM}$ and almost every $(t,x)\in \BBR^{n+1}_+$ it holds that
		\begin{equation}\label{eq:strong lower ellipticity bounds}
		\text{Re} \left(\sum\nolimits_{\substack{|\alpha|=|\beta|=m \\ 1\leq i, j\leq N}} a_{\alpha,\beta}^{i,j}(t,x) \xi_\beta^j\overline{\xi_\alpha^i}\right)\ge \lambda \|\xi\|^2.\end{equation}
		
		It is immediate that the strong ellipticity condition \eqref{eq:strong lower ellipticity bounds} implies the strong G{\aa}rding inequality \eqref{eq:lower Garding ellipticity estimate}.\ The converse is not true except if $m=1$ and $N=1$.\ For constant coefficients \eqref{eq:lower Garding ellipticity estimate} is by a Fourier transform argument equivalent to \eqref{eq:strong lower ellipticity bounds} for the specific choice $\xi_\beta^j\coloneqq\xi^\beta\eta_j$ for $\eta\in \BBC^M$ and $\eta\in \BBC^N$. 
		See \cite[\S0.4]{AT00} and \cite[\S1]{AHMT02} for a discussion of the relation between different notions of ellipticity. 
	
	\subsection{Sobolev and Lions spaces}
		
		We review some of the classical results concerning distributions with integrable derivatives. First, recall that integrability of higher order derivatives implies that the distribution itself is locally integrable (cf.\ \cite[Corollary 2.1]{DL54}). We deduce from \cite[\S1.1.11]{M11} the following version of the Poincar\'e inequality.
		\begin{lemma}
		\label{lem:generalized Poincare}
		Let $p\ge 1$ and $u\in \mathscr D'(B(0,1))$ be a distribution with derivatives of order $m$ in $L^p(B(0,1))$. Let $\omega$ be an open set with $\overline{\omega }\subseteq B(0,1)$. Then, there exists a polynomial $P \in \pol$ 
		\[P(x) = \sum_{|\alpha|\leq m-1} (u,\phi_\alpha)x^\alpha,\] so that 
		\begin{equation}
		\label{eq:generalized Poincare}
		\sum_{k=0}^{ m} \|\nabla^k(u-P)\|_{L^p(B(0,1))} \leq C \|\nabla^mu\|_{L^p(B(0,1))}.
		\end{equation}
		Here, the constant $C$ and the functions $\phi_\alpha\in \mathscr D(\omega)$ do not depend on $u$.
		\end{lemma}
		
		\noindent As a consequence of \eqref{eq:generalized Poincare} we obtain for all $r>0$ 
		\begin{equation*}
		\min_{P\in \mathcal P_{m-1}} \left(\sum_{k=0}^{ m} r^{k-m}\|\nabla^k(u-P)\|_{L^p(B(0,r))} \right)\leq C \|\nabla^mu\|_{L^p(B(0,r))}.
		\end{equation*} 

We also note the Gagliardo--Nirenberg inequality.
		 \begin{lemma}[{\cite[Theorem 1.5.2]{ChM12}}]
		 	\label{lem:Gagliardo--Nirenberg inequality}
		 	Let $m\in \BBN$, $p, r\in [1,\infty] $ and $u\in L^p(\BBR^n)\cap L^r(\BBR^n)$ with the distributional derivatives satisfying $\nabla^m u \in L^p(\BBR^n)$. Then for integer $0\leq k\leq m$, $\theta\in [\frac{k}{m},1]$ (except $\theta=1$ if $m-k-\frac{n}{2}\in \BBN$) and any multi-index $\gamma\in \BBN^n$ with $|\gamma|=k $ it holds $\partial^\gamma u \in L^q(\BBR^n)$, for $q$ given by 
		 	$$\frac{1}{q}=\frac{k}{n}+\theta \left(\frac{1}{p}-\frac{m}{n}\right)+(1-\theta)\frac{1}{r}.$$
		 	\noindent Moreover,  
		 	\begin{equation*} 
		 	\|\partial^\gamma u\|_{L^q}\lesssim_{n,m} \|\nabla^m u\|_{L^p}^{\theta} \|u\|_{L^r}^{1-\theta}.
		 	\end{equation*}	
		 	
		 \end{lemma}
		 \noindent The Gagliardo--Nirenberg inequality will be used frequently throughout this work. To establish (higher) integrability of the intermediate derivatives, we will mostly refer to the case when $\theta=\frac{k}{m}$ and the exponent $q$ is given by $$\frac{1}{q}=\frac{k}{m}\frac{1}{p}+\frac{m-k}{m}\frac{1}{r}.$$

The weak formulation of the parabolic equation \eqref{eq:the parabolic equation} uses the distributional derivative $\partial_t u$ of a weak solution and its higher order spatial derivatives. In particular, integrability of $\nabla^m u$ provides additional information about the distribution $\partial_t u$. From this, we can deduce some regularity of $u$, as the next result states.
			\begin{lemma}[{\cite[Chap.\ XVIII, \S2-3]{DL92}}]
				\label{lem:absolute continuity}
				Let $V$ and $H$ be complex, separable Hilbert spaces and $V'$ be the antidual of $V$. Assume that $V$ is dense in $H$ such that $V\hookrightarrow H\hookrightarrow  V'$ and $H$ is dense in $V'$. Consider the inhomogoneneous space 
				$$W(a,b;V,V') \coloneqq \{u \mid u\in L^2(a,b;V),\, \partial_t u \in L^2(a,b;V')\}.$$
				Then every $h\in W(a,b;V,V')$ is equal almost everywhere to a continuous function of $[a,b]$ to $H$. Moreover, for any $v\in  W(a,b;V,V')$ the function $t\mapsto \langle h(t), v(t)\rangle$ is absolutely continuous over $[a,b]$ with distributional derivative 
				$$\frac{d}{dt}\langle h(t), v(t)\rangle = \langle h'(t), v\rangle_{V',V} + \overline{\langle v'(t), h(t)\rangle_{V',V}}.$$
			\end{lemma}
			
			\noindent We will frequently apply Lemma \ref{lem:absolute continuity} with $V=H^m_0(\Omega)$, $H=L^2(\Omega)$ and $V'=H^{-m}(\Omega)$ for any open $\Omega\subseteq \BBR^n$ and refer to the spaces $W(a,b;V,V')$ from the above definition as (inhomogeneous) Lions spaces.
			 To explain the choice of vocabulary, we note that following \cite{AMP15} we could introduce the (homogeneous) space \[\hlions\coloneqq \{u\in \mathscr D' (\BBR^{n+1}_+)\mid \nabla^m u\in L^2(0,\infty;L^2(\BBR^{n})),\, \partial_t u \in L^2(0,\infty;\hsobdual)\},\] which, however, does not fit into the setting from Lemma \ref{lem:absolute continuity}. Nevertheless, it can be shown (see \cite[Lemma 3.1]{AMP15} for a proof in the second order case) that every $u\in \hlions$ can be uniquely decomposed in $u=v+P$ with $v\in \hlions \cap \mathscr C _0([0,\infty);L^2(\BBR^n))$ and $P\in \mathcal P_{m-1}$. This decomposition shows the existence of the trace of any weak solution $u$ to \eqref{eq:the parabolic equation} with $\nabla^m u \in L^2(L^2)$. Here, we do not address the homogeneous spaces explicitly, as we present an alternative proof of this part of the well-posedness results. The advantage of our method is that it applies to the case $p<2$ with the tent space condition $\nabla^m u \in \tentspace$.
		\subsection{Tent and Kenig-Pipher spaces} 
		\label{sec:tent} 
		In this short section we introduce and remark on spaces, which play a crucial role in the $L^p(\BBR^n)$ well-posedness theory we develop in this work. They are variants of spaces presented in \cite[\S2]{AMP15}, where due to the structure of our problem, the homogenity needed to be changed from $\sqrt{t}$ to $\sqrt[2m]{t}$. For further details we refer to the paper of Coifman, Meyer and Stein \cite{CMS85} (for tent spaces) as well as \cite{AMP15} and references therein.

	The parabolic tent spaces $\tentspace$ adapted to the order of our equation were introduced in Definition \ref{def:tent}. In addition, we remark that those spaces are reflexive if $1<p<\infty$ and in this regime, the duality $(\tentspace)'=T^{p',2}_m$ holds, where $1/p'+1/p=1$ and the duality is given by the $L^2$ inner product on $\BBR^{n+1}_+$. Moreover, we have the following crucial observation.
		\begin{remark}
		\label{rem:eqv norm tent space}
		 By \cite[Theorem 3]{CMS85} an equivalent norm on $\tentspace$ for $2<p<\infty$ is given by 
		 $$\tentspace \ni f \mapsto \bigg\|\sup_{x\in B}\left(\int_0^{r_B^{2m}} \fint _{B}|f(t,y)|^2 dy dt\right)^{1/2} \bigg\|_{L^p}.$$ 
		\end{remark}
		
		We proceed with the analogy of non-tangential maximal functions, which play a crucial role in the Hardy space theory. Here, we use a  parabolic version of the non-tangential maximal function, introduced by Kenig and Pipher in \cite{KP93} in the context of elliptic equations.
		\begin{definition} \label{def:nontan} For $0<p\leq \infty$ let $ X^p_m$ be the space of functions $u\in L^2_{loc}(\BBR^{n+1}_+)$, for which the non-tangential maximal function 
		\[\nontan u(x)\coloneqq \sup_{\delta>0}\left(\fint_{\delta/2}^\delta\fint_{B(x,\sqrt[2m]{\delta})}|u(t,y)|^2dydt\right)^{1/2}\] belongs to $L^p(\BBR^n)$. We write $\|u\|_{X^p_m}\coloneqq\|\nontan u\|_{L^p}<\infty.$
		\end{definition}
		\noindent The space $X^p_m$ is a Banach space for $1\leq p\leq\infty$. 
	
		\subsection{Campanato type spaces $BMO_m(\BBR^n)$ and $L^p_m(\BBR^n)$}
		\label{sec:Campanato}
		\noindent We first recall the space of functions bounded mean oscillations of John and Nirenberg \cite{JN61}. We say that $f\in L^1_{loc}(\BBR^n)$ belongs to $BMO(\BBR^n)$, if the \textit{sharp function} $$f^\#(x)  =\sup_{B\ni x}  \fint_{B}\left|f(y)-f_{B}\right| dy $$  belongs to $L^{\infty}(\BBR^n)$,	where $f_B\coloneqq \fint_B f dx$. Examples of $BMO(\BBR^n)$ functions are for example constants and $\log |p|$ for any polynomial $p $, we refer to the books \cite{G09, S93} for details. By the John--Nirenberg Lemma a seminorm on $BMO(\BBR^n)$ is given by
		$$f\mapsto\|f\|_{BMO}\coloneqq\sup_{x\in \BBR^n} \sup_{r>0}\inf_{c\in \BBC}\left(\fint_{B(x,r)}\left|f(y)-c\right|^2 dy\right)^{1/2}.$$ We recall the well-known  
		 Carleson measure characterization of $BMO(\BBR^n)$ by C. Fefferman and Stein.
		 \begin{proposition}[{\cite[Chap.\ VI, \S4.3]{S93}}]
		 	\label{prop:carleson}
		 	A Borel measure $\mu$ on $\mathcal{B}(\BBR^{n+1}_+)$ satisfying
		 	\[\|\mu\|_{\mathcal C}\coloneqq\sup_{x\in \BBR^n} \sup_{r>0} \frac{1}{|B(x,r)|}\int_0^{r}\int_{B(x,r)} d\mu<\infty\] is called a Carleson measure. For a locally integrable function $f$ the following holds
		 	\begin{enumerate}[label=\upshape{(\roman*)}]
		 		\item Assume $f\in BMO(\BBR^n)$, then $ d\mu(x,t)=|t\nabla v(t^2,x)|^2\frac{dxdt}{t}$ is a Carleson measure, where $v(t,x)=e^{t\Delta}f(x)$ is the solution to the heat equation with initial datum $f$. Moreover,
		 		\begin{equation*}
		 		\|\mu\|_{\mathcal C}=
		 		\sup_{x\in \BBR^n} \sup_{r>0}\frac{1}{|B(x,r)|}\int_0^{r^2}\int_{B(x,r)}|\nabla v(t,y)|^2dydt\lesssim_{n}\|f\|_{BMO}^2.
		 		\end{equation*}
		 		\item Assume that $f$ satisfies the growth condition \begin{equation}
		 		\label{eq:growth_condition}
		 		\int_{\BBR^n}\frac{|f(x)|}{(1+|x|)^{n+1}}dx<\infty.
		 		\end{equation} If for $v$ as above $d\mu(x,t)=|t\nabla v(t^2,x)|^2\frac{dxdt}{t}$ is a Carleson measure, then $f\in BMO(\BBR^n)$ and $\|f\|_{BMO} \lesssim_n \|\mu\|_{\mathcal C}^{1/2}. $
		 	\end{enumerate}
		 \end{proposition}
		 
		 Proposition \ref{prop:carleson} allows to compare the $T^{\infty,2}$ norm of the gradient of the solution to the Cauchy problem $\partial_t v = \Delta v$, $v_0=f$, given by the propagator $v(t,x)=e^{t\Delta}f(x)=\Gamma(t,0)f(x)$, with the $BMO(\BBR^n)$ norm of initial data. We will derive a generalization of this result, where the heat semigroup is replaced by the family of propagators associated to our equation \eqref{eq:the parabolic equation}. We also show that the growth condition \eqref{eq:growth_condition} is not necessary.
		 
		 Since polynomials $P\in \mathcal P_{m-1}$ are trivially solutions of \eqref{eq:the parabolic equation} and satisfy $\|\nabla^m P \|_{\tentspace} =0$, we adjust the $BMO$-type space as follows. 
	
		 \begin{definition}
		 	\label{def:BMO higher order}
		 	We say that a function $f\in L^2_{loc} (\BBR^n)$ belongs to the class $BMO_{m}(\BBR^n)$ if 
		 	\[ \|f\|_{BMO_m}\coloneqq\sup_{x\in \BBR^n} \sup_{r>0} \inf_{P\in \mathcal P_{m-1}}\left( \fint_{B(x,r)}\left|f(y)-P(y)\right|^2dy\right)^{1/2}<\infty.
		 	\]
		 \end{definition}
		 
		 The $BMO_m(\BBR^n)$ space was first introduced by Campanato in \cite{C64}, who investigated the regularity of such functions in dependence of the power of the factor $r^{-1}$ in front of the integral. His estimates in \cite{C64} rapidly lead to the conclusion that $BMO_m(\BBR^n)$ equals the $BMO(\BBR^n)$ space modulo $\pol$. The complete result appears in \cite[Theorem 1]{JTW06}.
		 \begin{proposition}[{\cite[Theorem 1]{JTW06}}]
		 			\label{prop:structure of BMO_m}
		 			Let $f\in BMO_m(\BBR^n)$. Then there exists a unique polynomial $P\in \pol$ satisfying $P(0)=0$, so that $f-P\in BMO(\BBR^n)$. Moreover, there is a constant $C=C_{n,m}>0$ such that 
		 			\[\|f\|_{BMO_m}\leq \|f-P\|_{BMO}\leq C \|f\|_{BMO_m}.\]
		\end{proposition}

		 Analogously, we introduce the polynomial $L^p$ spaces for $p>2$. 
		 
		 \begin{definition}
		 	\label{def:L^p higher order}
		 	Let $p>2$. We say that a function $f\in L^2_{loc} (\BBR^n)$ belongs to $L^p_{m}(\BBR^n)$ if 
		 	\[ \|f\|_{L^p_m}\coloneqq\bigg\|\sup_{B\ni x} \inf_{P\in \mathcal P_{m-1}}\left( \fint_{B}\left|f(y)-P(y)\right|^2dy\right)^{1/2}\bigg\|_{L^p}<\infty.
		 	\]
		 \end{definition}
		\begin{proposition}
		 			\label{prop:structure of L^p_m}
		 			Let $p\in (2,\infty)$ and $f\in L^p_m(\BBR^n)$. There exists a unique polynomial $P\in \pol$, so that $f-P\in L^p(\BBR^n)$. Moreover, there is a constant $C=C_{n,m,p}>0$ such that 
		 			\[ C^{-1}\|f\|_{L^p_m}\leq \|f-P\|_{L^p} \leq C \|f\|_{L^p_m}.\]
		\end{proposition}
		Before proving Proposition \ref{prop:structure of L^p_m}, let us collect some remarks, which we use without comment later on.
		\begin{enumerate}[label={\upshape(\roman*)}]
			\item 
			For $p\in (2,\infty)$ and $f\in L^2_{loc} (\BBR^n)$, the relations $f\in L^p_{m}(\BBR^n)$ and $f\in BMO_{m}(\BBR^n)$ can be equivalently expressed in terms of the polynomial sharp function $f^{\#,m}$ by requiring $f^{\#,m} \in L^p(\BBR^n)$ or $f^{\#,m} \in L^\infty(\BBR^n)$, correspondingly. Here,
			\[f^{\#,m} (x)\coloneqq \sup_{B\ni x} \inf_{P\in \mathcal P_{m-1}}\left( \fint_{B}\left|f(y)-P(y)\right|^2dy\right)^{1/2}.\]
			\item For each ball $B=B(x_0,r)$ and $f \in L^2_{loc} (\BBR^n)$ the infimum
			\[ \inf_{P\in \mathcal P_{m-1}}\left( \fint_{B}\left|f(y)-P(y)\right|^2dy\right)^{1/2}\] is attained for the polynomial given by the orthogonal projection $\mathbb{P}_{x_0,r}(f)$ of $f$ onto $\mathcal P_{m-1}$ with respect to the scalar product  $(f,g)\mapsto \fint_{B(x_0,r)}f\overline g  dx$. 
			
			\item The minimizing polynomial $\mathbb{P}_{x_0,r}(f)$ for $f\in L^2_{loc}(\BBR^n)$ on $B(x_0,r)$ satisfies 
			$$\|\mathbb{P}_{x_0,r}(f)\|_{L^\infty(B(x_0,r))}\lesssim  \fint_{B(x_0,r)}|f(x)|dx$$	with some constant $C>0$ depending on $m$ and $n$.
		\end{enumerate}
		Above points remain true if we introduce a weight $\omega$, that is a non-negative, bounded, radially symmetric weight function on $B(0,1)$ satisfying $\int \omega =1$ and $0<c<\omega $ on $B(0,1/2)$, and rescale it appropriately on each ball $B$.
	For the proof of Proposition \ref{prop:structure of L^p_m}, we need the following result. The case $p=\infty$ implies Proposition \ref{prop:structure of BMO_m}.
		 \begin{lemma}
		 	\label{lem:filter out the pol}
		 	Let $p\in (2,\infty]$ and $f\in L^2_{loc} (\BBR^n)$. Then there exists $P\in \mathcal P_{m-1}$ such that
		 	\begin{equation}\label{eq:sharp function norm}
		 	\| (f-P)^{\#,1}\|_{L^p} =\bigg\|\sup_{B\ni x} \inf_{c\in \BBC}\left( \fint_{B}\left|(f-P)(y)-c\right|^2dy\right)^{1/2}\bigg\|_{L^p}\lesssim \| f^{\#,m}\|_{L^p}.
		 	\end{equation}
		 \end{lemma}
		 \begin{proof} 
		 	Clearly, it is enough to prove 
		 	\begin{align*}
		 	(f-P)^{\#,1}(x) \lesssim  	f^{\#,m}(x)<\infty
		 	\end{align*} for some polynomial $P\in\pol$ and almost every $x\in \BBR^n$. This requires comparing the coefficients of the minimizing polynomials on different balls and is an easy consequence of the proof of \cite[Theorem 1]{JTW06}. We only sketch the main steps. First, the assumption gives us \[C_{x_0,r_0}(f)=\sup_{r\ge r_0} \inf_{P\in \mathcal P_{m-1}}\left( \fint_{B(x_0,r)}\left|f(y)-P(y)\right|^2dy\right)^{1/2}<\infty\]
		 	for all $x_0\in \BBR^n$ and $r_0>0$.
		 	We write $\mathbb{P}_{x_0,r}(f)$ as
		 	\begin{align*}
		 	\mathbb{P}_{x_0,r}(f)(y)=\sum_{|\alpha|\leq m-1}c_{\alpha}^{x_0,r}(y-x_0)^\alpha
		 	\end{align*} for some coefficients $\left\{ c_{\alpha}^{x_0,r}\mid|\alpha|\leq m-1\right\}\subseteq \BBC$. Using the bound on the minimal polynomials we easily arrive then at the crucial estimate 
		 	 \begin{align*}
		 	 \sum_{|\alpha|\leq m-1}\left|c_{\alpha}^{x_0,r}-c_{\alpha}^{x_0,2r}\right|r^{|\alpha|}
		 	 &\lesssim C_{x_0,r}(f).	
		 	 \end{align*}
		 	 Thus, for $\alpha\neq 0$, we obtain a Cauchy sequence and find the limiting coefficients for $x_0$ and $r$. Repeating of the arguments shows the independence of the coefficients of the radius, while the $x_0$ dependence is accumulated in the constant, if we center the limiting polynomial at the origin.
		 \end{proof}
		\begin{proof}[Proof of Proposition \ref{prop:structure of L^p_m}]
			Let $p\in (2,\infty)$ and $f\in L^p_m(\BBR^n)$. We know 
			$ (f-P)^{\#,1}\in {L^p(\BBR^n)}$ for some polynomial $P\in \pol$ given by Lemma \ref{lem:filter out the pol}. Consider 
			\[C_{x_0,r}\coloneqq\sup_{R\ge r} \left( \fint_{B(x_0,R)}\left|(f-P)(y)-(f-P)_{B(x_0,R)}\right|^2dy\right)^{1/2}.\] 
			Then 
			\begin{equation} \label{eq:decay}C_{x_0,r}\lesssim r^{-\frac{n}{p}}\| (f-P)^{\#,1}\|_{L^p} .\end{equation}
			Indeed, assume $x\in B(x_0,r)$, then $C_{x_0,r} \leq (f-P)^{\#,1}(x)$ and integrating over $B(x_0,r)$ gives 
			\[|B(x_0,r)|^{\frac{1}{p}}C_{x_0,r}\leq \| (f-P)^{\#,1}\|_{L^p}.\]
			
			Comparing the means $(f-P)_{B(x_0,r)}$ and $(f-P)_{B(x_0,2r)}$ and the decay from \eqref{eq:decay} imply the convergence of $(f-P)_{B(x_0,r)}$ as $r$ tends to infinity. Again, the limit does not depend on $x_0$ and denote it by $d_0$. Define $g\coloneqq f-P-d_0$. Then it holds for every $r>0$
			\begin{equation}
			\label{eq:estimate 3}
			\sup_{x_0\in \BBR^n}\left(\fint_{B(x_0,r)}|g|^2dx\right)^{1/2}\lesssim r^{-\frac{n}{p}}\| (f-P)^{\#,1}\|_{L^p}\lesssim r^{-\frac{n}{p}}\| f^{\#,m}\|_{L^p}.
			\end{equation}
			
			On the other hand, we have $g^{\#}(x)\leq  g^{\#,1}(x)$ for every $x\in \BBR^n$, thus
			\begin{equation}
			\label{eq:g sharp}
			\| g^{\#}\|_{L^p}\lesssim \| f^{\#,m}\|_{L^p}<\infty.
			\end{equation}
			
			We now demonstrate that estimates \eqref{eq:estimate 3} and \eqref{eq:g sharp} imply $g\in L^p(\BBR^n)$. Our proof is based on the following consequence of \cite[Chap.\ IV, \S2]{S93}.
			\begin{lemma}[{\cite[Chap.\ IV, \S2]{S93}}]
				\label{lem:stein hardy}
				Let $g$ be a bounded tempered distribution on $\BBR^n$ and suppose that $h\in \mathcal H^1(\BBR^n)$. Then
				\[\bigg|\int_{\BBR^n} g(x) h(x) dx \bigg| \lesssim \int_{\BBR^n} g^\#(x)\mathcal M_{HL}h(x)\,dx.\]  
			\end{lemma}
			
			 Note that in \cite{S93} this lemma was stated with the Hardy--Littlewood maximal function replaced by the grand maximal operator $\mathcal{M_F}$ for any fixed finite collection $\mathcal F$ of seminorms on the Schwartz space (for definition see \cite[Chap.\ III, \S1.2]{S93}). We can, however, estimate $\mathcal{M_F}h\leq C \mathcal M_{HL}h$, see \cite[Chap.\ II, \S2.1]{S93}. We also recall that any bounded, compactly supported function $h$ with $\int h =0$ belongs to $\mathcal H^1(\BBR^n)$ (see \cite[Chap.\ III, \S5, Remark 5]{S93}).
	
			\begin{claim} Let $h\in \mathscr C_c^\infty(\BBR^n)$. Then it holds for $q\in(1,2)$ with $\frac{1}{p}+\frac{1}{q}=1$
				\[\bigg|\int_{\BBR^n} g(x) h(x) dx \bigg| \lesssim \| f^{\#,m}\|_{L^p}\|h\|_{L^q}.\]
			\end{claim}
			\begin{claimproof}
			 Consider a sequence of molifiers $(\eta_\varepsilon)_{\varepsilon>0}$, where $\eta_\varepsilon=\varepsilon^{-n}\eta(\cdot/\varepsilon)$ for some non-negative $\eta \in \mathscr C_c^\infty(\BBR^n)$ supported in $B(0,1)$ and satisfying $\int \eta =1$. Set $g_\varepsilon \coloneqq g\ast \eta_\varepsilon$ for $\varepsilon>0$. Then $ g_\varepsilon$ defines a bounded tempered distribution, as we can show $ g_\varepsilon \in L^\infty(\BBR^n)$. 
			 
			 Indeed, for any $x\in \BBR^n$ we have the uniform estimate
			 \[| g_\varepsilon(x)|\lesssim_\eta \frac{1}{\varepsilon^n}\bigg|\int_{B(x,\varepsilon
			 	)}g(y)dy\bigg|\lesssim\left(\fint_{B(x,\varepsilon)}|g(y)|^2dy\right)^{1/2}\stackrel{\eqref{eq:estimate 3}}{\lesssim} \varepsilon^{-\frac{n}{p}}\| f^{\#,m}\|_{L^p}.\]
			 
			 Moreover, for $h\in \mathscr C_c^\infty(\BBR^n)$ choose some $R>0$ with $\text{supp}\, h \subseteq B(0,R)$ and write $$ h_R = h-h_{B(0,R)}\mathbbm{1}_{B(0,R)}.$$ Then $ h_R \in \mathcal H^1(\BBR^n)$ and so by Lemma \ref{lem:stein hardy}
			   \begin{equation}\label{eq:stein conv}\bigg|\int_{\BBR^n} g_\varepsilon(x)h_R(x)dx\bigg| \lesssim \|  g_\varepsilon^{\#}\|_{L^p}\|\mathcal M_{HL} h_R\|_{L^q}.\end{equation}
			  
			  By the boundedness of the Hardy--Littlewood maximal function, there exists $C_q>0$ so that
			   \begin{equation}\label{eq:stein conv2}\|\mathcal M_{HL}h_R\|_{L^q}\leq C_q\| h_R\|_{L^q}\leq C_q\left(\| h\|_{L^q}+|B(0,R)|^{\frac{1}{q}}|h_{B(0,R)}|\right) \lesssim_{q,n}\| h\|_{L^q}\end{equation} holds independently of $R>0$.
			   Furthermore, we have for all $x\in \BBR^n$ and $\varepsilon>0$
			   \begin{equation}
			   \label{eq:convolved sharp function}
			   g_\varepsilon^{\#}(x)=\left(\eta_\varepsilon \ast g\right)^\# (x)\lesssim \eta_\varepsilon \ast g^\# (x)
			   \end{equation} 
			   and, consequently, $g_\varepsilon^{\#}\in L^p(\BBR^n)$ with the uniform bound 
			    \begin{equation}
			    \label{eq:convolved sharp function p}\sup_{\varepsilon>0} \|g_\varepsilon^{\#}\|_{L^p}\lesssim \sup_{\varepsilon>0} \|\eta_\varepsilon\ast g^{\#}\|_{L^p}\lesssim \| g^{\#}\|_{L^p}.
			    \end{equation} 
			   
			   We now prove \eqref{eq:convolved sharp function}. By Fubini's Theorem 
			    \begin{align*}
			    \left(\eta_\varepsilon \ast g\right)^\# (x)
			    &
			    \lesssim\sup_{B(x_0,\rho)\ni x} \fint_{B(x_0,\rho)}\fint_{B(0,\varepsilon)}\eta\left(\frac{y}{\varepsilon}\right)\bigg| g(z-y)
			    -\fint_{B(x_0,\rho)}g(z-y)dz\bigg|dy\,dz.
			    \end{align*} 
			    
			    We use Fubini's Theorem, then translate variables and take the supremum inside the first integral to arrive at 
			      \begin{align*}
			      \left(\eta_\varepsilon \ast g\right)^\# (x)
			      &\leq\sup_{B(x_0,\rho)\ni x} \fint_{B(0,\varepsilon)}\eta\left(\frac{y}{\varepsilon}\right)\sup_{B(\tilde x_0, \tilde \rho)\ni x-y}\fint_{B(\tilde x_0,\tilde\rho)}\bigg| g(z)
			      -\fint_{B(\tilde x_0,\tilde\rho)}g(z)dz\bigg|dz\,dy\\
			      &\lesssim \frac{1}{\varepsilon^n}\int_{B(0,\varepsilon)}\eta\left(\frac{y}{\varepsilon}\right) g^\# (x-y) dy =   \eta_\varepsilon \ast g^\# (x).
			      \end{align*} Putting \eqref{eq:stein conv}, \eqref{eq:stein conv2} and \eqref{eq:convolved sharp function p} together we have for all $R>0$ and $\varepsilon>0$
			       \begin{equation}\label{eq:stein conv3}\bigg|\int_{\BBR^n} g_\varepsilon(x)h_R(x)dx\bigg| \lesssim \|  g^{\#}\|_{L^p}\| h\|_{L^q}.\end{equation}
			       
			       On the other hand, since $g\in L^2_{loc}(\BBR^n)$ we have
			    \[\int_{\BBR^n}g(x)h(x)dx = \lim_{\varepsilon\to 0}\int_{\BBR^n}g_\varepsilon(x)h(x)dx.\] 
			    
			    Let $\varepsilon_0>0$ be small enough so that $\text{supp}\, h  \subseteq B(0,\varepsilon_0^{-1})$. For $\varepsilon< \varepsilon_0$, we consider the decomposition $h= h_{1/\varepsilon}+h_{B(0,1/\varepsilon)}\mathbbm{1}_{B(0,1/\varepsilon)}$ and estimate using Fubini's Theorem
			    \begin{align*}
			    \bigg|\int_{B(0,1/\varepsilon)} g_\varepsilon(x) h_{B(0,1/\varepsilon)} dx \bigg|
			    &\lesssim \varepsilon^{\frac{n}{q}}\|h\|_{L^q}\fint_{B(0,\varepsilon)} \eta\left(\frac{y}{\varepsilon}\right)\int_{B(0,1/\varepsilon)}\big| g(x-y)\big|dxdy\\
			    &= \varepsilon^{\frac{n}{q}-n}\|h\|_{L^q}\int_{B(0,\varepsilon)} \eta_\varepsilon\fint_{B(-y,1/\varepsilon)}\big| g(z)\big|dzdy\\
			   &\hspace{-0.2cm} \stackrel{\eqref{eq:estimate 3}}{\lesssim} \varepsilon^{\frac{n}{q}-n}\varepsilon^{\frac{n}{p}}\|h\|_{L^q}\| f^{\#,m}\|_{L^p}\int \eta_\varepsilon(y)dy
			   \\&=\| f^{\#,m}\|_{L^p}\|h\|_{L^q}.
			    \end{align*}
			    Concluding,
			      	    \begin{align*}\int_{\BBR^n}g(x)h(x)dx &= \lim_{\varepsilon\to 0}\int_{\BBR^n}g_\varepsilon(x)h(x)dx\\&= \limsup_{\varepsilon\to 0}\left(\int_{\BBR^n}g_\varepsilon h_{1/\varepsilon} dx + \int_{\BBR^n}g_\varepsilon h_{B(0,1/\varepsilon)}\mathbbm{1}_{B(0,1/\varepsilon)}dx\right)\\&\lesssim\| f^{\#,m}\|_{L^p}\|h\|_{L^q}.\end{align*}
			    This proves the claim.
			\end{claimproof}
			
			Finally, we argue by density of the test functions in $L^p(\BBR^n)$ and conclude by the Claim that $g=f-P\in L^p(\BBR^n)$ and there exists a constant dependent on $m,\, n$ and $p$ such that 
			 \[\|f-P\|_{L^p}\leq C\| f^{\#,m}\|_{L^p}.\]
			
			Uniqueness of $P$ follows from the fact that the only $p$-integrable polynomial is the zero polynomial. Notice that $f^{\#,m}(x)=(f-P)^{\#,m}(x)$ and 
			$$(f-P)^{\#,m}(x)\leq (\mathcal M_{HL}|f-P|^2)^{1/2}(x).$$ 
			
			As $p>2$, the maximal function $\mathcal M_{HL} $ is bounded on $L^{p/2}(\BBR^n)$, hence \[\| f^{\#,m}\|_{L^p}\leq C\|f-P\|_{L^p}.\qedhere\]
		\end{proof}
	
			\subsection{Semigroup theory for autonomous operators}
		\label{sec:aut}
		Undoubtedly, semigroups play a special role in the context of autonomous parabolic systems,
		
		so let us review some of their essential properties.
		There is an extensive existing literature on this subject. What follows can be found for example in \cite[\S0.1, \S0.4]{AT00} (sectorial operators and their functional calculus) and \cite{D95, A07}.
		
		Let $L= (-1)^m\mbox{div}_mA\nabla^m$ be a time-independent operator with bounded measurable coefficients satisfying the strong G{\aa}rding inequality \eqref{eq:lower Garding ellipticity estimate}. The ellipticity assumption implies that $L$ is maximal accretive on $L^2(\BBR^n)$, thus $-L$ generates a $\mathscr C_0$-semigroup of contractions 
		$(e^{-tL})_{t\ge 0}$ on $L^2(\BBR^n)$.
		
		The semigroup $(e^{-tL})_{t\ge 0}$ satisfies Davies--Gaffney $L^2$ off-diagonal estimates, that is there exist constants $C,\;c>0$ such that for all $E,\;F \subseteq \BBR^n$ closed and disjoint, $t>0$ and $f\in L^2(\BBR^n)$ 
		\[\|\mathbbm{1}_E e^{-tL}(\mathbbm{1}_F f)\|_{L^2}\leq C e^{-c\left(\frac{d(E,F)}{t^{1/(2m)}}\right)^{\frac{2m}{2m-1}}}\|\mathbbm{1}_F f\|_{L^2}.\] Estimates of this form occur to be crucial for the results in the aforementioned literature and so will they be in this work. We shall follow the method by Davies \cite{D95} to prove the off-diagonal estimates for our propagators. 
	
		By the solution of Kato's square root conjecture in \cite{AHMT02}, the domain of the operator $L^{1/2}$ is given by the inhomogeneous Sobolev space $\sob$ and we have the uniform bound 
		\begin{equation}\label{eq:Kato} \sup_{t>0}\|\nabla^m e^{-tL}u\|_{L^2}\lesssim \sup_{t>0}\|L^{1/2} e^{-tL}u\|_{L^2}\lesssim \|L^{1/2}u\|_{L^2}\lesssim \|\nabla ^m u\|_{L^2}. \end{equation}
		
By \cite[Remark 3.2]{AHMT02}, for every polynomial $P\in \pol$ the equality $e^{-tL}P=P$ holds in $L^2_{loc}(\BBR^n)$.
		
		We summarize some well-established facts about the $L^p$ theory for the semigroup $e^{-tL}$ (cf.\ \cite{A07}).
		Let $p_-(L)$ and $p_+(L)$ be the infimum and correspondingly supremum over all $p\in [1,\infty]$ such that $\sup_{t>0} \|e^{-tL}\|_{\mathscr L(L^p)}<\infty$. Further, let $q_-(L)$ and $q_+(L)$ be the infimum and correspondingly supremum over all $q\in [1,\infty]$ such that $\sup_{t>0} \|\sqrt t \nabla^m e^{-tL}\|_{\mathscr L(L^q)}<\infty$. Then it holds
		\begin{enumerate}[label=\upshape{(\roman*)}]
			\item $(p_-(L),p_+(L))=(1,\infty)$ if $n\leq 2m$ (in this case the semigroup is given by a kernel with Gaussian bounds),
			\item $[\frac{2n}{n+2m},\frac{2n}{n-2m}]\subseteq(p_-(L),p_+(L))$ if $n> 2m$,
			\item $q_-(L)=p_-(L)$ ,
			\item $q_+(L)=\infty$ if $n=1$, 
			\item $q_+(L)>2 $ and $p_+(L)\ge q_+(L)^{*m}$, where $q\mapsto q^*$ is the Sobolev exponent mapping given by $q^*=\frac{nq}{n-q}$ if $q<n$ and $q^*=\infty$ otherwise. Thus, if $n>mq$ we have $p_+(L)\ge 	\frac{nq}{n-mq}$.	
		\end{enumerate}
		For $p\in (q_-(L),q_+(L))$ the Riesz transforms $\nabla^m L^{-1/2}$ are bounded on $L^p(\BBR^n)$. Further, for exponents $p_-(L)<p\leq q<q_+(L)$ and $k=0,\dots, m$, we have the $L^p-L^q$ off-diagonal estimates of the form 
		\[\|\mathbbm{1}_E   t^{\frac{k}{2m}}\nabla^ke^{-tL}(\mathbbm{1}_F f)\|_{L^q}\leq C t^{\frac{n}{2m}(\frac{1}{q}-\frac{1}{p})}e^{-c\left(\frac{d(E,F)}{t^{1/(2m)}}\right)^{\frac{2m}{2m-1}}}\|\mathbbm{1}_F f\|_{L^p}.\] 
		We will return to the $L^p$ theory in Section \ref{sec:bound examples} and the Appendix.
	
	\section{A priori energy estimates}
	\label{sec:energy estimates}
	\setcounter{theorem}{0} \setcounter{equation}{0}

	This section follows \cite[\S3.2, \S4.1]{AMP15}.
	We begin with local energy estimates for weak solutions of \eqref{eq:the parabolic equation}. In the parabolic setting such estimates are known to be a crucial tool in proving the existence of weak solutions. A similar result (the autonomous case) can be found in \cite[Proposition 3.2]{CMY16}. Our proof relies on the iteration scheme by Barton \cite[Theorem 3.10]{B16}.
	\begin{proposition}
		 \label{pro:local_energy}
		Let $0\leq a<c<d< b\leq\infty$, $\delta>1$, $R>0$ and $x_0\in\BBR^n$. Suppose that $u$ is a local weak solution of \eqref{eq:the parabolic equation} in $(a,b)\times B(x_0,\delta R)$. Then it holds $$u\in \mathscr C ([c,d];L^2(B(x_0,r)))$$ for any $0<r<\delta R$ and there exists a constant $C>0$ depending on ellipticity and dimensions, such that for all $0<r<R$, integer $0\leq k\leq m$ and $a<\tilde a<c$, we have, with $B_r\coloneqq B(x_0,r),$
		\small\begin{equation*}
		  \|u(d,\cdot)\|_{L^2(B_{r})}^2 \leq C\left(\frac{1}{(R-r)^{2m}} + \frac{1}{d-c}\right)\int_c^d \|u(s,\cdot)\|_{L^2(B_R)}^2ds,
		\end{equation*}
		\small
		\begin{align*}
		 (c-\tilde a)\int_{c}^d \|\nabla ^k u(s,\cdot)\|_{L^2(B_r)}^2ds
		 \leq C\left(\frac{(d-c)}{(R-r)^{2k}} +(R-r)^{2m-2k}\right)\int_{c}^d\| u(s,\cdot)\|_{L^2(B_R)}^2 ds.
		\end{align*} 
	\end{proposition}
	\begin{proof}
		We first test the equation with a cut-off function to find for $0<r<R$ and $a'\in (a,d)$
			\begin{align}
			\label{claim1}
			&\|u(d,\cdot)\|_{L^2(B_r)}^2+\lambda\int_{a'}^d \|\nabla ^m u(s,\cdot)\|_{L^2(B_r)}^2ds \\ \notag&
			\leq \|u(a',\cdot)\|_{L^2(B_R)}^2+ \sum_{i=1}^{m-1}\frac{C}{(R-r)^{2m-2k}}\int_{a'}^d \|\nabla ^k u(s,\cdot)\|_{L^2(B_R\setminus B_r)}^2ds.
			\end{align}
	\paragraph{Proof of \eqref{claim1}} Let $\eta \in \testfunctions$ be a real-valued, non-negative function supported in $B_R$, equal $1$ on $B_r$ with $\|\nabla^k\eta\|_{L^\infty}\leq \frac{C}{(R-r)^k}$, for $k=0,\dots, m$. Then for any integer $0\leq k\leq m$
	\[\|\nabla^k (\eta^{2m} u)\|_{L^2((a',d)\times \BBR^n)}\lesssim_{n,m}\sum_{0\leq l \leq k} \left(\frac{C}{R-r}\right)^{k-l}\|\nabla ^{l} u\|_{L^2((a',d)\times B_R)}<\infty, \] 
	so $\eta ^{2m}u \in L^2(a',d; H^m_0(B_R))$. As $u$ is a local weak solution on $(a,b)\times B(x,\delta R)$, we can use the density of test functions on $(a',d)\times B_R$ in $L^2(a',d;H^m_0(B_R))$ to conclude $\partial_t u \in L^2(a',d;H^{-m}(B_R))$. Thus, $\partial_t (\eta^{2m} u) \in  L^2((a',d);H^{-m}(B_R))$ and Lemma \ref{lem:absolute continuity} implies $\eta^{2m} u \in \mathscr C( [a',d]; L^2(B_R))$. Moreover, the map $c\to\|\eta ^{2m}u(c, \cdot)\|_{L^2}^2$ is absolutely continuous in $c\in [a',d]$ and we calculate 
	\begin{align*}
	I\coloneqq\|\eta^{2m} u(d, \cdot)\|_{L^2}^2-\|\eta^{2m} u(a', \cdot)\|_{L^2}^2
	&= - 2\text{Re}\int_{a'}^d \langle A(t,\cdot)\nabla^m  u(t,\cdot), \nabla^m\left(\eta^{4m} u(t,\cdot)\right)\rangle dt 
	\end{align*}
	where the inner pairing is the usual $L^2$ scalar product. By Leibnitz rule this equals further
		\begin{align*}
		I&=- 2\text{Re}\int_{a'}^d \sum_{|\alpha|=|\beta|=m} \left[\int_{B_R}a_{\alpha,\beta}\partial^\beta u \left( \eta^{2m}\overline{\partial^\alpha(\eta^{2m}u)}+\sum_{\gamma<\alpha}c_{\alpha,\gamma}\partial^{\alpha-\gamma}(\eta^{2m})\overline{\partial^{\gamma}(\eta^{2m}u)}\right)dx\right] dt 
		\end{align*}
		
		The second term in the round brackets can be written as $\sum_{\gamma<\alpha}\eta^{2m}\Phi_{\alpha,\gamma}\overline{\partial^{\gamma}u}$ where $\Phi_{\alpha,\gamma}$ are functions supported in $B_R\setminus B_r$ and $\|\Phi_{\alpha,\gamma}\|_{L^\infty}\leq C (R-r)^{-(|\alpha|-|\gamma|)}$, which leads to
		\begin{align*}
			I&=- 2\text{Re}\int_{a'}^d \sum_{|\alpha|=|\beta|=m} \left[\int_{B_R}a_{\alpha,\beta}(t,x)(\eta^{2m}\partial^\beta u) \left( \overline{\partial^\alpha(\eta^{2m}u)}+\sum_{\gamma<\alpha}\Phi_{\alpha,\gamma}\overline{\partial^{\gamma}u}\right)dx\right] dt 
		\end{align*} 
		
		Next, rewrite $\eta^{2m}\partial^\beta u$, so that the term $\partial^\beta(\eta^{2m}u)$ appears. Ellipticity bounds imply then
		\begin{align*}
		&\|\eta^{2m} u(d, \cdot)\|_{L^2}^2-\|\eta^{2m} u(a', \cdot)\|_{L^2}^2+2\lambda \int_{a'}^d \|\nabla ^m 
		(\eta^{2m}u)(s,\cdot)\|_{L^2(B_R)}^2ds \\
		&\lesssim_{\Lambda,n,m}  \int_{a'}^d \|\nabla ^m 
		 (\eta^{2m}u)(s,\cdot)\|_{L^2(B_R\setminus B_r)}\left(\sum_{k=0}^{m-1}(R-r)^{-(m-k)}\|\nabla ^k u(s,\cdot)\|_{L^2(B_R\setminus B_r)}\right)ds \\&+  \int_{a'}^b\sum_{k=0}^{m-1}(R-r)^{-(2m-2k)}\|\nabla ^k u(s,\cdot)\|_{L^2(B_R\setminus B_r)}^2ds 
		 \end{align*} Finally, we use the Cauchy inequality on the first summand, choosing the overall constant in front of the factor $\int_{a'}^d \|\nabla ^m 
		 (\eta^{2m}u)(s,\cdot)\|_{L^2(B_R)}^2ds$ on the right hand side to be smaller than $\lambda$. Subtracting this quantity on both sides gives the claim.
		 \hfill $\blacksquare$
	 
		We remark that the proof shows 
		 \begin{align}\label{claim1_strong}
		 &\|u(d,\cdot)\|_{L^2(B_\rho)}^2+\lambda\int_{a'}^d \|\nabla ^m u(s,\cdot)\|_{L^2(B_\rho)}^2ds \\\notag & \lesssim\|u(d,\cdot)\psi(\cdot)\|_{L^2(B_r)}^2+\lambda\int_{a'}^d \|\nabla ^m( u(s,\cdot)\psi(\cdot))\|_{L^2(B_r)}^2ds \\\notag &\leq \int_{a'}^d\sum_{k=0}^{m-1}\frac{C_{\Lambda,\lambda, n,m}}{(r-\rho)^{2m-2k}} \|\nabla ^k u(s,\cdot)\|_{L^2(B_r\setminus B_\rho)}^2ds +\|u(a',\cdot)\|_{L^2(B_r)}^2
		 \end{align} for any $0<\rho<r<R$ and cut-off function $\psi$ with radii-dependent decay as specified above. We still need to deal with the intermediate derivatives. This type of inequality was encountered in \cite[Theorem 3.10]{B16} in the context of higher order autonomous elliptic equations and Barton used an iteration over the annuli to show that the terms coming from the intermediate derivatives can be neglected. The same reasoning as in the proof of \cite[Theorem 3.10]{B16} shows the self improving property of our inequality \eqref{claim1_strong}, namely we obtain
		
		 \begin{claim} Assume that \eqref{claim1_strong} holds for any $0<\rho<r<R$ and some $a<a'<d<b$. Then there is a constant $C$ dependent on ellipticity and dimensions, such that if $0<r<R$ then $u$ satisfies the stronger inequalities \begin{enumerate}[label=\upshape { (\roman*)}]
		 	\item  $
		 	\int_{a'}^d \|\nabla ^m u(s,\cdot)\|_{L^2(B_r)}^2ds \leq  \frac{C}{(R-r)^{2m}}\int_{a'}^d \| u(s,\cdot)\|_{L^2(B_R\setminus B_r)}^2ds +C\|u(a',\cdot)\|_{L^2(B_R)}^2
		 	$.
		 	\item For integer $0\leq j \leq m$, \begin{align*}
		 	\int_{a'}^d \int_{B_r}|\nabla ^j u|^2dxds \leq  \frac{C}{(R-r)^{2j}}\int_{a'}^d \int_{B_R}| u|^2dxds +C(R-r)^{2m-2j}\|u(a',\cdot)\|_{L^2(B_R)}^2.
		 	\end{align*}
		 	\item $\|u(d,\cdot)\|_{L^2(B_r)}^2\leq \frac{C}{(R-r)^{2m}}\int_{a'}^d \| u(s,\cdot)\|_{L^2(B_R)}^2ds +C\|u(a',\cdot)\|_{L^2(B_R)}^2.$
		 \end{enumerate}
	 \end{claim}
		Let us finish the proof of Proposition \ref{pro:local_energy}. Integrating the estimate (i) over $a'\in [c,d]$ and applying Fubini's Theorem gives us
		\begin{align*}
		\int_c^d(s-c) \int_{B_r}|\nabla ^m u|^2dxds&=\int_c^d\int_{a'}^d \int_{B_r}|\nabla ^m u|^2dxdsda' \\&\leq  \frac{C}{(R-r)^{2m}}\int_c^d\int_{a'}^d\int_{B_R\setminus B_r}| u|^2dxdsda' +C\int_c^d\int_{B_R}|u|^2dxda'
		 \\&\leq C\left( \frac{ (d-c)}{(R-r)^{2m}}+1\right)\int_c^d\int_{B_R}| u|^2dxds.
		\end{align*} That is for any $a<\tilde a < c$ we obtain
		\begin{align*}
		(c-\tilde a)\int_c^d \int_{B_r}|\nabla ^m u|^2dxds&\leq C\left( \frac{ (d-c)}{(R-r)^{2m}}+1\right)\int_c^d\int_{B_R}| u|^2dxds.
		\end{align*}
		The bound for the intermediate derivatives follows in the same way if we integrate (ii). 
		
		Finally, we integrate (iii) in $a'$ to arrive at 
		   \begin{equation*}
		   \|u(d,\cdot)\|_{L^2(B_{r})}^2 \leq C\left(\frac{1}{(R-r)^{2m}} + \frac{1}{d-c}\right)\int_c^d \|u(s,\cdot)\|_{L^2(B_R)}^2ds.\qedhere
		   \end{equation*}
	\end{proof}
	
	When studying the $L^p(\BBR^n)$ well-posedness theory, it will be possible to reduce some proofs to the case $p>2$ by duality. For this reason we introduce the backwards in time equation, the propagator for which will turn out to be the adjoint of the propagator for some equation of type \eqref{eq:the parabolic equation}. 
	\begin{definition}
		\label{def:backwards equation}
		Let $T>0$ and $\tilde A \in L^\infty(\BBR^{n+1}_+;\BBC^{MN\times MN})$ satisfy the ellipticity estimates. We say $u\in L^2_{loc}(0,T; H^m_{loc}(\Omega))$ is a local weak solution to the backwards in time equation up to time $T>0$,
		\begin{equation}
		\label{eq:backwards}
		\partial_s u(s,x) = (-1)^m \text{div}_m \tilde A(s,x) \nabla^m u \quad\mbox{on}\quad (0,T)\times \Omega,
		\end{equation} if for any $\phi \in \mathscr C_c^\infty((0,T)\times \Omega)$ it holds
		\[-\int_0^T\int_\Omega u(t,x)\overline{\partial_t \phi (t,x)}dxdt=\int_0^T\int_\Omega\tilde A(t,x)\nabla^m u(t,x)  \overline{\nabla^m \phi(t,x)} dxdt.\]
	\end{definition}
	\begin{remark} \label{rem:backwards eq}
		 We see directly from the weak formulation above that if $u$ is a weak solution of \eqref{eq:backwards} on $(0,T)\times\Omega$, then $u(T-t,x)$ is a local weak solution of \eqref{eq:the parabolic equation} on $(0,T)\times\Omega$ with $A(t,x)=\tilde A(T-t,x)$. Thus, we obtain the continuity in time of such solutions with values in $H^m_{loc}(\Omega)$, as well as the analogous quantitative energy estimates from Proposition \ref{pro:local_energy}.
	\end{remark}
	
	The local energy estimates can be used to derive reversed H\"older estimates. For this, note that $(0,\infty)\times \BBR^n$, equipped with the quasi-distance \[d((t,x),(s,y))\coloneqq\max\{|t-s|^{1/(2m)},|x-y|\}\] and the Lebesgue measure, is a space of homogeneous type. We denote $B_R(t,x)=[t-R^{2m},t+R^{2m}]\times B(x,R)$.
	\begin{lemma}
		\label{lem:hoelder}
		Let $q=2+\frac{4m}{n}$. Then there exists a constant $C_{\lambda,\lambda,m,n}>0$, such that for all global weak solutions $u$ to \eqref{eq:the parabolic equation}, $(t,x)\in (0,\infty)\times \BBR^n$ and all $(4r)^{2m}\in (0, t)$ the following inequality holds 
		\begin{equation}
		\left(\fint_{B_r(t,x)} |u(s,y)|^qdyds\right)^{1/q}\leq C \left(\fint_{B_{4r}(t,x)} |u(s,y)|^2dyds\right)^{1/2}.
		\end{equation}
	\end{lemma}
	\begin{proof}
		This follows from Lemma \ref{lem:Gagliardo--Nirenberg inequality} an Proposition \ref{pro:local_energy}. For the full argument, see the proof of \cite[Lemma 4.1]{AMP15}.
	\end{proof}
	Similarly as in \cite[\S4.1]{AMP15}, the reversed H\"older equality above holds for an improved exponent $\tilde q >q $, which is an application of a Gehring's Lemma type argument for spaces of homogeneous type, see \cite{BB11}. For the exponent on the right hand side we can even choose any $p\in [1,2]$ (for a proof of this self-improving property in the setting of spaces of homogeneous type we refer to \cite[Theorem B1]{BCF15}).
	\begin{corollary}
		\label{cor:rev_hoelder}
		There exist $C>0$ and $\tilde q >2+\frac{4m}{n}$ both dependent on the ellipticity and dimensions, such that for any global weak solution $u$ of \eqref{eq:the parabolic equation}, every $(t,x)\in (0,\infty)\times \BBR^n$ and all $(4r)^{2m}\in (0, t),$ we have  
		\begin{align*}
		\left(\fint_{B_{r}(t,x)} |u(s,y)|^2dyds\right)^{1/2}\leq\left(\fint_{B_r(t,x)} |u(s,y)|^{\tilde q}dyds\right)^{1/{\tilde q}}
		\leq C \left(\fint_{B_{4r}(t,x)} |u(s,y)|dyds\right).
		\end{align*}
	\end{corollary}
	\section{Traces of tent space solutions}
	\setcounter{theorem}{0} \setcounter{equation}{0}
	\label{sec:traces tent space}
We show that global weak solutions $u$ to \eqref{eq:the parabolic equation} satisfying $\|\nabla^m u \|_{\tentspace}<\infty$ possess a distributional trace, which, up to a polynomial, lies in $L^p$ or $BMO$ space. If $p\in(1,2]$ we also deduce a bound on the $L^\infty(L^p)$-norm of the solution modified by the same polynomial. 
		\begin{lemma}\label{lem:trace existence}
			Let $p\in [1,2]$ and $u$ be a global weak solution of \eqref{eq:the parabolic equation} with $S=\|\nabla^m u \|_{\tentspace}<\infty$.
			Then it holds $\|\nabla^m u \|_{L^2((\varepsilon,\infty)\times \BBR^n)}<\infty$ for all $\varepsilon>0.$ Further, there exists a unique distributional trace of $u$ at $t=0$.
		\end{lemma}
		\begin{proof}
			We first claim
			\begin{align*}
			\|\nabla^m u \|_{L^2((\varepsilon,\infty)\times \BBR^n)}\lesssim \varepsilon ^{\frac{n}{2m}(\frac{1}{2}-\frac{1}{p})}S.
			\end{align*} For $p=2$ there is nothing to show. 
			For $p<2$ we write for $\varepsilon>0$
			\begin{align*}
			\|\nabla^m u \|_{L^2((\varepsilon,\infty)\times \BBR^n)} &= \left(\int_{\BBR^n} \int_{\varepsilon}^\infty \fint_{B(x,\sqrt[2m]{t}/4)}|\nabla^m u |^2dydtdx\right)^{1/2} \\
			&= \left(\sum_{z\in \frac{\sqrt[2m]{\varepsilon}}{4\sqrt{n}}\BBZ^n} \int_{Q(z)}\int_{\varepsilon}^\infty \fint_{B(x,\sqrt[2m]{t}/4)}|\nabla^m u |^2dydtdx\right)^{1/2},
			\end{align*}
			where for $z\in \frac{\sqrt[2m]{\varepsilon}}{4\sqrt{n}}\BBZ^n$ we denoted by $Q(z)$ the cube $z+\frac{\sqrt[2m]{\varepsilon}}{4\sqrt{n}}(0,1)^n$. We write $c_{Q(z)}$ for the center of $Q(z)$ and realize for any $t>\varepsilon$ and $x\in Q(z)$
			\[B(x,\sqrt[2m]{t}/4)\subseteq B(c_{Q(z)}, \sqrt[2m]{t}/2)\subseteq B(x, \sqrt[2m]{t}),\] hence it holds
			\begin{align*}
			\|\nabla^m u \|_{L^2((\varepsilon,\infty)\times \BBR^n)} &\lesssim  \left(\sum_{z\in \frac{\sqrt[2m]{\varepsilon}}{4\sqrt{n}}\BBZ^n} \int_{Q(z)}\int_{\varepsilon}^\infty \fint_{B(c_{Q(z)}, \sqrt[2m]{t}/2)}|\nabla^m u |^2dydtdx\right)^{1/2}\\
			&\lesssim \varepsilon^{\frac{n}{2m}\frac{1}{2}} \left(\sum_{z\in \frac{\sqrt[2m]{\varepsilon}}{4\sqrt{n}}\BBZ^n} \int_{\varepsilon}^\infty \fint_{B(c_{Q(z)}, \sqrt[2m]{t}/2)}|\nabla^m u |^2dydt\right)^{1/2}\\
			&\lesssim \varepsilon^{\frac{n}{2m}\frac{1}{2}} \left(\sum_{z\in \frac{\sqrt[2m]{\varepsilon}}{4\sqrt{n}}\BBZ^n} \left(\int_{\varepsilon}^\infty \fint_{B(c_{Q(z)}, \sqrt[2m]{t}/2)}|\nabla^m u |^2dydt\right)^{p/2}\right)^{1/p}\\
			&\lesssim \varepsilon^{\frac{n}{2m}(\frac{1}{2}-\frac{1}{p})} \left(\sum_{z\in \frac{\sqrt[2m]{\varepsilon}}{4\sqrt{n}}\BBZ^n}\int_{Q(z)} \left(\int_{\varepsilon}^\infty \fint_{B(x, \sqrt[2m]{t})}|\nabla^m u |^2dydt\right)^{p/2}dx\right)^{1/p}\\
			&\lesssim \varepsilon^{\frac{n}{2m}(\frac{1}{2}-\frac{1}{p})}S,
			\end{align*} 
			where in the third line we needed the embedding of the sequence spaces $\ell_p \hookrightarrow \ell_2$ as $p<2$.
			The weak formulation of the equation \eqref{eq:the parabolic equation} implies $\partial_t u \in L^2(\varepsilon,\infty; H^{-m}(\BBR^n))$. Let $\nu \in \mathscr C^\infty ((0,\infty))$ satisfy $\nu=1$ on $(0,1)$ an $\nu =0$ on $(2,\infty)$. Applying Lemma \ref{lem:absolute continuity} for intervals $[\varepsilon,2]$ and any $\varepsilon \in (0,1)$, we obtain for every $\psi \in \testfunctions$
			\begin{align*}
			(u(\varepsilon, \cdot),\psi)_{L^2}=(u(\varepsilon, \cdot),\psi \nu (\varepsilon))_{L^2}
			=-\int_{\varepsilon}^{\infty}(\partial_t u, \psi \nu)dt -\int_{1}^{2} (u,\psi\partial_t \nu) dt.
			\end{align*}
			Thus we find for $0<\delta<\varepsilon<1$ 
			\begin{align*}
			|(u(\varepsilon, \cdot),\psi)_{L^2}-(u(\delta, \cdot),\psi)_{L^2}|&=\left|\int_{\varepsilon}^{\infty}(A\nabla^m u(t), \nabla^m \psi \mathbbm{1}_{(\delta,\varepsilon)}(t)\nu(t))dt\right|\\
			&\lesssim \|\nabla^mu \|_{\tentspace} \|\nabla^m \psi \mathbbm{1}_{(\delta,\varepsilon)}(t)\nu(t)\|_{T^{p',2}_m}.
			\end{align*}
			Simple estimates with help of the Hardy--Littlewood maximal function yield, as $p\leq 2$,
			\[\|\nabla^m \psi \mathbbm{1}_{(\delta,\varepsilon)}(t)\nu(t)\|_{T^{p',2}_m}\lesssim |\varepsilon-\delta|\|\nabla^m \psi\|_{L^{p'}(\BBR^n)},\] thus there exists a distributional limit $u_0\in \mathscr D'(\BBR^n)$, as claimed. 
		\end{proof}
		\begin{lemma}\label{lem:trace integrability}
			Let $p\in (1,2]$ and $u$ be a global weak solution of \eqref{eq:the parabolic equation} with $S=\|\nabla^m u \|_{\tentspace}<\infty$ and distributional trace $u_0$. Then there exists a unique polynomial $P\in \pol$ with $u_0-P\in L^p(\BBR^n)$ and 
			\begin{equation*}
			\sup_{t\ge 0}\|u(t)-P\|_{L^p(\BBR^n)}\lesssim S.
			\end{equation*}
		\end{lemma}
		
		\begin{proof} 
			Consider the weight $\omega=\eta^2$ for some non-negative, radially symmetric $\eta \in \mathscr C_c^\infty (B(0,1/3^{2m}))$. We assume $\int \omega =1$ and $0<c=\omega $ on $B(0,1/4^{2m})$.
			We first show for $x_0\in \BBR^n$, $r>0$ and $t\in (r,2r)$
			\begin{equation}\label{eq:01}
			\fint_{B(x_0,\sqrt[2m]{r})}\left|u(t,y)-\mathbb{P}^\omega_{x_0,\sqrt[2m]{r}}(u(t,\cdot))(y)\right|^2\omega\left(\frac{y-x_0}{\sqrt[2m]{r}}\right)dy\lesssim\int_{r/2}^{2r}\fint_{B(x_0,\sqrt[2m]{s})}|\nabla^m u|^2dys, \end{equation}
			where $\mathbb{P}^\omega_{x_0,\sqrt[2m]{r}}$ denotes the corresponding orthogonal projection on $\pol$ from Section \ref{sec:Campanato} with respect to the weighted scalar product. 
			
			By scaling and translation it suffices to consider $x_0=0$ and $r=1$. The minimizing polynomial can be replaced by any other $p\in\pol$, for instance $p=P(1)$, where \[P(t)(x) = \sum_{|\alpha|\leq m-1} (u(t,\cdot),\phi_\alpha)x^\alpha\] is defined for every $t>0$ with functions $\phi_\alpha\in \mathscr D(B(0,1/2))$ as in Lemma \ref{lem:generalized Poincare}. The crucial observation is that the coefficients of $P(t)$ are absolutely continuous over the interval $[1/2,2]$. Indeed, the distributional derivative of $c_\alpha$ is given by
			$-(A\nabla^m u(t,\cdot),\nabla^m\phi_\alpha)_{L^2(B(0,1/2))}$ and thus belongs to $L^1(1/2,2)$. Therefore, for each $0\leq|\alpha|\leq m-1$ and any $t\in (1,2)$, 
			\begin{equation}
			\label{eq:bound on the coefficients1}
			|c_\alpha(t)-c_\alpha(1)|\lesssim \left(\int_{1/2}^2 \int_{B(0,1)}|\nabla^m u|^2dxds\right)^{1/2}
			\end{equation}	
			with constant depending only on $\phi_\alpha$, the ellipticity and dimensions. Estimate \eqref{eq:01} follows now directly from the local energy estimates (cf.\ Proposition \ref{pro:local_energy}), the Poincar\'e inequality from Lemma \ref{lem:generalized Poincare} and \eqref{eq:bound on the coefficients1}.
			As we consider the weighted scalar product, the same arguments show that also the coefficients of $\mathbb{P}^\omega_{x_0,\sqrt[2m]{r}}(u(t,\cdot))$ are absolutely continuous over $[r,2r]$ and we have 
			\begin{equation}\label{eq:02}
			\fint_{B(x_0,\sqrt[2m]{r/4})}\left|u(t,y)-\mathbb{P}^\omega_{x_0,\sqrt[2m]{r}}(u(r,\cdot))(y)\right|^2dy\lesssim\int_{r/2}^{2r}\fint_{B(x_0,\sqrt[2m]{s})}|\nabla^m u|^2dys. \end{equation} By definition, we observe $\mathbb{P}^\omega_{x_0,\sqrt[2m]{r}}(\mathbb{P}^\omega_{x_0,\sqrt[2m]{2r}})=\mathbb{P}^\omega_{x_0,\sqrt[2m]{2r}}$ and the known bounds on minimal polynomials together with \eqref{eq:02} lead to the estimate 
			\begin{equation*}
			\|\mathbb{P}^\omega_{x_0,\sqrt[2m]{r}}(u(r,\cdot))-\mathbb{P}^\omega_{x_0,\sqrt[2m]{2r}}(u(2r,\cdot))\|_{L^\infty(B(0,\sqrt[2m]{r}))}\lesssim \left(\int_{r/2}^{2r}\fint_{B(x_0,\sqrt[2m]{s})}|\nabla^m u|^2dys\right)^{1/2}\lesssim r^{-\frac{n}{2m}\frac{1}{p}}S. \end{equation*}
			Analogously as in Lemma \ref{lem:filter out the pol} and Proposition \ref{prop:structure of L^p_m}, we obtain a limiting polynomial as $r$ tends to $\infty$, which does not depend on the center of the ball. We thus get for any $x_0\in \BBR^n$ and $t >0$  
			\begin{equation}\label{eq:04}
			\fint_{B(x_0,\sqrt[2m]{t/4})}\left|u(t,y)-P(y)\right|^2dy\lesssim\int_{t/2}^{\infty}\fint_{B(x_0,\sqrt[2m]{s})}|\nabla^m u|^2dys .\end{equation} 
			From here, we obtain with Fubini's Theorem and H\"older inequality,
			\begin{align*}
			\|u(t)-P\|_{L^p(\BBR^n)}	
			\leq \left\|\left(\fint_{B(\cdot,\sqrt[2m]{t/4})} |u(t)-P|^2dy\right)^{1/2}\right\|_{L^p(\BBR^n)}\lesssim S.
			\end{align*}
			Thus $\sup_{t>0} \|u(t)-P\|_{L^p(\BBR^n)}\lesssim S.$ As any weak-* limit must coincide with the distributional one, we conclude that $u_0-P \in L^p(\BBR^n)$. The polynomial must clearly be unique, as the only $p$-integrable polynomial is the zero polynomial.
		\end{proof}
		We now turn to the case $p\in(2,\infty]$.
		\begin{lemma} \label{lem:trace p>2}
			Let $p\in (2,\infty]$ $u$ be a global weak solution $u$ to \eqref{eq:the parabolic equation} with $S=\|\nabla^m u \|_{\tentspace}<\infty$. Then $u$ has an $L^2_{loc}(\BBR^n)$ trace $u_0$ at $t=0$ and there exists a unique (up to a constant if $p=\infty$) polynomial $P\in \pol$ such that
			\begin{equation}\label{eq:05}
			\|u_0-P\|_{Y}\lesssim S.
			\end{equation} Here, $Y=L^p(\BBR^n)$ if $p\in(2,\infty)$ or $Y=BMO(\BBR^n)$ if $p=\infty$. 
		\end{lemma}
		\begin{proof}
			Our assumptions imply for every $r>0$ and $x_0\in \BBR^n$
			\begin{equation}
			\label{eq:eqv norm2}
			\sup_{R\ge r} \left(\int_0^{R^{2m}} \fint_{B(x_0,R)}\left|\nabla^m u\right|^2dy\right)^{1/2}=:C_{x_0,r}(|\nabla^mu|)\lesssim r^{-\frac{n}{p}} \|\nabla^m u\|_{\tentspace}.
			\end{equation}
			Indeed, recall Remark \ref{rem:eqv norm tent space} and restrict the integration domain to the ball $B(x_0,r)$.
			In particular, $\nabla^m u \in L^2(0,T;L^2(B(x_0,R)))$ for any $T,\,R>0$ and $x_0\in \BBR^n$. 
			
			To obtain the $L^2_{loc}(\BBR^n)$ trace of $u$ at $t=0$ it suffices to show $u\in L^2(0,1;H^m(B(x_0,1)))$ for any $x_0\in\BBR^n$. Indeed, Lemma \ref{lem:absolute continuity} implies as at the beginning of the proof of Proposition \ref{pro:local_energy} that $u\in \mathscr C([0,1];L^2(B(x_0,1/2)))$.

			By translation, it is enough to assume $x_0=0$. We adapt the arguments from Lemma \ref{lem:trace integrability}. Lemma \ref{lem:generalized Poincare} implies for almost every $t\in (0,1)$ and any $k\in \{0,\dots,m\}$
			\begin{equation}\label{eq:Poincare in use}
			\|\nabla^k (u(t,\cdot) -P(t))\|_{L^2(B(0,1))} \leq C \|\nabla^mu(t,\cdot)\|_{L^2(B(0,1))}
			\end{equation} with the polynomial $P(t)(x) = \sum_{|\alpha|\leq m-1} (u(t,\cdot),\phi_\alpha)x^\alpha$. As $\nabla^m u \in L^2(0,1;L^2(B(0,1)))$, we deduce that the coefficients of $P(t)$ are absolutely continuous over $[0,1]$. For each $0\leq|\alpha|\leq m-1$, there exist $c_\alpha(0)\in \BBC$ such that for any $t\in (0,1)$, 
				\begin{equation}
				\label{eq:bound on the coefficients}
				|c_\alpha(t)-c_\alpha(0)|\leq \Lambda \int_0^t \int_{B(0,1/2)} |\nabla^m u||\nabla^m\phi_\alpha|dxds \lesssim \left( \int_0^1 \int_{B(0,1)} |\nabla^m u|^2dxds\right)^{1/2}
				\end{equation}	
				with constant depending only on $\phi_\alpha$, the ellipticity and dimensions.
				Define the polynomial $P(0)$ by $P(0)(x) = \sum_{0\leq|\alpha|\leq m-1}c_\alpha(0) x^\alpha$. Combining \eqref{eq:Poincare in use} and \eqref{eq:bound on the coefficients} gives
				\begin{align}\label{eq:poincare used}
				\sum_{k=0}^{m}\|\nabla^k (u -P(0))\|_{L^2(0,1;B(0,1))}&\lesssim C_{0,1}(|\nabla^m u|).
				\end{align}
		
			We continue with showing that the trace $u_0$ lies in the desired space. Consider a weight $\omega$ as in Lemma \ref{lem:trace integrability}. It suffices to show
				\[\|(u_0)^{\#,m}_\omega\|_{L^p}\lesssim \|\nabla^m u \|_{\tentspace},\]
			where we introduced the weighted sharp function 
				\[f^{\#,m}_\omega(x)=\sup_{B(x_0,r)\ni x} \inf_{P\in \mathcal P_{m-1}}\left( \frac{1}{r^n}\int_{B(x_0,r)}\left|f(y)-P(y)\right|^2\omega\left(\frac{y-x_0}{r}\right)dy\right)^{1/2}.\] It is of course enough to prove the pointwise inequality, which after translation and rescaling reads
				\begin{align*}
				\inf_{P\in \mathcal P_{m-1}}\left( \int_{B(0,1)}\left|u_0(y)-P(y)\right|^2\omega(y)dy\right)^{1/2}\lesssim C_{0,1}(|\nabla^m u|).
				\end{align*} We show 
				\begin{align}
				\label{eq:easy bound on trace}
				\left( \int_{B(0,1)}\left|u_0(y)-P(0)(y)\right|^2\omega(y)dy\right)^{1/2}\lesssim C_{0,1}(|\nabla^m u|)
				\end{align} for the polynomial $P(0)\in \pol$ obtained above.
				We have by \eqref{eq:poincare used} $ (u-P(0))\eta \in L^2(0,1;H^m_0(B(0,1)))$ and, using $\nabla^m u\in L^2(0,1;L^2(B(0,1)))$ and the equation, also $\partial_t( u-P(0))\eta\in L^2(0,1;H^{-m}(B(0,1))).$ 
				Thus, by Lemma \ref{lem:absolute continuity} the map $$[0,1]\ni t \to \|(u(t,\cdot)-P(0))\eta\|_{L^2(B(0,1))}^2$$ is absolutely continuous. The function $\phi(t)=1-t$ is $\mathscr C^1$ on $[0,1]$, so we may estimate
				\begin{align*}
				\int_{B(0,1)}\left|u_0-P(0)\right|^2\omega dy&=-\int_0^1 \partial_t\left( (1-t)\int_{B(0,1)}\left|u(t,y)-P(0)(y)\right|^2\omega(y)dy\right)dt\\
				&=\int_0^1 \int_{B(0,1)}\left|u(t,y)-P(0)(y)\right|^2\omega(y)dydt\\&-2\text{Re} \int_0^1  (1-t)\langle\partial_tu(t,\cdot),(u(t,\cdot)-P(0))\omega\rangle_{H^{-m}(B(0,1)),H^m_0(B(0,1))} dt.
				\end{align*} 
				
				The first summand is bounded by \eqref{eq:poincare used}. For the second one, we approximate $(1-t)$ by a sequence of smooth compactly supported functions on $(0,1)$, use the equation \eqref{eq:the parabolic equation} and pass to the limit to obtain 
				\begin{align*}
				\bigg|\int_0^1  (1-t)\langle\partial_tu,(u-P(0))\omega\rangle dt\bigg|
				&\lesssim \sum_{k=0}^{m} \int_0^1 \int_{B(0,1)} |\nabla^m u||\nabla^k (u-P(0))||\nabla^{m-k} \omega|dxdt 
				\\&\hspace{-0.1cm}\stackrel{\eqref{eq:poincare used}}{\lesssim} \int_0^1 \int_{B(0,1)} |\nabla^m u|^2 dxdt .	
				\end{align*}
				This finishes the proof.\qedhere
		\end{proof}
		In Proposition \ref{pro:p>2 pol existence} we prove the converse inequality to \eqref{eq:05}. If $p<2$, this type of estimate will be shown to hold under some extra assumptions on the operator $L$.
		
	\section{Energy well-posedness and the propagators}
		\setcounter{theorem}{0} \setcounter{equation}{0}
		\label{sec:energy solutions}
	We demonstrate that in the energy case, Lemma \ref{lem:trace integrability} easily leads to the uniqueness of solutions.
	\begin{theorem}
		\label{thm:global_well}
		The Cauchy problem \eqref{eq:Cauchy problem} is well-posed for $Y=L^2(\BBR^n)$ and $$X=\{u\in \mathscr D'(\BBR^{n+1}_+)\mid\nabla^mu\in L^2(\BBR^{n+1}_+)\}.$$ 
		For any $u_0\in L^2(\BBR^n)$ and $T>0$, the unique global weak solution $u\in X$ with trace $u_0$ satisfies $$u\in\mathscr C_0([0,\infty); L^2(\BBR^n))\cap L^2(0,T; H^m(\BBR^n))$$ and $\|u(t,\cdot)\|_{L^2}$ is decreasing. We have the global estimates
		\begin{align}
		\label{eq:global energy sol}
		\|u_0\|_{L^2}=\|u\|_{L^\infty(L^2)}\leq\sqrt{2\Lambda}\|\nabla^m u\|_{L^2(L^2)}\leq \sqrt{\frac{\Lambda}{\lambda}}\|u_0\|_{L^2}.
		\end{align}
	\end{theorem}
	\begin{proof} 
		{\par\noindent\textbf{Uniqueness.}\space} Let $u$ be a global weak solution of \eqref{eq:the parabolic equation} with $\nabla^m u \in L^2(L^2)$ and assume that the $L^1_{loc}(\BBR^n)$ trace $u_0$ of $u$ exists and belongs to $L^2(\BBR^n)$. Then $u_0$ equals the distributional trace and the polynomial $P$ from Lemma \ref{lem:trace integrability} is the zero polynomial. We obtain from inequality \eqref{eq:04} for every $t\ge0$
			\begin{equation}
				\|u(t)\|_{L^2(\BBR^n)}^2\lesssim\int_{t/2}^{\infty}\int_{\BBR^n}|\nabla^m u|^2dys.\end{equation} 
	Thus, $u\in L^2(0,T; H^m(\BBR^n))$ for any $T>0$ and the norm $\|u(t)\|_{L^2(\BBR^n)}$ is decreasing and vanishing at infinity. 
	Consequently, we deduce from the weak formulation $\partial_t u\in L^2(a,b;H^{-m}(\BBR^n))$ for any $0\leq a<b<\infty$, so by Lemma \ref{lem:absolute continuity} the map $[a,b] \ni t \to \|u(t)\|_{L^2}^2$ is absolutely continuous with 
	\begin{align*}
	\|u(a)\|_{L^2}^2-\|u(b)\|_{L^2}^2
	&= 2\text{Re}\int_a^b \int_{\BBR^n}A(t,x) \nabla^m u(t,x) \overline{ \nabla^m u(t,x)}dxdt.
	\end{align*} Here, we could treat $u$ as a test function by continuity. The global estimates follow now easily by taking limits and using the ellipticity bounds. Finally, \eqref{eq:global energy sol} provides uniqueness. 
	
		{\par\noindent\textbf{Existence.}\space} 
		Given $u_0\in L^2(\BBR^n)$, a solution with $L^2(\BBR^n)$ trace $u_0$ can be constructed by finite dimensional Galerkin approximations, see for example \cite[Chap.\ XVIII, \S3.1-3]{DL92}. Another constructive proof is based on an approximation of the coefficient matrix $A$ by piecewise constant in time matrices and uses the semigroup theory. This has been done in the second order case in \cite[Theorem 3.11]{AMP15}.
	\end{proof}
	In Theorem \ref{thm:global_well} we could also start from any time $s\ge 0 $ and initial data $u_s \in L^2(\BBR^n)$ to obtain the unique weak solution $u_s(t,\cdot)$ to \eqref{eq:the parabolic equation} on $(s,\infty)\times \BBR^n$ satisfying $\nabla ^m u \in L^2(s,\infty; L^2)$ and $u_s(s,\cdot)=u_s$. This gives rise to the central object of our study, the propagator $\Gamma(t,s)$, defined for $0\leq s\leq t<\infty$ by 
	\[\Gamma(t,s)u_s(x) \coloneqq u_s(t,x), \quad\mbox{whenever} \quad (t,x) \in [s,\infty)\times \BBR^n.\] 
	
	By Theorem \ref{thm:global_well}, the propagators are contractions on $L^2(\BBR^n)$ and it is easily shown by uniqueness of solutions that $\Gamma (t,t)=I$ holds for $t\ge 0$ and $\Gamma(t,s)\Gamma(s,r)=\Gamma(t,r)$ is true on $L^2(\BBR^n)$, whenever $0\leq r\leq s \leq t$. Moreover, for any $s\ge0$ we have
	\begin{equation}
	\label{eq:continuity of prop}
	[s,\infty) \ni t \mapsto \Gamma(t,s) \in \mathscr C_0 ([s,\infty); \mathscr L(L^2)).
	\end{equation}
	\begin{definition}
		\label{def:propagators}
		With the above notation, we call
			\begin{equation}
			\label{eq:propagators family}
			\{\Gamma(t,s) \mid 0\leq s\leq t<\infty\}\subseteq \mathscr L (L^2)\end{equation}
		the family of propagators associated to \eqref{eq:the parabolic equation}.
	\end{definition}
	If the coefficient matrix is given on the entire space $\BBR^{n+1}$, we could define the propagators also for negative times. Similarly, recalling Remark \ref{rem:backwards eq}, we easily obtain existence of propagators for the backward equation \eqref{eq:backwards}, 
		\begin{equation}
	\label{eq:backpropagators family}
	\{\tilde\Gamma(t,T)| \; -\infty< t\leq T\}\subseteq \mathscr L (L^2).\end{equation}
	We set $\tilde\Gamma(t,T) = \Gamma(T-t,0)$ if $t\in (-\infty, T]$, where $\Gamma(t,s)$ is defined as above for the matrix $A(t,x)\coloneqq\tilde A(T-t,x)$ on $(-\infty, T)$ and constant otherwise (we study the backwards equation on $[0,T]$, so the precise extension of $A$ to $t>T$ is irrelevant for later applications). 
	
	As announced before, the adjoints of \eqref{eq:propagators family} on a finite time interval can be expressed by the backwards propagators to \eqref{eq:backwards} for special choice of $\tilde A$.  
	\begin{lemma}\label{lem:backwards_prop}
		Let $T>0$ and fix some coefficient matrix $A$. Consider the associated family \eqref{eq:propagators family} and the backwards propagators \eqref{eq:backpropagators family} for $\tilde A = A^*$ (the conjugate transpose) defined up to time $T>0$. Then it holds 
		\begin{equation*} 
		\tilde \Gamma(t,T)=\Gamma(T,t)^* \quad\mbox{for every} \quad 0\leq t\leq T.
		\end{equation*} Thus for all $h\in L^2(\BBR^n)$, $t\mapsto \Gamma(T,t)^*h$ is strongly continuous from $[0,T]$ into $L^2(\BBR^n)$, and, consequently, $t\mapsto \Gamma(T,t)h$ is weakly continuous from $[0,T]$ into $L^2(\BBR^n)$.  
	\end{lemma}
	\begin{proof}
		An easy consequence of Lemma \ref{lem:absolute continuity} and the weak formulation. See \cite[Proposition 3.17]{AMP15}. 
	\end{proof}
	
	The following $L^2$ off-diagonal bounds are a replacement for kernel bounds and will, for example, allow us to extend the family \eqref{eq:propagators family} to a broad class of functions.
	\begin{proposition} \label{prop:off2}
		The family of propagators from \eqref{eq:propagators family} satisfies $L^2$ off-diagonal estimates. That is, there exist constants $c,\; C>0$ depending only on the ellipticity and dimensions, such that for all closed sets $E,\;F\subseteq \BBR^n$, any function $f\in L^2(\BBR^n)$ and $0\leq s<t<\infty$ it holds
		\begin{equation}
		\label{eq:off-diag}
		\|\mathbbm{1}_E\Gamma(t,s)(\mathbbm{1}_F f)\|_{L^2}\leq C\exp\left\{ -c\left(\frac{d(E,F)}{(t-s)^{1/(2m)}}\right)^{\frac{2m}{2m-1}}\right\}\|\mathbbm{1}_F f\|_{L^2}.
		\end{equation}
	\end{proposition}
	\begin{proof}
		Without loss of generality we assume $s=0$ and use the time consistency of the propagators otherwise. Indeed, $\Gamma(t,s)h$ can be expressed as $\Gamma_s(t-s,0)h$ where $\Gamma_s$ arises from the matrix $A(t,x)=A(t+s,x)$, satisfying the same ellipticity bounds as $A$. 
		
		Let $E,\;F$ be two closed sets with $d=d(E,F)>0$. Note that there is nothing to show if $d(E,F)=0$, provided we choose $C\ge 1$. Consider the function $h\colon \BBR^n\to [0,d/2]$ given by $$h(x)=\min \left\{\max\{0;d(x,F)-d/4\};d/2\right\}.$$ 
		
		In particular, $h=0$ on the $d/4$-neighborhood of $F$ and $h=d/2$ on the $d/4$-neighborhood of $E$. Let $\eta\in\testfunctions$ be a non-negative function supported in the unit ball with $\int \eta =1$ and set $\eta_\varepsilon=\varepsilon^{-n}\eta(\cdot/\varepsilon)$ for $\varepsilon >0$. Finally, let us define $\phi \coloneqq h\ast \eta_{d/8}$. Then $\phi$ is smooth, non-negative and satisfies $\phi_{|F}=0$, $\phi_{|E}=d/2$ and $\|\partial^\alpha \phi\|_{L^\infty} \lesssim d^{1-|\alpha|}$ for any multi-index $\alpha$.
		
		Let $\kappa>0$ be a parameter to be specified later and consider the conjugated propagator $$\Gamma^{\kappa\phi}(t,0)\colon L^2(\BBR^n) \ni f\mapsto e^{\kappa\phi}\Gamma (t,0)(e^{-\kappa\phi}f).$$ 
		
		We derive an $L^2$ bound for this operator in dependence on $t, \, d $ and $\kappa$. First notice that $u(t)= e^{-\kappa\phi}\Gamma^{\kappa\phi}(t,0)f=\Gamma (t,0)(e^{-\kappa\phi}f)$ is a global energy solution of \eqref{eq:the parabolic equation}, which belongs to $ L^2(0,T;\sob)$ for any $T>0$ and $\partial_t \Gamma^{\kappa\phi}(t,0)f \in L^2(0,\infty; H^{-m}(\BBR^n))$. Hence, by Lemma \ref{lem:absolute continuity}, we have for almost every $t>0$
		\begin{align*}
			\frac{d}{dt}\|\Gamma^{\kappa\phi}(t,0)f\|_{L^2}^2=\frac{d}{dt}\| e^{\kappa\phi}u(t)\|_{L^2}^2 =	2\text{Re}\langle \partial_t( e^{\kappa\phi} u(t)), e^{\kappa\phi} u(t)\rangle_{ H^{-m}(\BBR^n),\sob}.
		\end{align*}
		
		On the other hand, $A\nabla^m u \in L^2(0,T;L^2)$ for $T>0$, so, by continuity, $ e^{2\kappa\phi} u$ can be used as a test function and we have for almost every $t>0$
		\begin{align*}
		 \frac{d}{dt}\|w(t)\|_{L^2}^2
		 &= - 2\text{Re} \langle A(t,\cdot)\nabla^m (e^{-\kappa\phi} w(t)), \nabla^m\left(e^{\kappa\phi} w(t)\right)\rangle,
		\end{align*}
		where we put $w(t)=e^{\kappa\phi} u(t)=	\Gamma^{\kappa\phi}(t,0)f\in L^2(\BBR^n)$. We further calculate, similarly as in Proposition \ref{pro:local_energy}, using the product rule
		\begin{align*}
			 \frac{d}{dt}\|w(t)\|_{L^2}^2
			 =&- 2\text{Re} \sum_{|\alpha|=|\beta|=m} \int a_{\alpha,\beta}(e^{-\kappa\phi}\partial^\beta w) \left(e^{\kappa\phi}\overline{\partial^\alpha w}+\sum_{\gamma<\alpha}e^{\kappa\phi}\Phi_{\alpha,\gamma}\overline{\partial^{\gamma}w}\right)dx \\
			&-2\text{Re} \sum_{|\alpha|=|\beta|=m} \int a_{\alpha,\beta}\left(\sum_{\xi<\beta}e^{-\kappa\phi}\Phi_{\beta,\xi}\partial^\xi w\right) \left(e^{\kappa\phi}\overline{\partial^\alpha w}+\sum_{\gamma<\alpha}e^{\kappa\phi}\Phi_{\alpha,\gamma}\overline{\partial^{\gamma}w}\right)dx,
			\end{align*}
			where this time $\Phi_{\alpha,\gamma}\in C^\infty(\BBR^n)$ satisfy $\|\Phi_{\alpha,\gamma}\|_{L^\infty}\lesssim \sum_{(l,s)\colon\; l\ge1,\; l+s=|\alpha|-|\gamma|}\kappa^ld^{-s}$. The exponential factors cancel and we use the ellipticity estimates to obtain
			\begin{align*}
			\frac{d}{dt}\|w(t)\|_{L^2}^2
			\leq& -2\lambda \| \nabla^m w(t)\|_{L^2}^2 +  C\|\nabla ^m w(t)\|_{L^2}\sum_{k=0}^{m-1}\left(\sum_{(l,s)\colon\; l\ge1,\; l+s=m-k}\kappa^ld^{-s}\right)\|\nabla ^k w(t)\|_{L^2} \\&+  C\sum_{k=0}^{m-1}\left(\sum_{(l,s)\colon\; l\ge1,\; l+s=m-k}\kappa^{2l}d^{-2s}\right)\|\nabla ^k w(t)\|_{L^2}^2
			\end{align*} for some $C>0$ dependent on ellipticity and dimensions. We estimate the norms of the intermediate derivatives with the Gagliardo--Nirenberg inequality, see Lemma \ref{lem:Gagliardo--Nirenberg inequality}, and use Cauchy's inequality to absorb the highest order factors in the negative term, which we then drop, to arrive at
			\begin{align*}
				\frac{d}{dt}\|w(t)\|_{L^2}^2
					\lesssim \sum_{k=0}^{m-1}\left(\sum_{(l,s)\colon\; l\ge1,\; l+s=m-k}(\kappa^{l}d^{-s})^{\frac{2m}{m-k}}\right)\|w(t)\|_{L^2}^2.
			\end{align*}
			
			Notice that $w(0)=f$, so, by Gr\"onwall's inequality, there exists a constant $C>0$, such that we have for any $\kappa>0$ \[\|\Gamma^{\kappa\phi}(t,0)f\|_{L^2}^2\leq e^{C\left(\kappa^{2m}+\sum_{k=2}^{m}\left(\sum_{l,s\ge1,\; l+s=k}(\kappa^{l}d^{-s})^{2m/k}\right)\right)t}\|f\|_{L^2}^2,\] where we sum over $s\ge 1$ after taking the corresponding factor for $s=0$, namely $\kappa^{2m}$, out from the sum. Let us assume without loss of generality that $f$ is supported in $F$. Then it holds $e^{-\kappa\phi}f=f$ and $\mathbbm{1}_E\Gamma(t,0) f=e^{-\kappa d/2}\Gamma^{\kappa\phi} (t,0)f$, hence, by above,
			\begin{align*}
			\|\mathbbm{1}_E\Gamma(t,0) f\|_{L^2}^2&\leq  e^{-c\kappa d(E,F)}e^{C\left(\kappa^{2m}+\sum_{k=2}^{m}\sum_{l,s\ge1,\; l+s=k}(\kappa^{l}d^{-s})^{2m/k}\right)t}\|f\|_{L^2}^2.
			\end{align*} 
			
			Since the above estimate holds for any $\kappa>0$, we can choose the one for which the expression $-c\kappa d+C\kappa^{2m}t$ attains its minimum. We let $\kappa = \tilde c(d/t)^{1/(2m-1)}$ for an appropriate $\tilde c $. Secondly, we observe that up to a constant $\kappa d= \left(d^{2m}/t\right)^{1/(2m-1)}$ and $\kappa^{2m}t=\left(d^{2m}/t\right)^{1/(2m-1)}$, whence $$(\kappa^{l}d^{-s})^{2m/k}t = (\kappa^{2m}t)^{l/k}(d^{2m}/t)^{-s/k}=\left(\frac{d^{2m}}{t}\right)^{\frac{1}{2m-1}\frac{l}{k}-\frac{s}{k}}.$$ 
			
			Taking the restrictions on parameters $s$ and $l$ into account we see $(1/(2m-1)) l/k-s/k\leq0$, which allows us to conclude that whenever $d^{2m}/t \ge1$, then the off-diagonal estimate \eqref{eq:off-diag} is true for some constants $C,\, c>0$. We eventually enlarge $C\ge 1$, such that $C\exp (-c) \ge 1$ holds and so \eqref{eq:off-diag} remains true if $d^{2m}/t <1$. 
	\end{proof}
	 
	 An application of the Riesz--Thorin interpolation between the result from Proposition \ref{prop:off2} and uniform $\mathscr L (L^p(\BBR^n))$ bounds of the propagators gives us the following.
	 \begin{lemma} \label{lem:off3} Let $1\leq q\leq \infty$ and assume \[\sup\limits_{0\leq s\le t <\infty} \|\Gamma(t,s)\|_{\mathscr L (L^q)}<\infty.\] 
	 	\begin{enumerate}[label={\upshape(\roman*)}]
	 	\item If $q\in [1,2)$, then for all $r\in (q,2]$ there exists a constant $\alpha_r>0$ such that for any closed sets $E,\;F\subseteq \BBR^n$ and $0\leq s<t<\infty$ we have for $f\in L^r(\BBR^n)$
	 	\[\|\mathbbm{1}_E\Gamma(t,s)(\mathbbm{1}_F f)\|_{L^2}\lesssim (t-s)^{-\frac{n}{2m}(\frac{1}{r}-\frac{1}{2})}\exp\left\{-\alpha_r\left(\frac{d(E,F)}{(t-s)^{1/(2m)}}\right)^\frac{2m}{2m-1}\right\}\|\mathbbm{1}_F f\|_{L^r}.\]
		\item If $q\in (2,\infty]$, then for all $r\in (2,q]$ there exists a constant $\alpha_r>0$ such that for any closed sets $E,\;F\subseteq \BBR^n$ and $0\leq s<t<\infty$ we have for $f\in L^2(\BBR^n)$ 
	 	\[\|\mathbbm{1}_E\Gamma(t,s)(\mathbbm{1}_F f)\|_{L^r}\lesssim (t-s)^{-\frac{n}{2m}(\frac{1}{2}-\frac{1}{r})}\exp\left\{-\alpha_r\left( \frac{d(E,F)}{(t-s)^{1/(2m)}}\right)^\frac{2m}{2m-1}\right\}\|\mathbbm{1}_F f\|_{L^2}.\]
	 \end{enumerate}
	 \end{lemma}
	 \begin{proof}
	 	This is a consequence of the reversed H\"older estimates from Corollary \ref{cor:rev_hoelder}, bounds from Proposition \ref{prop:off2} and the assumption. See \cite[Lemmas 4.9 and 4.11]{AMP15} for details. 
	 \end{proof}
	 We finish this section addressing inhomogeneous autonomous equations and presenting a type of Duhamel's principle for the propagators. 
	 \begin{lemma}
	 	\label{lem:duhamel}
	 	Let $L_0=(-1)^{m}\text{div}_m \underline{A}(x)\nabla^m$ be an autonomous  elliptic operator. For $f\in L^2(L^2)$, $h\in L^2(\BBR^n)$ and $t>0$ we define
	 	$$u(t,\cdot) \coloneqq e^{-tL_0}h + \mathcal R_{L_0} f(t,\cdot),$$ where $$ \mathcal R_{L_0}f(t,\cdot)\coloneqq\sum_{|\beta|=m} \int_0^t e^{-(t-s)L_0}\partial^\beta f_\beta(s,\cdot)ds$$ is the $L^2(L^2)=T^{2,2}_m\to X^2_m$ bounded map from Proposition \ref{prop:ext_R}. Then $u(0)=h$ in $ L^2(\BBR^n)$, $\nabla^m u\in T^{2,2}_m$ and it holds for all $\phi \in \testfunctions$, \[\langle u,\partial_t \phi\rangle = \langle \underline A \nabla^mu,\nabla^m \phi \rangle + \langle f, \nabla^m \phi \rangle.\]
	 \end{lemma}
	 \begin{proof}
	 	Uniqueness is immediate by Theorem \ref{thm:global_well}, if we look at the difference of two potential solutions. For existence, with Proposition \ref{prop:ext_R} on hand, it is routine to check the claim for $f\in\mathscr D (\BBR^{n+1}_+)$. The quoted result ensures in particular that $\|\nabla^m \mathcal R_{L_0}f \|_{L^2(L^2)}\lesssim \|f\|_{L^2(L^2)}$. For general $f$, we rely on the approximation by test functions. See \cite[Lemma 6.12]{AMP15} for details.
	 \end{proof}
	 \begin{corollary}
	 	\label{cor:duhamel prop} Let $L_0=(-1)^{m}\text{div}_m \underline{A}(x)\nabla^m$ be an autonomous elliptic operator. Then the propagators from \eqref{eq:propagators family} can be represented in $L^2(\BBR^n)$ by 
	 	 \begin{equation*} \label{eq:Duhamel_propagators}
	 	 \Gamma(t,0)h=e^{-tL_0}h
	 	 +\int_0^t e^{-(t-s)L_0}\text{div}_m\left( (A(s,\cdot)-\underline A)\nabla^m \Gamma(s,0)h\right)\, ds,\end{equation*} where $h\in L^2(\BBR^n)$.
	 \end{corollary}
	 \begin{proof} The statement follows directly from Lemma \ref{lem:duhamel}, the decomposition $A= (A-\underline{A}) +\underline A$ and uniqueness of energy solutions from Theorem \ref{thm:global_well}.
	\end{proof}
	 \section{Existence of weak solutions for rough initial data}
	 \label{sec:existence results}
	 \setcounter{theorem}{0} \setcounter{equation}{0}

	\subsection{Initial data with controlled growth in $L^2_{loc}(\BBR^n)$}
	\label{sec:existence big p}
	Thanks to their off-diagonal decay properties, we can use the propagators to construct global weak solutions to \eqref{eq:the parabolic equation} with rough initial data. Our first result leads to a straightforward extension of those operators to Campanato-type $L^p$ spaces from Section \ref{sec:Campanato}.

	\begin{lemma}\label{lem:extension to polynomial growth}
		Let $f\in L^2_{loc} (\BBR^n)$ be such that for any $x_0\in \BBR^n$ there exist $C>0$ and $N\in \BBN$ with $$ \| f\|_{L^2(B(x_0,r)}\leq Cr^N$$ for all $r\ge 1$. Then $\Gamma(t,0)f$ exists in $ \mathscr C([0,\infty); L^2_{loc}(\BBR^n))$. Precisely, fix some  $B_0=B(x_0,r)$ and let $B_k\coloneqq B(x_0,2^kr)$ for $k\in \BBN$. Then for any $K\subseteq \BBR^n$ compact,  
		the limit $$\lim_{k\to\infty }\mathbbm{1}_K \Gamma(t,0) (\mathbbm{1}_{B_k}f)$$ exists in $L^2(\BBR^n)$ for all times $t\in[0,\infty)$, depends continuously on $t\in[0,\infty)$ and is independent of the initial choice of $x_0\in \BBR^n $ and $ r>0$. 
		
		 Moreover,
		the function $u_f\colon (t,x)\mapsto \Gamma(t,0)f(x)$ is a global weak solution of \eqref{eq:the parabolic equation} and satisfies $\lim_{t\to 0}u_f(t,\cdot)=f$ in $ L^2_{loc}(\BBR^n)$.
	\end{lemma} 
	\begin{proof} The claim follows from Lemma \ref{prop:off2} and Proposition \ref{pro:local_energy} by localization. The $L^2_{loc}$ convergence to the trace follows by Lemma \ref{prop:off2} and Lebesgue Dominated Convergence.\qedhere
	\end{proof}
	In particular, for any polynomial $P$ and $f\in L^p(\BBR^n)$ with $p\in [2,\infty]$ or $f\in BMO(\BBR^n)$, a global weak solution of the parabolic equation \eqref{eq:the parabolic equation} with $L^2_{loc}(\BBR^n)$ initial data $f+P$ can be obtained with Lemma \ref{lem:extension to polynomial growth} by $$(t,x)\mapsto \Gamma(t,0)(f+P)(x).$$ 
	
	Our next result is the conservation property for polynomials $P\in \pol$, according to which 
	we can rewrite the above solution in the $L^2_{loc}(\BBR^n)$ sense as 
	$$(t,x)\mapsto \Gamma(t,0)(f)(x)+P(x).$$ 
	\begin{proposition}
		\label{pro:conservation_property_polynome}
		Let $0\leq s\leq t<\infty$ and $P\in\pol$. Then $\Gamma(t,s)P=P $ in $L^2_{loc}(\BBR^n)$.
	\end{proposition}
	\begin{proof}
		Without loss of generality assume $s=0$. For any $0<t<\infty$ and $P\in\pol$, Lemma \ref{lem:extension to polynomial growth} implies that $\Gamma(t,0)P$ is well-defined in $L^2_{loc}(\BBR^n)$. It is easy to see that on any compact set $K\subseteq \BBR^n$ we have 
		$$ \mathbbm{1}_K\Gamma(t,0) P= \lim_{R\to\infty}\mathbbm{1}_K\Gamma(t,0) (\chi_RP) \quad \mbox{in}\quad L^2(\BBR^n),$$ 
	 where $\chi_R\coloneqq\chi(\cdot/R)$ is a smooth cut-off function with $\mathbbm{1}_{B(0,1)}\leq \chi\leq \mathbbm{1}_{B(0,2)}$. 
		
		We aim to show that for any $K\subseteq \BBR^n$ compact
		$$\mathbbm{1}_KP=\lim_{R\to\infty}\mathbbm{1}_K\Gamma(t,0) (\chi_RP) \quad \mbox{in}\quad L^2(\BBR^n).$$
		\begin{claim}
			For any $h\in \mathscr C_c (\BBR^n)$ it holds 
			$$\langle \Gamma(t,0)P, h\rangle_{L^2}=\lim_{R\to\infty}\langle \Gamma(t,0)(\chi_RP), h\rangle_{L^2}=\langle P, h\rangle_{L^2}.$$
		\end{claim} 
		\begin{claimproof}
			\renewcommand{\qedsymbol}{}
			If we denote by $\tilde \Gamma$ is the propagator associated to the matrix $\tilde A (s,x)$ given by $A^*(t-s,x)$ if $s\in(-\infty,t]$ and $A^*(0,x)$ otherwise, then
			$$\langle \Gamma(t,0)(\chi_RP), h\rangle_{L^2}=\langle \chi_RP, \Gamma(t,0)^*h\rangle_{L^2}=\langle \chi_RP, \tilde\Gamma(t,0)h\rangle_{L^2},$$ according to Lemma \ref{lem:backwards_prop} and the definition of the adjoint propagators. By the energy well-posedness from Theorem \ref{thm:global_well} it holds $$\left( s\mapsto u_h(s,\cdot)\coloneqq\tilde\Gamma(t-s,0)h \right) \in W(0,t;H^m(\BBR^n),H^{-m}(\BBR^n))$$ and also $\chi_RP\in  W(0,t;H^m(\BBR^n),H^{-m}(\BBR^n))$ as it is independent of $t$. Therefore, Lemma \ref{lem:absolute continuity} applies and we can write

			\begin{align}\label{eq:estimate 4}
				\langle \chi_RP, u_h(0)\rangle_{L^2} &=\langle \chi_RP,  u_h(t)\rangle_{L^2}		
				-\int_0^t\overline{\langle \partial_s u_h(s), \chi_RP\rangle}_{ H^{-m}(\BBR^n),\sob} ds. 
			\end{align} We next use $ u_h \in W(0,t;H^m(\BBR^n),H^{-m}(\BBR^n))$ to approximate $\chi_RP$ in $L^2(0,t;\sob)$ by a sequence of test functions $(s,x)\mapsto (\eta_\varepsilon(s)\chi_R(x)P(x))_{\varepsilon>0}$ to arrive at
			\begin{align*}
			I\coloneqq	-\int_0^t\overline{\langle \partial_s u_h(s), \chi_RP\rangle}_{ H^{-m}(\BBR^n),\sob} ds
				&=			
				\int_0^t\int_{B(0,2R)}A^*(t-s,x) \nabla^m u_h(s,x) \overline{ \nabla^m (\chi_RP)(x)}dxds.
				\end{align*} 
			 
			\noindent Assume without loss of generality $\text{supp}\, h \subseteq B(0,1)$. Then, denoting the degree of $P$ by $d$,
			\begin{align*}
			|I|	&\lesssim\int_0^t\int_{B(0,8)}|\nabla^m u_h(s,x)| |\nabla^m (\chi_RP)(x)|dxds + \sum_{k=3}^\infty \int_0^t\int_{B_{k+1}\setminus B_k}|\nabla^m u_h(s,x)| |\nabla^m (\chi_RP)(x)|dxds \\
				&\lesssim_{P,\chi} R^{d-m}\int_0^t\int_{B(0,8)}|\nabla^m u_h(s,x)|dxds + R^{d-m}\sum_{k=3}^\infty  \int_0^t\int_{B_{k+1}\setminus B_k}|\nabla^m u_h(s,x)|dxds. 
			\end{align*}
			Here we put $B_k=B(0,2^k)$. Since $d<m$, the first term tends to zero as $R\to \infty$. We need to show the finiteness of $\sum_{k=3}^\infty  \int_0^t\int_{B_{k+1}\setminus B_k}|\nabla^m u_h|dxds$. For this we use the following refined version of the estimates from Proposition \ref{pro:local_energy} (we keep the same notation). 
			 
			 For any $0<r<R$, $0<\zeta<\xi<\min(R-r,r)$ and the annuli $S(r,\zeta)=B(x_0,r+\zeta)\setminus B(x_0,r-\zeta)$ it holds
			 	 \begin{equation}
			 	 \label{eq:new version}
			 	 \int_{a'}^d \|\nabla ^m u(s,\cdot)\|_{L^2(S(r,\zeta))}^2ds \lesssim  \frac{1}{(\xi-\zeta)^{2m}}\int_{a'}^d \| u(s,\cdot)\|_{L^2(S(r,\xi))}^2ds +\|u(a',\cdot)\|_{L^2(S(r,\xi))}^2
			 	 .\end{equation}	
			Estimate \eqref{eq:new version} is achieved by covering the annulus $S(r,\zeta)$ with finitely many slightly overlapping balls and use the already known estimates.

			Back to our setting, \eqref{eq:new version} and off diagonal estimates for the propagators allow to estimate for $k\ge 3$
			\begin{align*}
				\int_0^t\int_{B_{k+1}\setminus B_k}|\nabla^m u_h|dxds
				&\hspace{-0.15cm}\stackrel{\eqref{eq:new version}}{\lesssim} 2^{\frac{kn}{2}} \left(2^{-2m(k-1)}\int_0^t\int_{S(2^k+2^{k-1},2^{k})}| u_h|^2dxds+\|h\|_{L^2(S(2^k+2^{k-1},2^{k}))}^2\right)^{1/2}\\
				&\lesssim 2^{\frac{k(n-2m)}{2}} \exp\left(-c\left(\frac{2^{2m(k-2)}}{t}\right)^{1/(2m-1)}\right) \|h\|_{L^2},
			\end{align*}
			up to constants depending on $t,n,m,\lambda$ and $\Lambda$. Thus, the whole series converges, $$\sum_{k=3}^\infty  \int_0^t\int_{B_{k+1}\setminus B_k}|\nabla^m u_h|dxds<\infty.$$ 
			Passing to the limit as $R\to \infty$ in \eqref{eq:estimate 4} we conclude with $u_h(t)=h$ that
			$$\vspace{-0.8cm}\lim_{R\to\infty}\langle (\chi_RP), u_h(0)\rangle_{L^2}
			=\lim_{R\to\infty}\langle (\chi_RP),  u_h(t)\rangle_{L^2}=\langle P,  h\rangle_{L^2}.$$
		\end{claimproof}
		
	\end{proof}
	
	If $1<p<2$ then $f\in L^p(\BBR^n)$ does not imply $f\in L^2_{loc}(\BBR^n)$, so the results obtained so far do not ensure the existence of weak solutions with such an initial data. Indeed, as we will see, the theory for the cases $1<p<2$ and $2\leq p \leq \infty$ differs a lot and we so present the existence results in separate sections. 
	\begin{proposition}
		\label{pro:p>2 pol existence}
		Let $f\in BMO(\BBR^n)$ or $p\in [2,\infty)$ and $f\in L^p(\BBR^n)$. For any polynomial $P\in \pol$ the function 
		$$\left(t\mapsto u_{f+P}(t,\cdot)\coloneqq\Gamma(t,0)f+P\right) \in L^2_{loc}(0,\infty;H^m_{loc}(\BBR^n))$$
		defines a global weak solution to \eqref{eq:the parabolic equation}, satisfying   $\lim_{t\to 0}u_{f+P}(t,\cdot)=f+P$ in $ L^2_{loc}(\BBR^n)$. It holds $$\nabla^m u_{f+P}=\nabla^m u_{f} \in \tentspace$$ with $p\in [2,\infty)$ if $f\in L^p(\BBR^n)$, or with $p=\infty$ if $f\in BMO(\BBR^n)$. Further, we have
		\[\|f\|_{Y}\sim \|\nabla^m u_{f}\|_{\tentspace}\] with $Y=L^p(\BBR^n)$ or $Y=BMO(\BBR^n)$ correspondingly.
	\end{proposition}
	\begin{proof}

		Clearly, in both of the considered cases, Lemma \ref{lem:extension to polynomial growth} combined with Proposition \ref{pro:conservation_property_polynome} ensures the validity of the first half of our claim, that is for any polynomial $P\in \pol$, $$u_{f+P}\coloneqq\Gamma(t,0)f+P$$ is a global weak solution to \eqref{eq:the parabolic equation} with $L^2_{loc}(\BBR^n)$ trace at $t=0$ given by $f+P$. 
		
		Note first that $\nabla^m P=0$, so we indeed have $\nabla^m u_{f+P}=\nabla^m u_{f}$. 
	
		Recalling the uniqueness statements from Propositions \ref{prop:structure of BMO_m} and \ref{prop:structure of L^p_m} we only need to show 
		\[\|f^{\#,m}\|_{L^p}\sim \|\nabla^m u_{f}\|_{\tentspace}\] for any $p\in (2,\infty]$, whenever the left hand side is finite. We prove here the bound 
		\begin{equation*}
		\|\nabla^m u_{f}\|_{\tentspace}\lesssim \|f^{\#,m}\|_{L^p}, 
		\end{equation*} as the converse inequality has been proven in Lemma \ref{lem:trace p>2}, without assuming that $u_f$ is given by the propagator.
		With the equivalent norm on $\tentspace$ from Remark \ref{rem:eqv norm tent space}, it will be enough to show that for any ball $B(x_0,r)$ it holds
		\begin{equation}
			\label{eq:estimate 5}
			\left(\int_0^{r^{2m}} \fint _{B(x_0,r)}|\nabla^m u_{f}|^2 dy dt\right)^{1/2} \lesssim C_{x_0,r}(f) \coloneqq\sup_{R\ge r} \inf_{p\in \mathcal P_{m-1}}\left( \fint_{B(x_0,R)}\left|f(y)-p(y)\right|^2dy\right)^{1/2}.
		\end{equation} 
		 
		 Our proof of \eqref{eq:estimate 5} follows the idea of the classical Carleson measure estimates quoted in Proposition \ref{prop:carleson} (i). 
		 Let $B=B(x_0,r)$ and $2B=B(x_0,2r)$. Consider $Q_{2B}=\mathbb{P}_{x_0,2r}(f)$, the minimizing polynomial for $f$ on $2B$. We introduce the decomposition
		 $$f=f_1+f_2+Q_{2B},$$ where $$f_1=\mathbbm{1}_{2B}(f-Q_{2B})\quad \mbox{and}\quad f_2=\mathbbm{1}_{\BBR^n\setminus{2B}}(f-Q_{2B}).$$ Due to the conservation property from Proposition \ref{pro:conservation_property_polynome} it is $\nabla^m \Gamma(t,0)(Q_{2B})=\nabla^m(Q_{2B})=0$ in $L^2_{loc}(\BBR^n)$ and we only need to estimate the parts corresponding to $f_1$ and $f_2$. To this end note that $f_1\in L^2(\BBR^n)$ implies
		 \[|B|^{-1/2}\|\nabla^m \Gamma(t,0)f_1\|_{L^2(L^2)}\lesssim|B|^{-1/2}\|f_1\|_{L^2}\lesssim C_{x_0,r}(f)\] with constants depending only on the ellipticity and the dimension.
		 
		 For $f_2$, let us first recall an inequality from the proof of Proposition \ref{pro:local_energy} (here with the parameters $b=r^{2m}$, $0<a'<b$, $ R=2r$). For any $h\in L^2(\BBR^n)$ supported in $\BBR^n \setminus 2B$ it holds
		 \begin{align*}
		 \int_{a'}^b \|\nabla ^m \Gamma(s,0)h\|_{L^2(B)}^2ds \lesssim \int_{a'}^b\frac{1}{r^{2m}} \| \Gamma(s,0)h\|_{L^2(2B\setminus B)}^2ds +\|\Gamma(a',0)h\|_{L^2(2B)}^2.
		 \end{align*}
		 The continuity in time of energy solutions implies $\lim_{a'\to 0}\|\Gamma(a',0)h\|_{L^2(2B)}=\|h\|_{L^2(2B)}=0$. We can take the limit as $a'$ tends to zero on both sides of the above inequality to obtain 
		 \begin{align*}\label{eq:estim_norm3}
		 \int_{0}^{r^{2m}} \|\nabla ^m \Gamma(s,0)h\|_{L^2(B)}^2ds \lesssim \fint_{0}^{r^{2m}} \| \Gamma(s,0)h\|_{L^2(2B)}^2ds.
		 \end{align*}
		 Setting $h=f_2$ and $B_k\coloneqq B(x_0,2^kr)$ we thus have
		 \begin{align*}
		 \left(\int_0^{r^{2m}}\fint_B|\nabla^m \Gamma(s,0)f_2(y)|^2dyds\right)^{1/2}
		 &\lesssim  \sum_{k\ge1}\left(\fint_0^{r^{2m}}\fint_{2B}|\Gamma(s,0)(f_2\mathbbm{1}_{B_{k+1}\setminus B_{k}})(y)|^2dyds\right)^{1/2}.
		 \end{align*} For $k\in \BBN$ and $s\in(0,r^{2m})$, we estimate with the off-diagonal bounds from Proposition \ref{prop:off2} 
		  \begin{align*}
		 \fint_B|\nabla^m \Gamma(s,0)(f_2\mathbbm{1}_{B_{k+1}\setminus B_{k}})|^2dy
		  &\lesssim 
		  \exp\left(-c\left(\frac{(2^kr)^{2m}}{s}\right)^{1/(2m-1)}\right)|B|^{-1}\|f-Q_{2B}\|_{L^2(B_{k+1})}^2,
		 \end{align*}
		 where we used that $d(B_{k+1}\setminus B_{k},B)=2^kr-r\ge 2^{k-1}r$. Moreover, we have 
		 \begin{align*}
		 |2B|^{-1}\|f-Q_{2B}\|_{L^2(B_{k+1})}^2&\leq 2^{nk}C_{x_0,r}(f)^2 +|2B|^{-1}\|\mathbb{P}_{x_0,2^{k+1}r}(f)-Q_{2B}\|_{L^2(B_{k+1})}^2.
		 \end{align*} Since $\mathbb{P}_{x_0,2^{k+1}r}(f)$ stays invariant under the projection $\mathbb{P}_{x_0,2r}$, we estimate for $y\in B_{k+1}$ 
		  \begin{align*}
		  |\mathbb{P}_{x_0,2^{k+1}r}(f)-Q_{2B}|(y)
		 & \lesssim_m 2^{k(n+m)}C_{x_0,r}(f)
		  \end{align*} with constants independent on $r$. So, there exists $N\in \BBN$ with
		  \begin{align*}
		  |2B|^{-1}\|f-Q_{2B}\|_{L^2(B_{k+1})}^2&\lesssim 2^{Nk}C_{x_0,r}(f)^2.
		  \end{align*}
		  Finally,
		 \begin{align*}
		 \left(\int_0^{r^{2m}}\fint_B|\nabla^m \Gamma(s,0)f_2(y)|^2dyds\right)^{1/2} 
		 &\lesssim 
		 \sum_{k\ge1}\left(\fint_0^{r^{2m}}\exp\left(-c2^{2mk/(2m-1)}\right)2^{Nk}C_{x_0,r}(f)^2ds\right)^{1/2}\\
		 &\lesssim 
		 C_{x_0,r}(f),
		 \end{align*} where we used $s\in (0,r^{2m})$ to dispense with the ratio $(r^{2m}/s)$ in the exponential. This gives \eqref{eq:estimate 5}.
	\end{proof}
	
	The following proposition contains the higher order counterpart of the existence results \cite[Corollaries 5.5 and 7.2]{AMP15} for the Cauchy problem \eqref{eq:Cauchy problem} for $Y=L^p(\BBR^n)$ with $p\in [2,\infty]$ and spaces $X=X_m^p$ and $X=L^\infty(0,\infty;L^p(\BBR^n))$.
	\begin{proposition}
			\label{pro:p>2 existence}
		Let $p\in [2,\infty]$ and $f\in L^p(\BBR^n)$. Consider the solution $u_f(t,\cdot)=\Gamma(t,0)f$ to \eqref{eq:the parabolic equation} obtained in Proposition \ref{pro:p>2 pol existence}. Then the following are true.
		\begin{enumerate}[label={\upshape(\roman*)}]
			\item $u_{f}\in X^p_m(\BBR^n)$ and 
			$\|f\|_{L^p}\sim \|u_f\|_{X^p_m}.$
			\item Under the uniform boundedness assumption 
			\begin{equation} \tag{UBC[p]}\label{ubc1} \sup\limits_{0\leq s\le t <\infty}  \|\Gamma(t,s)\|_{\mathscr L (L^p)}<\infty,\end{equation}  it holds $u_f \in  L^\infty(0,\infty;L^p(\BBR^n))$ and $\|f\|_{L^p}\sim \|u_f\|_{L^\infty(L^p)}.$
		\end{enumerate}
	\end{proposition} 
	
	Before we proceed to the proof, we point out that the condition in (ii) is necessary, in the sense that we have for any $p\in[1,\infty]$ $$\esssup\limits_{0\leq s\le t <\infty}  \|\Gamma(t,s)\|_{\mathscr L (L^p)}\leq C<\infty \Longrightarrow \sup\limits_{0\leq s\le t <\infty}  \|\Gamma(t,s)\|_{\mathscr L (L^p)}\leq C<\infty. $$ 

	We argue as follows. Suppose $p\in (1,\infty)$ and let $f,\,g \in \mathscr D(\BBR^n)\subseteq L^2(\BBR^n)$. With $q\in(1,\infty)$ being the dual exponent to $p$, it holds for almost every $0\leq s\leq t<\infty$ by assumption 
	\[\big| \langle \Gamma(t,s)f,g \rangle  \big|\leq C \|f\|_{L^p}\|g\|_{L^{q}}.\] Recall from \eqref{eq:continuity of prop} and Lemma \ref{lem:backwards_prop} that the maps $$[s,\infty)\ni t\mapsto \langle \Gamma(t,s)f,g \rangle \quad \mbox{and}\quad [0,t]\ni s\mapsto \langle \Gamma(t,s)f,g \rangle$$ are continuous. Hence, the previous bound holds for every $0\leq s\leq t<\infty$ and so, by density, the propagator $\Gamma(t,s)\in \mathscr L(L^2)$ admits a unique continuous extension to $L^p(\BBR^n)$. In addition, it holds uniformly in $0\leq s\leq t<\infty$ that $$ \|\Gamma(t,s)\|_{\mathscr L (L^p)}\leq C.$$ 
	
	For $p=1$ the same reasoning works if we consider $f\in \mathscr D(\BBR^n)$ and $g\in L^\infty(\BBR^n)$ with compact support and for $p=\infty$ we just reverse the roles of $f$ and $g$. 
	\begin{proof} Clearly, in (ii) there is nothing to prove. Claim (i) can be seen directly if $p\in(2,\infty]$. Indeed, let $u_f(t,\cdot)=\Gamma(t,0)f$ with $f\in L^p(\BBR^n)$ and $p\in(2,\infty]$. For any $x \in \BBR^n$ and $r>0$ introduce the annuli $S_k(x,r)=B(x,2^{k+1}r)\setminus B(x,2^kr)$ for $k\ge 1$ and $S_0(x,r)= B(x,2r)$. 
		
	Then, for any $\delta>0$, we decompose $$f=f\mathbbm{1}_{S_0(x,\sqrt[2m]{\delta})}+\sum_{k\ge 1}f\mathbbm{1}_{S_k(x,\sqrt[2m]{\delta})}.$$ As we have already seen, the $L^2$ off-diagonal bounds allow to estimate with some $N_{n,m}\in \BBN$
		\begin{align*}
		\left(\fint_{\delta/2}^\delta \fint_{B(x,\sqrt[2m]{\delta})}|\Gamma(t,0)f(y)|^2dydt\right)^{1/2}
		&\lesssim\sum_{k\ge 0} 2^{kN}\exp\left(-c2^\frac{2mk}{2m-1}\right) \left( \fint_{B(x,2^{k+1}\sqrt[2m]{\delta})}|f(y)|^2dy\right)^{1/2}\\&\lesssim 
		\left(\mathcal M_{HL}|f|^2\right)^{1/2}(x). 
		\end{align*}
	
	Thus, since $p>2$, we conclude $\nontan (u_f)\in L^p(\BBR^n)$ with 
	\[\|u_f\|_{X^p_m}=\|\nontan (u_f)\|_{L^p}\lesssim \|f\|_{L^p}.\] In particular it holds for $p\in(2,\infty)$ (we exclude $p=\infty$, for which $ \|\nabla^m u_f\|_{\tentspace}\sim \|f\|_{BMO}$)
	\begin{equation}
	\label{eq:non tan conical estimate}
	\|u_f\|_{X^p_m}\lesssim  \|\nabla^m u_f\|_{\tentspace}.
	\end{equation}The reversed inequality 
	\[\|f\|_{L^p}\lesssim \|u_f\|_{X^p_m}\] can be also proven directly by the off-diagonal bounds combined with a Fatou-type argument, see \cite[Lemma 4.6 (ii)]{AMP15}. This then implies for $p\in(2,\infty]$
		\begin{equation}
		\label{eq:conical non tan estimate}
		\|\nabla^m u_f\|_{\tentspace}\lesssim\|u_f\|_{X^p_m}.
		\end{equation}
	 However, both arguments break down if $p=2$, because they require the $L^{p/2}$ boundedness of the maximal function. The estimates \eqref{eq:non tan conical estimate} and \eqref{eq:conical non tan estimate} remain true though, even if $p=2$ and we state the results in Section \ref{sec:comparability results}.
	\end{proof}

	\subsection{Initial data in $L^p(\BBR^n)$ with $p<2$}
	\label{sec:existence small p}
	\noindent Lemma \ref{pro:p>2 pol existence} does not cover $L^p(\BBR^n)$ initial data if $p<2$. In this case even the existence of solutions in the non-tangential space requires the boundedness assumption on the propagators. For the second order case of the results in this section, see \cite[Corollary 5.10]{AMP15}.  
	\begin{lemma}\label{lem:x^p implies L^p}
		Let $p\in(1,2]$ and $u\in X^p_m$ be a global weak solution to \eqref{eq:the parabolic equation}. Then $u\in L^\infty(0,\infty;L^p)$ and 
		\[\|u\|_{L^\infty(L^p)}\lesssim \|u\|_{X^p_m}.\]
	\end{lemma}
	\begin{proof} Assume  $u\in X^p_m$ is a global weak solution to \eqref{eq:the parabolic equation}. For $t>0$, Fubini's Theorem and H\"older's inequality for $p<2$ give
	\begin{align*}
	\|u(t,\cdot)\|_{L^p}^p
	&\leq \int_{\BBR^n}\left(\fint_{B(y,\sqrt[2m]{t}/2)}|u(t,x)|^2dx\right)^{p/2}dy.
	\end{align*}
	
	We estimate with help of the a priori energy bounds from Proposition \ref{pro:local_energy}
	\[\left(\fint_{B(y,\sqrt[2m]{t}/2)}|u(t,x)|^2dx\right)^{1/2}\lesssim \left(\fint_{t/2}^t\fint_{B(y,\sqrt[2m]{t})}|u(t,x)|^2dx\right)^{1/2}.\]
	Thus, $	\|u(t,\cdot)\|_{L^p}^p\lesssim \|\nontan u \|_{L^p}^p$ holds for every $t>0$. 
	\end{proof}
	Let us observe the following. If \eqref{eq:Cauchy problem} is well-posed for $(L^p(\BBR^n),X^p_m)$, then there exists a continuous solution map $u_f\colon L^p(\BBR^n)\ni f \mapsto u_f\in X^p_m$, which by Lemma \ref{lem:x^p implies L^p} 
 	satisfies \[\|u_f\|_{L^\infty(L^p)}\lesssim \|f\|_{L^p(\BBR^n)}.\]
 	From such $L^\infty(L^p)$ bounds we are able to deduce that the evolution of $u_f$ at positive times must be governed by the propagators, see Corollary \ref{cor:estimate_for_uniqueness} in the next section. Combined with the estimate above, this gives a heuristical explanation, why it is meaningful for $p<2$ to assume a uniform boundedness condition for the propagators. We introduce the uniform boundedness condition as follows
	\begin{equation} \tag{UBC[p]}\label{ubc}\sup\limits_{0\leq s\le t <\infty} \sup\limits_{\substack{h\in \testfunctions, \\  \|h\|_{L^p(\BBR^n)}\leq 1}} \|\Gamma(t,s)h\|_{L^p(\BBR^n)}<\infty.\end{equation} The same reasoning applies to the well-posedness in $X=\{u\in \mathscr D'(\BBR^{n+1})\mid \nabla^m u \in \tentspace\}$, see Lemma  \ref{lem:trace integrability}. 
	
	The following result is the higher order analogue of \cite[Lemma 4.10]{AMP15}.
	\begin{lemma} \label{lem:wp_small}
		Let $p\in [1,2)$ and assume \ref{ubc} holds. Then for all $f\in L^p(\BBR^n)$ the function $u_f\colon (t,x)\mapsto \Gamma(t,0)f$ is a global weak solution of \eqref{eq:the parabolic equation} and $u_f\in   L^\infty(0,\infty;L^p)$.
	
		Moreover, for all $r\in (p,2)$, it holds 
		\[\|u_f\|_{L^\infty(L^r)}\lesssim\|u_f\|_{X^r_m}\lesssim \|f\|_{L^r}.\]
	\end{lemma}
	\begin{proof}
		Let $f\in L^p(\BBR^n)$. By \ref{ubc}, $u_f $ is well-defined in $L^\infty(0,\infty;L^p(\BBR^n))$. We first need to show $u_f\in  L^2_{loc}(0,\infty;H^m_{loc}(\BBR^n))$. 
		Suppose first $f\in L^2(\BBR^n)\cap L^p(\BBR^n)$. We then have for any $\tau>0$ and $x\in \BBR^n$ by Proposition \ref{pro:local_energy}
		\[\|\Gamma(\tau,0)f\|_{L^2(B(x,\sqrt[2m]{\tau}/2))}^2\lesssim \fint_{\tau/2}^\tau \fint_{B(x,\sqrt[2m]{\tau})} |u_f(s,y)|^2dyds.\]
		
		By Vitali's Covering Lemma, there is a finite number $K_{n,m}$ such that for any $\tau>0$ there exist points $\{\nu_i\mid i\in \BBN, \,1\leq i\leq K \}\in B(0,\sqrt[2m]{\tau})$, such that 
		\[B(x,\sqrt[2m]{\tau})\subseteq \bigcup_{i=1}^KB(x+\nu_i,\sqrt[2m]{\tau/2^{4m+1}}).\]
		The rescaling factor is chosen so that the new radii $r=\sqrt[2m]{\tau/2^{4m+1}}$ satisfy $(4r)^{2m}\leq \tau/2$. Hence, for any $k=0,\dots,4m-1$ and $t_k=\tau/2+2^kr^{2m}$ the balls $\{B_r(t,x+\nu_i)\}_{1\leq i\leq K}  $ satisfy the assumptions of the reversed H\"older estimates of Corollary \ref{cor:rev_hoelder}. Hence, we obtain
		\[\|\Gamma(\tau,0)f\|_{L^2(B(x,\sqrt[2m]{\tau}/2))}\lesssim \sum_{i=1}^K\sum_{k=0}^{4m-1} \left(\fint_{B_{4r}(t_k,x+\nu_i)} |u_f(s,y)|^pdyds\right)^{1/p}\lesssim_{\tau} \|f\|_{L^p}.\]
		
		Thus, by density of $L^2(\BBR^n)\cap L^p(\BBR^n)$ in $ L^p(\BBR^n)$ with respect to the $L^p$-norm, $\Gamma(t,0)f$ can be defined for every $f\in L^p(\BBR^n)$ and any $t>0$ as an object in $L^2_{loc}(\BBR^n)$ and satisfies 
		\[\|\mathbbm{1}_{B(0,R)}\Gamma(t,0)f\|_{L^2}\lesssim_{R,t}\|f\|_{L^p}\] for any $R>0$, $t>0$. From here, a routine application of Proposition \ref{pro:local_energy} shows 
		\begin{equation}\label{eq:bound der}\|\nabla^k \Gamma(t,0)f\|_{L^2(K)}\lesssim_{K} \|f\|_{L^p}\end{equation}
		on any compact $K\subseteq \BBR^{n+1}_+$ and integer $0\leq k\leq m$. 
		\noindent We now show that $u_f$ satisfies \eqref{eq:the parabolic equation} in the sense of distributions on $\BBR^{n+1}_+$. Let $\varepsilon>0$ and $f_\varepsilon\in L^p(\BBR^n)\cap L^2(\BBR^n)$ with $\|f-f_\varepsilon\|_{L^p} <\varepsilon$. Then $u_{f_\varepsilon}$ is a global solution of \eqref{eq:the parabolic equation} and we estimate for $\phi\in\mathscr D(\BBR^{n+1}_+)$
		\begin{align*}
		&\left| - \int_0^\infty \int_{\BBR^n}u_f\overline{\partial _t \phi}dydt + \int_0^\infty \int _{\BBR^n} A\nabla^m u_f \overline {\nabla^m \phi}\right|dydt
		\\&\leq \int_0^\infty \int_{\BBR^n}|u_f-u_{f_\varepsilon}||\partial _t \phi|dydt + \Lambda\int_0^\infty \int _{\BBR^n} |\nabla^m (u_f -u_{f_\varepsilon})||\nabla^m \phi|dydt
		\\& \lesssim\|u_f-u_{f_\varepsilon}\|_{L^\infty(L^p)}+\|\nabla^m (u_f-u_{f_\varepsilon})\|_{L^2(\text{supp}(\nabla^m \phi))} \stackrel{\eqref{eq:bound der}}{\lesssim}\|f-f_\varepsilon\|_{L^p}<\varepsilon.
		\end{align*}
		Finally, let $r\in(p,2)$. For $x\in \BBR^n$, $\delta>0$ and $\kappa \in(p,r)$, we have with Lemma \ref{lem:off3} and the decomposition $f=f\mathbbm{1}_{S_0(x,\sqrt[2m]{\delta})}+\sum_{k\ge 1}f\mathbbm{1}_{S_k(x,\sqrt[2m]{\delta})}$, for some $N=N_{n,m}\in \BBN$
		\begin{align*}
		\left(\fint_{\delta/2}^\delta \fint_{B(x,\sqrt[2m]{\delta})}|\Gamma(t,0)f(y)|^2dydt\right)^{1/2}
		&\lesssim\sum_{k\ge 0} 2^{kN}\exp\left(-c2^\frac{2mk}{2m-1}\right) \left( \fint_{B(x,2^{k+1}\sqrt[2m]{\delta})}|f(y)|^\kappa dy\right)^{1/\kappa}\\&\lesssim 
		\left(\mathcal M_{HL}|f|^\kappa\right)^{1/\kappa}(x). 
		\end{align*} By the $L^{r/\kappa}(\BBR^n)$ boundedness of the Hardy--Littlewood maximal function, we conclude 
		\[\|u_f\|_{X^r_m}\lesssim \|f\|_{L^r}.\qedhere\] 
	\end{proof}
		\subsection{Comparability of $\|\nabla^m u\|_{\tentspace}$ and $\|u\|_{X^p_m}$ - results}
		\label{sec:comparability results}
		For reference, we collect here statements on the validity of 
		\begin{equation}
		\label{eq:comparability}
		\|u\|_{X^p_m}\sim \|\nabla^m u\|_{\tentspace}
		\end{equation} for global weak solutions $u$ of \eqref{eq:the parabolic equation}. All proofs were moved to the Appendix (cf.\ Section \ref{sec:comparability proofs}) being technical adaptations of the estimates from \cite[\S7]{AMP15} to the higher order case. 
		\begin{proposition}
			\label{prop:tent_kenig_comp_smallp}
			Let $p\in (\frac{n}{n+m},\infty)$ and $f\in L^2(\BBR^n)$. Suppose $u_f(t,\cdot)=\Gamma(t,0)f$ is such that $\nabla^m u_f \in \tentspace$. Then $u\in X^p_m$ and
			\[ \|u_f\|_{X^p_m}\lesssim\|\nabla^m u_f \|_{\tentspace}.\]
		\end{proposition}
		
		The proof of Proposition \ref{prop:tent_kenig_comp_smallp} requires three ingredients. Namely, we begin with the representation formula from Corollary \ref{cor:duhamel prop}, argue that the claim is true for $L_0=(-1)^m\Delta^m$ and use the properties of the integral operator $\mathcal R_{L_0}$ from Proposition \ref{prop:ext_R}.
		
		The reversed inequality holds without the assumption on the form of the solution. For $p<2$ our methods require the stronger ellipticity assumption \eqref{eq:strong lower ellipticity bounds}.
		\begin{proposition}
			\label{prop:comp_tent_kenig_smallp}
			Assume $p\in [1,2)$ and that $L$ satisfies the strong ellipticity bounds \eqref{eq:strong lower ellipticity bounds}. If $u\in X^p_m$ is a global weak solution to \eqref{eq:the parabolic equation}, then $\nabla^m u \in \tentspace$ and 
			\[ \|\nabla^m u \|_{\tentspace}\lesssim \|u\|_{X^p_m}.\]
		\end{proposition}
		\begin{proposition}
			\label{prop:comp_tent_kenig_bigp}
			Assume that $2\leq p\leq \infty$ and $u\in X^p_m$ is a global weak solution to \eqref{eq:the parabolic equation}. Then 
			\[ \|\nabla^m u \|_{\tentspace}\lesssim \|u\|_{X^p_m}.\]
		\end{proposition}
		All of the above estimates were known at least for the second order case and some of them for the higher order autonomous case. In the next section, we contribute further bounds, namely
		\begin{enumerate}[label={\upshape(\roman{*})}]
			\item We can treat arbitrary solutions satisfying $\nabla^m u \in \tentspace$ (and not only those given by the propagators applied to an $L^p(\BBR^n)$ function as in Proposition \ref{prop:tent_kenig_comp_smallp}) and obtain 
			\[\|u_0-P\|_{L^p}\sim \|u-P\|_{X^p_m}\sim\|\nabla^m u \|_{\tentspace}\] if $p\in[2,\infty)$ and 
			\[\|u_0-P\|_{L^p}\lesssim \|u-P\|_{L^\infty(L^p)}\lesssim\|\nabla^m u \|_{\tentspace}\] if $p\in(1,2]$, where the polynomial $P\in\pol$ is unique. 
			\item Under the strong ellipticity assumption \eqref{eq:strong lower ellipticity bounds}, if $p\in (1,2)$ and \ref{ubc} holds, then 	\[\|u_0-P\|_{L^q}\sim \|u-P\|_{L^\infty(L^q)}\sim\|u-P\|_{X^q_m}\sim\|\nabla^m u \|_{T^{q,2}_m}\] is true for all $q\in(p,2)$ and a unique polynomial $P\in\pol$. 
		\end{enumerate}
	\section{Uniqueness results and tent space well-posedness}
	\label{sec:uniqueness results}
	\setcounter{theorem}{0} \setcounter{equation}{0}
	\subsection{Interior uniqueness}
	The following local representation result is the main step towards the uniqueness results and is based on \cite[Theorem 5.1]{AMP15}.
	
	\begin{theorem} \label{thm:int_repr}
		Let $u$ be a local weak solution of \eqref{eq:the parabolic equation} on $(a,b)\times\BBR^n$ and $c>0$ be the constant from \eqref{eq:off-diag}. Assume that for some $\gamma<\frac{c}{2^{4m/(2m-1)}(b-a)^{1/(2m-1)}}$ it holds
		\begin{equation}
		\label{eq:ass_int_repr}
		 M\coloneqq \int_{\BBR^n}\left(\int_a^b\int_{B(x,\sqrt[2m]{b})} |u(t,y)|^2 dy dt \right) ^{1/2}e^{-\gamma|x|^{2m/(2m-1)}}dx< \infty.\end{equation} 
		Then $u(t,\cdot)=\Gamma(t,s)u(s,\cdot)$ for $a<s\leq t <b$ in the sense of 
		\begin{equation}\label{eq:int_repr}
		\int_{\BBR^n}u(s,x)\overline{\Gamma(t,s)^*h(x)}dx=\int_{\BBR^n}u(t,x)\overline{h(x)}dx \;\;\; \text{whenever } h\in \mathscr C _c (\BBR^n).\end{equation}
	\end{theorem}
	\begin{proof}
			For the proof in the second order case, see \cite[Theorem 5.1]{AMP15}. Inspection of the proof reveals that with the local energy estimates from Proposition \ref{pro:local_energy}, the same argumentation applies in the higher order setting.
	\end{proof}
	\begin{remark}
		\label{rem:equivalent 5.1}
		In the setting of Theorem \ref{thm:int_repr} it holds with $\sigma=\frac{2m}{2m-1}$
		\[\int_a^b \int_{\BBR^n}|u(t,y)|e^{- 2^{\sigma-1}\gamma|y|^{2m/(2m-1)}}dydt\lesssim_{b,\gamma} M\lesssim_{b,\gamma} \int_a^b \int_{\BBR^n}|u(t,y)|e^{-2^{1-\sigma} \gamma|y|^{2m/(2m-1)}}dydt.\]
	\end{remark}
	\noindent This follows from the inequality 
	\[|a+b|^\sigma\leq 2^{\sigma-1} (|a|^\sigma +|b|^\sigma) \quad\mbox{for all}\quad a,\,b\in \BBR\] and, for the first estimate, the H\"older inequality or a covering argument combined with Corollary \ref{cor:rev_hoelder} for the second estimate. We carried out this argument in the proof of Lemma \ref{lem:wp_small}. 
	
	Let us present a simple application of Theorem \ref{thm:int_repr} 
	\begin{corollary} \label{cor:estimate_for_uniqueness}
			Let $p\in [1,\infty]$ and $u$ be a global weak solution of \eqref{eq:the parabolic equation} with $u\in L^\infty(L^p)$ or $u\in X_{m}^p$. Then for any $0<s<t<\infty$ 
			\[ u(t,\cdot)=\Gamma(t,s)u(s,\cdot)\] holds in the sense of \eqref{eq:int_repr}. 
	\end{corollary}	
	\begin{proof}
		This is an immediate consequence of $e^{-c|x|^{2m/(2m-1)}}\in L^{p'}(\BBR^n)$ for any $c>0$ and $1\leq p'\leq\infty$ given by $\frac{1}{p}+\frac{1}{p'}=1$, a covering argument and Remark \ref{rem:equivalent 5.1}.
	\end{proof}
	We make the following observation. The interior representation result considers only \textit{interior} times and states how the solutions propagate for such times. It is a separate task to identify the trace of the solution and show that the propagation formula can be extended to $s=0$. Let us outline the strategy for proving well-posedness of the Cauchy problem \eqref{eq:Cauchy problem} for $Y=L^p(\BBR^n)$ with $1\leq p\leq \infty$ and a seminormed space $(X,\|\cdot\|_X)$. 
	
	\textbf{Step 1.} Start with  global weak solution $u$ to \eqref{eq:the parabolic equation} and prove the existence of the $L^q_{loc}(\BBR^n)$ trace $u_0$ at $t=0$ for some $q\in[1,\infty)$. For this, use the bound $\|u\|_X<\infty$ and the equation. 
	
	\textbf{Step 2.} Use the control $\|u\|_X<\infty$ to check the assumption \eqref{eq:ass_int_repr} and conclude for any time $t>0$ and sequence of times $(t_n)_{n\in \BBN}\subseteq (0,t)$ with $t_n\to 0$
	the representation
		\begin{equation}\label{eq:5}
		\int_{\BBR^n}u(t_n,x)\overline{\Gamma(t,t_n)^*h(x)}dx=\int_{\BBR^n}u(t,x)\overline{h(x)}dx \;\;\; \text{whenever } h\in \mathscr C _c (\BBR^n).\end{equation}
	
	\textbf{Step 3.} Proceed to the limit as $n$ tends to infinity in \eqref{eq:5} rigorously using the convergence from Step 1. and the properties of the propagators. This gives $u(t,\cdot)=\Gamma(t,0)u_0(\cdot)$.
	
	\textbf{Step 4.} Prove that $u_0\in Y$ and that the unique solution obtained depends continuously on its trace. This amounts to showing 
	\[\|u_0\|_Y\sim \|u\|_X.\]
	\begin{example}
		\label{ex:example wp}
			Let $ X=L^\infty(L^2)$ and $u\in X$ be a global weak solution to \eqref{eq:the parabolic equation}. To obtain the unique representation of $u$ by the propagators, we follow the strategy above. By the continuity results of Proposition \ref{pro:local_energy}
			\begin{equation}\label{eq:eq6}\sup_{t>0}\|u(t,\cdot)\|_{L^2}<\infty.\end{equation}
			Step 1. By weak compactness of $L^2(\BBR^n)$ and \eqref{eq:eq6}, there is a sequence $(t_n)_{n\in \BBN}\subseteq (0,1)$ with $t_n\to 0$ and $u_0\in L^2(\BBR^n)$ with 
			\[u(t_n,\cdot)\rightharpoonup u_0 \quad \mbox{in } L^2(\BBR^n) \mbox{ as } n\to \infty.\]
			Step 2. See Corollary \ref{cor:estimate_for_uniqueness}.
			\newline Step 3. Given \eqref{eq:5}, use the weak convergence from the Step 1 and strong $L^2(\BBR^n)$ continuity of $[0,t]\ni t\mapsto \Gamma(t,t_n)^*h$ from Lemma \ref{lem:backwards_prop}. We obtain $u(t,\cdot)=\Gamma(t,0)u_0$.
			\newline Step 4. Follows immediately from Theorem \ref{thm:global_well}. At the same time this step shows the uniqueness of the trace. 
	\end{example}
	
	Summarizing, for $p=2$ we obtain the following well-posedness results.
	\begin{theorem}\label{thm:full l2}
		For a distribution $u\in \mathscr D'(\BBR^{n+1}_+)$ it is equivalent
		\begin{enumerate}[label={\upshape(\roman*)}]
			\item $u$ is a global weak solution of \eqref{eq:the parabolic equation} and $\nabla^m u\in L^2(L^2)$.
			\item $u$ is a global weak solution of \eqref{eq:the parabolic equation} and $u-P\in L^\infty(L^2)$ for some polynomial $P\in\pol$.
			\item $u$ is a global weak solution of \eqref{eq:the parabolic equation} and $u-P\in X^2_m$ for some polynomial $P\in\pol$.
			\item There are unique $ f\in L^2(\BBR^n)$ and $P\in\pol$, so that $u(t,\cdot)-P=\Gamma(t,0)f$ in $L^2(\BBR^n)$ for $t>0$.
		\end{enumerate}
		In this case, $u-P$ is the energy solution with trace $f$ obtained in Theorem \ref{thm:global_well} and satisfies therein stated bounds. Moreover, 
		$
		\|\nabla^m u\|_{T^{2,2}_m}\sim\|u-P\|_{X^2_m}\sim\|u-P\|_{L^\infty(L^2)}\sim\|f\|_{L^2(\BBR^n)}.
		$
	\end{theorem} 
	\begin{proof}
		Points (i), (ii) an (iv) are equivalent by Lemma \ref{lem:trace integrability}, Example \ref{ex:example wp} and Theorem \ref{thm:global_well}. Lemma \ref{lem:x^p implies L^p} proves the implication (iii) to (ii) and (ii) follows from (iv) by Proposition \ref{prop:tent_kenig_comp_smallp}.
	\end{proof}
	\subsection{Tent space solutions $2<p\leq\infty$}
	We now approach the tent space well-posedness for $p>2$. The next result is of course true for $p=2$, if we let $Y=L^2(\BBR^n)+\mathcal P_{m-1}$ and set $\|f+P\|_Y\coloneqq \|f\|_{L^2}$, see Theorem \ref{thm:global_well} or Theorem \ref{thm:full l2}.
	\begin{theorem} \label{thm:wp pol}
		Let $p\in(2,\infty]$. For a distribution $u\in \mathscr D'(\BBR^{n+1}_+)$ it is equivalent
		\begin{enumerate}[label={\upshape(\roman*)}]
			\item $u$ is a global weak solution of \eqref{eq:the parabolic equation} and $\nabla^mu\in \tentspace$.
		
			\item There is a unique $ f\in Y$ such that $u(t,\cdot)=\Gamma(t,0)f$ in $L^2_{loc}(\BBR^n)$ for $t>0$,
		\end{enumerate} where $Y=BMO_m(\BBR^n)$ if $p=\infty$ or $Y=L^p_m(\BBR^n)$ if $p\in (2,\infty)$.
		Moreover, it holds
		$$
		\|\nabla^m u\|_{\tentspace}\sim\|f\|_{Y}.
		$$
	\end{theorem} 
	\begin{proof}
		We only need to show (i) implies (ii), as the other direction was proven in Proposition \ref{pro:p>2 pol existence}. Let $p\in(2,\infty]$ and suppose $u$ is a global weak solution of \eqref{eq:the parabolic equation} with $S=\|\nabla^m u\|_{\tentspace}<\infty$.
		 
		\textbf{Step 1.} By Lemma \ref{lem:trace p>2}, there exists a unique $L^2_{loc}(\BBR^n)$ trace $f$ of $u$ at $t=0$ an it holds $f\in Y$ as claimed.
\textbf{Step 2.} For the interior uniqueness, recall that with the notation from Lemma \ref{lem:trace p>2}, it holds $C_{x_0,r}(|\nabla^mu|)\lesssim r^{-\frac{n}{p}} S.$	
Arguing as for \eqref{eq:easy bound on trace} in Lemma \ref{lem:trace p>2} (with $\phi(t)=t$), we obtain for $x_0\in \BBR^n$, $r>0$ and $t\in (0,r^{2m})$
\begin{equation}\label{eq:1}\int_{B(x_0,r)}\left|u(t,y)-\mathbb{P}_{x_0,r}^\omega(u(t,\cdot))(y)\right|^2\omega_{x_0,r}(y)dy\lesssim C_{x_0,r}(|\nabla^m u|)^2. \end{equation}

We argue as before that the coefficients $c_\alpha^{x_0,r}(t)$ of $\mathbb{P}_{x_0,r}^\omega(u(t,\cdot))$ are absolutely continuous over $[0,r^{2m}]$ and there are some constants $c_\alpha^{x_0,r}(0)\in \BBC$ with  
\begin{align*}|c_\alpha^{x_0,r}(t)-c_\alpha^{x_0,r}(0)|\lesssim_\omega r^{-m} \int_0^{r^{2m}}\fint_{B(x_0,r)}\big|\nabla^m u(t,x)\big|dxdt\lesssim C_{x_0,r}(|\nabla^m u|).\end{align*}

Recalling the $L^2_{loc}(\BBR^n)$ convergence $u(t,\cdot)\to f$ as $t\to 0$, we see that the coefficients $c_\alpha^{x_0,r}(0)$ correspond to the ones of $\mathbb{P}_{x_0,r}(f)$. Thus we have for $y\in B(x_0,r)$ and $t\in[0,r^{2m}]$,
\begin{equation}\label{eq:2} \left|\mathbb{P}_{x_0,r}(u(t,\cdot))(y)-\mathbb{P}_{x_0,r}(f)(y)\right|\lesssim C_{x_0,r}(|\nabla^m u|).\end{equation}
Furthermore, by Proposition \ref{prop:structure of L^p_m} if $p\in(2,\infty)$ or Proposition \ref{prop:structure of BMO_m} if $p=\infty$, the means $\fint_{B(x_0,r)} |f|\,dx$ grow at most polynomially in $|x_0|$ and so does $\|\mathbb{P}_{x_0,r}(f)\|_{L^\infty(B(x_0,r))}$. Combining this fact with \eqref{eq:1} and \eqref{eq:2} shows that $u$ satisfies the integrability condition \eqref{eq:ass_int_repr} from Theorem \ref{thm:int_repr} on any cylinder $(0,b)\times B(x,\sqrt[2m]{b})$ with $ 0<b<\infty$. Consequently, $u(t,\cdot)=\Gamma(t,s)u(s,\cdot)$ for any $0<s\leq t <\infty$, in the sense that for $h\in \mathscr C_c(\BBR^n)$ 
\begin{equation}\label{eq:int_repr_BMO}
\int_{\BBR^n}u(s,x)\overline{\Gamma(t,s)^*h(x)}dx=\int_{\BBR^n}u(t,x)\overline{h(x)}dx .\end{equation}

		\textbf{Step 3.} We need to show that representation \eqref{eq:int_repr_BMO} holds up to the boundary, that is, for $s=0$. To this end, fix $t>0$ and let $(s_k)_{k\in\BBN}\subseteq (0,t)$ be a sequence converging to zero. 
		By averaging, 
		\begin{equation}\label{eq:LDC}
		\int_{\BBR^n}u(s_k,x)\overline{\Gamma(t,s_k)^*h(x)}dx=\int_{\BBR^n}\fint_{B(x,\sqrt[2m]{t}/2)}u(s_k,y)\overline{\Gamma(t,s_k)^*h(y)}dydx.\end{equation} We will apply the Lebesgue Dominated Convergence Theorem to the sequence $(g_k)_{k\in\BBN}$ with $$g_k(x)\coloneqq\fint_{B(x,\sqrt[2m]{t}/2)}u(s_k,y)\overline{\Gamma(t,s_k)^*h(y)}dy.$$ First of all, we have $\Gamma(t,s_k)^*h\to \Gamma(t,0)^*h$ in $L^2(\BBR^n)$ as $k\to \infty$ (Lemma \ref{lem:backwards_prop}) and additionally the $L^2_{loc}(\BBR^n)$ convergence $u(t,\cdot)\to f$ as $t\to 0$ holds.  Thus, as $k$ tends to infinity
		$$g_k(x)\to \fint_{B(x,\sqrt[2m]{t}/2)}f(y)\overline{\Gamma(t,0)^*h(y)}dy$$ for every $x\in \BBR^n$. Applying $L^2$ off-diagonal bounds and \eqref{eq:1} together with \eqref{eq:2} gives us 
		\begin{align*}
		|g_k(x)|&\leq \sum_{j=1}^\infty \left(\fint_{B(x,\sqrt[2m]{t}/2)}|u(s_k,y)|^2dy\right)^{1/2}\|\mathbbm{1}_{B(x,\sqrt[2m]{t})}\Gamma(t,s_k)^*(\mathbbm{1}_{S_j(x,\sqrt[2m]{t})}h)\|_{L^2}\\
		&\lesssim \sum_{j=1}^\infty \left(\int_{B(x,\sqrt[2m]{t})}|u(s_k,y)|^2\omega_{x,\sqrt[2m]{t}}(y)dyds\right)^{1/2}e^{-\tilde c{2^{2m/(2m-1)j}}\left(\frac{t}{t-s_k}\right)^{1/(2m-1)}}\|\mathbbm{1}_{S_j(x,\sqrt[2m]{t})}h\|_{L^2}\\
		&\lesssim \sum_{j=1}^\infty \left( C_{x,\sqrt[2m]{t}}(|\nabla^mu|)+\|\mathbb{P}_{x,\sqrt[2m]{t}}(f)\|_{L^\infty(B(x,\sqrt[2m]{t}))}\right) e^{-\tilde c{2^{2m/(2m-1)j}}}\|\mathbbm{1}_{S_j(x,\sqrt[2m]{t})}h\|_{L^2}.
		\end{align*}
		Since $$\sup_{x\in\BBR^n}C_{x,\sqrt[2m]{t}}(|\nabla^mu|)\lesssim t^{-{n/(2mp)}}\|\nabla^mu\|_{\tentspace},$$ the term in the brackets grows at most polynomially in $|x|$  and so it is integrable when multiplied by $e^{-\alpha|x|^{\frac{2m}{2m-1}}}$ with arbitrary $\alpha>0 $. So, we can proceed as in the proof of Theorem \ref{thm:int_repr} to see that there exist $N\in \BBN$ and a constant $c>0$ such that for $0<\alpha<ct^{-1/(2m-1)}$
		\[|g_k(x)|\lesssim  |x|^N e^{-\alpha|x|^{2m/(2m-1)}}\|h\|_{L^2}\] uniformly in $k\in \BBN$ (the constant does depend on $t$). Thus, the sequence $(g_k)_{k\in\BBN}$ has an integrable dominant. We pass to the limit as $k$ tends to infinity in \eqref{eq:LDC} and obtain that the function $$x\mapsto \fint_{B(x,\sqrt[2m]{t}/2)}f(y)\overline{\Gamma(t,0)^*h(y)}dy$$ is integrable for every $t>0$ and 
		\begin{equation}\label{eq:after LDC}
		\int_{\BBR^n}\fint_{B(x,\sqrt[2m]{t}/2)}f(y)\overline{\Gamma(t,0)^*h(y)}dydx=\int_{\BBR^n}u(t,x)\overline{h(x)}dx .\end{equation}
		
		If we can show the integrability of $x \mapsto f(x)\overline{\Gamma(t,0)^*h(x)}$ then an application of Fubini's Theorem finishes the proof. For this we argue in the same style as above, since due to $f\in BMO_m(\BBR^n)$ or $f\in L^p_m(\BBR^n)$ if $p\in(2,\infty)$, we can control the averages $(\fint_{B(x,\sqrt[2m]{t})}|f(y)|^2)^{1/2}$ by some polynomial in $|x|$. Thus we can swap the integrals in \eqref{eq:after LDC} and obtain
		\begin{equation} \label{eq:representation pol}
		\int_{\BBR^n}f(x)\overline{\Gamma(t,0)^*h(x)}dx=\int_{\BBR^n}u(t,x)\overline{h(x)}dx .\end{equation}
		Let $B_k\coloneqq B(0,2^k)$ for $k\in \BBN$. Then, by Lebesgue Dominated Convergence, it holds for the left hand side of \eqref{eq:representation pol} 
		\begin{align*}
			\int_{\BBR^n}f(x)\overline{\Gamma(t,0)^*h(x)}dx&=\lim_{k\to \infty}	\int\mathbbm{1}_{B_k}f(x)\overline{\Gamma(t,0)^*h(x)}dx\\
			&=\lim_{k\to \infty}	\int\Gamma(t,0)(\mathbbm{1}_{B_k}f)(x)\overline{h(x)}dx\\
			&=\int\Gamma(t,0)f(x)\overline{h(x)}dx.
		\end{align*} Here we used that $\mathbbm{1}_{B_k}(x)f\in L^2(\BBR^n)$, as well as Lemma \ref{lem:extension to polynomial growth} together with the fact that $\text{supp }h$ is compact. Combined with \eqref{eq:representation pol}, this implies $\Gamma(t,0)f(x)=u(t,x)$ in $L^2_{loc}(\BBR^n)$ for $0<t<\infty$.	
	\end{proof}
	Theorem \ref{thm:wp pol} leads to the following Carleson measure characterization of $BMO(\BBR^n)$.
	\begin{corollary}
		\label{cor:Carleson measure characterisation}
		For $f\in L^2_{loc}(\BBR^n)$ it is equivalent
		\begin{enumerate}[label={\upshape(\roman*)}]
			\item There exists a global weak solution $u$ to \eqref{eq:the parabolic equation}, for which $$ d\mu(x,t)=|t^m\nabla^m u(t^{2m},x)|^2\frac{dxdt}{t}$$ is a Carleson measure and the $L^2_{loc}(\BBR^n)$ trace of $u$ is given by $f$.
			\item There exists a polynomial $P\in \pol$ such that $f-P\in BMO(\BBR^n)$.
		\end{enumerate}
		Moreover, $\|\nabla u\|_{T^{\infty,2}_m}\sim\|f-P\|_{BMO}.$
	\end{corollary}
	By Proposition \ref{prop:comp_tent_kenig_bigp} and Theorem \ref{thm:wp pol} we further see the Cauchy problem \eqref{eq:Cauchy problem} is well-posed for $X=X^p_m$ and $Y=L^p(\BBR^n)$ if $p\in(2,\infty)$. If $p=\infty$, with this method, we only obtain a trace in $BMO(\BBR^n)$. However, copying the slice spaces estimates in the proof of \cite[Theorem 5.4]{AMP15} gives $f\in L^\infty(\BBR^n)$.
	Hence, $X^\infty_m$ is a well-posedness class for $L^\infty(\BBR^n)$.
	
	Under the uniform boundedness assumption on the propagators, another $L^p(\BBR^n)$ well-posedness class is given by $L^\infty(L^p)$. This is true also for $p=\infty$, but we neglect this case below to avoid distinguishing different cases.
	 \begin{theorem}
	 	\label{thm:wp p big}
	 	Let $p\in(2,\infty)$. Suppose that \ref{ubc}
	 	holds. Then for a distribution $u\in \mathscr D'(\BBR^{n+1}_+)$ it is equivalent
	 		\begin{enumerate}[label={\upshape(\roman*)}]
	 			\item $u$ is a global weak solution of \eqref{eq:the parabolic equation} and $\nabla^m u\in \tentspace$.
	 			\item $u$ is a global weak solution of \eqref{eq:the parabolic equation} and $u-P\in L^\infty(L^p)$ for some polynomial $P\in\pol$.
	 			\item $u$ is a global weak solution of \eqref{eq:the parabolic equation} and $u-P\in X^p_m$ for some polynomial $P\in\pol$.
	 			\item There are unique $ f\in L^p(\BBR^n)$ and $P\in\pol$, so that $u(t,\cdot)-P=\Gamma(t,0)f$ in $L^2_{loc}(\BBR^n)$ for $t>0$.
	 		\end{enumerate}
	 		In this case, $u-P$ is the solution with trace $f$ obtained in Proposition \ref{pro:p>2 existence} and it holds 
	 		$$
	 		\|\nabla^m u\|_{\tentspace}\sim\|u-P\|_{X^p_m}\sim\|u-P\|_{L^\infty(L^p)}\sim\|f\|_{L^p(\BBR^n)}.
	 		$$
	 	\end{theorem} 
	 \begin{proof}
	 	Points (i), (iii) and (iv) are equivalent without assuming \ref{ubc}, see Theorem \ref{thm:wp pol}, Propositions \ref{pro:p>2 existence} and \ref{prop:comp_tent_kenig_bigp}. 
	 	Also, (iv) implies (ii) by \ref{ubc}, so we only need to show the converse. We proceed as in Example \ref{ex:example wp}. For the third step, we need to show 
	 	\begin{equation}
	 	\label{eq:eq8}
	 	\|\Gamma(t,s)^*h-\Gamma(t,0)^*h\|_{L^{p'}}\to 0
	 	\end{equation} as $s\to 0$ for every $h\in \mathscr C_c(\BBR^n)$. This is seen by localization and H\"older's inequality combined with off-diagonal estimates and $L^2$ continuity results, cf.\ \cite[Proposition 5.11]{AMP15}.
	\end{proof}
	\begin{remark}\label{rem:continuity big p}
		Let $p\in (2,\infty)$ and suppose \ref{ubc}. Then for all $r\in (2,p)$ the global weak solution $u_f(t,\cdot)=\Gamma(t,0)f$ with $f\in L^r(\BBR^n)$ belongs to $\mathscr C_0([0,\infty);L^r(\BBR^n))$.
	\end{remark}
	\begin{proof}
		We interpolate between $r=2$, where the conclusion is true, and $r=p$, where \ref{ubc} holds. Precisely, let $f\in L^r(\BBR^n)$ and $g\in \mathscr D (\BBR^n)$ be such that
		$\|g-f\|_{L^r}<\varepsilon$. We estimate for any $0\leq s<t<\infty$ and $\theta\in (0,1)$ with $\frac{1}{r}=\frac{\theta}{2}+\frac{1-\theta}{p}$,
		\begin{align*}
		\|\Gamma(t,0)f-\Gamma(s,0)f\|_{L^r}&\leq 	\|\Gamma(t,0)(f-g)-\Gamma(s,0)(f-g)\|_{L^r}+	\|\Gamma(t,0)g-\Gamma(s,0)g\|_{L^r} \\ &\lesssim
		\varepsilon + \|\Gamma(t,0)g-\Gamma(s,0)g\|_{L^p}^{1-\theta}\|\Gamma(t,0)g-\Gamma(s,0)g\|_{L^2}^\theta\\
		&\lesssim \varepsilon + \|g\|_{L^p}^{1-\theta}\|\Gamma(t,0)g-\Gamma(s,0)g\|_{L^2}^\theta.
		\end{align*} The claim follows easily.
	\end{proof}

	\subsection{Tent space solutions $1\leq p<2$}
	As outlined in Section \ref{sec:existence small p}, we need to assume \ref{ubc}. However, even with \ref{ubc}, we could not prove the continuous dependence of the solution on the data in the sense of 
	\[\|\nabla^m u_f \|_{\tentspace}\lesssim \|f\|_{L^p}.\]
	Contrary to the case $p>2$, our only general estimate in this direction is the result of Proposition \ref{prop:comp_tent_kenig_smallp}. On the other hand, the only available bound the non-tangential norm of the solution by the initial data requires enlarging the exponent slightly to $r\in(p,2)$. 
	
	We arrive at the same conclusion, when considering the class 
	$ L^\infty(0,\infty;L^p(\BBR^n))$. Indeed, as in the third step of the proof of Theorem \ref{thm:wp p big}, we need to show 
		\begin{equation*}
		\|\Gamma(t,s)^*h-\Gamma(t,0)^*h\|_{L^{p'}}\to 0
		\end{equation*} as $s\to 0$ for every $h\in \mathscr C_c(\BBR^n)$. Our previous proof was based on the H\"older inequality, since $p>2$, and $L^2$ convergence results for the propagators. This does not apply here, but provided the propagators are uniformly bounded for some $1\leq q<p<2$, we can again use interpolation between $2$, where the convergence is true and $q$, where the uniform bounds hold. We obtain the following.

		\begin{theorem} \label{thm:wp_small_p}
			Let $1\leq p<r<2$. Suppose \ref{ubc} holds. Then for a distribution $u\in \mathscr D'(\BBR^{n+1}_+)$ it is equivalent
			\begin{enumerate}[label={\upshape(\roman*)}]
				\item $u$ is a global weak solution of \eqref{eq:the parabolic equation} in $ L^\infty(0,\infty;L^r(\BBR^n))$.
				\item $u$ is a global weak solution of \eqref{eq:the parabolic equation} in $X^r_{m}$.
				\item There is a unique $ f\in L^r(\BBR^n)$ such that $u(t,\cdot)=\Gamma(t,0)f$ in $L^r(\BBR^n)$ for $t>0$.
			\end{enumerate} In this case $u\in \mathscr  C_0([0,\infty);L^r(\BBR^n))$ and $\|f\|_{L^r}\sim \|u\|_{L^\infty(L^r)}\sim \|u\|_{X^r_m}.$
		\end{theorem}
	\begin{proof} See Lemma \ref{lem:wp_small} and Example \ref{ex:example wp}, complemented by the comment above. The continuity result is proven as in Remark \ref{rem:continuity big p}.
	\end{proof}
	\begin{theorem}\label{thm:tent space wp small p}
				Let $1\leq p<r<2$. Assume the strong ellipticity bounds \eqref{eq:strong lower ellipticity bounds} and \ref{ubc}. Then for a distribution $u\in \mathscr D'(\BBR^{n+1}_+)$ it is equivalent
				\begin{enumerate}[label={\upshape(\roman*)}]
						\item $u$ is a global weak solution of \eqref{eq:the parabolic equation} and $\nabla^m u\in T^{r,2}_m.$
						\item There is unique $ f\in L^r(\BBR^n)$ and $P\in\pol$, such that $u(t,\cdot)-P=\Gamma(t,0)f$ in $L^r(\BBR^n)$ for $t>0$.
				\end{enumerate} In this case, all statements of Theorem \ref{thm:wp_small_p} are true for $u-P$ and 
				\[\|f\|_{L^r}\sim \|u-P\|_{L^\infty(L^r)}\sim\|u-P\|_{X^r_m}\sim\|\nabla^m u \|_{T^{r,2}_m}.\] 
	\end{theorem}
	\begin{proof}
		By Lemma \ref{lem:trace integrability} (i) implies that statements of Theorem \ref{thm:wp p big} are true for the solution modified by a polynomial, hence (ii) follows. Proposition \ref{prop:comp_tent_kenig_smallp} leads to the converse implication.
	\end{proof}
\section{Uniform boundedness of the propagators and H\"older continuity of solutions}
	\setcounter{theorem}{0} \setcounter{equation}{0}
	\label{sec:bounds on propagators}
The uniform boundedness of propagators plays a peculiar role in the well-posedness theory from Section \ref{sec:uniqueness results}. In the remaining part of this work we study for which operators $L$ and exponents $p\in[1,\infty]$ the \ref{ubc} assumption is true. We begin with a few examples in the spirit of \cite{AMP15}.
	\subsection{Examples}
	\label{sec:bound examples}
 
	\subsubsection{Kernel bounds.} 
	\label{sec:kernel bounds} A sufficient condition for \ref{ubc} to hold is that the propagators satisfy kernel bounds. By this we mean that their Schwartz kernels $k(t,s,\cdot,\cdot)$ are measurable functions on $\BBR^n\times \BBR^n$ with
	\begin{equation}\label{eq:kernelbd}
	|k(t,s,x,y)|\leq c_1(t-s)^{-\frac{n}{2m}}\exp\left(-c_2\left(\frac{|x-y|^{2m}}{(t-s)}\right)^{\frac{1}{2m-1}}\right),\end{equation} 
	for some positive constants $c_1, \,c_2>0$, all $0\leq s<t<\infty$ and almost all $ x,\; y \in \BBR^n$.
	For any $f\in L^2(\BBR^n)$ and  $0\leq s<t<\infty$ we then have the integral representation 
	\[\Gamma(t,s)f=\int k(t,s,x,y)f(y)dy,\]
	which can be extended to hold for all $f\in L^p(\BBR^n)$, $p\in [1,\infty]$. Here are some examples of operators, for which \eqref{eq:kernelbd} holds
	\begin{enumerate}[label=\upshape(\roman{*})]
		\item $L=(-1)^m \Delta^m$ for any $m,\,n\in \BBN_+$ and more general for any $L$ with constant coefficients, see \cite[Proposition 45]{AQ00}.
		\item $L$ an autonomous operator and $n\leq 2m$, see \cite{D95, AT00}.
		\item $L$ an autonomous operator coefficients small in $BMO(\BBR^n)$ norm, \cite[Proposition 47]{AQ00}.
		\item In the autonomous case, condition \eqref{eq:kernelbd} is stable under small $L^\infty(\BBR^n)$-perturbations of the coefficients, see \cite[Proposition 43]{AQ00}.
		\item $m=N=1$ and the (not necessarily autonomous) coefficients of $L$ are real, see \cite{A67}. 
	\end{enumerate}
	For $N\ge 1 $ in the autonomous case see the references in the introduction of \cite{AQ00}. In particular, in all above listed cases all of the derived well-posedness results hold and, for any $p\in (1,\infty)$, the convergence to the trace holds in $L^p(\BBR^n)$.
	\subsubsection{Coefficients in $BV(L^\infty)$ for $p<2$.}
	\label{sec:coeff BV}
	In Sections \ref{sec:coeff BV} and \ref{sec:coeff small norm} we formulate the higher order counterparts of the results of \cite[\S 6]{AMP15}. The methods are identical.  As in loc.\ cit.\ we address here the case $p<2$ only.
		\begin{definition}
			\label{def:class M} We say that $\underline A\in L^\infty(\BBR^{n};\BBC^{NM\times NM})$ belongs to the class $\mathcal M(\Lambda,\lambda, q,M)$ for some $q\in[1,2)$ and $M\colon [2,q')\to (0,\infty)$, if it satisfies the ellipticity estimates with constants $\lambda,\, \Lambda>0$ and for all $s\in[2,q')$ we have 
			\begin{equation}\label{eq:assumption 2} \sup_{t>0}\|\sqrt{t} \nabla^me^{-tL_0^*}\|_{\mathscr L (L^s)}\leq M(s)<\infty,\end{equation} where $L_0$ denotes the autonomous operator arising from $\underline A$.
			\end{definition}
			
			Condition of the form of \eqref{eq:assumption 2} was first introduced in Section \ref{sec:aut} and is substantial in the proof of the boundedness of $\Mmaxreg$ in Proposition \ref{prop:ext_M}. Every operator $L_0$ will satisfy \eqref{eq:assumption 2} for some $q$ and $M$, which can be chosen to depend on the ellipticity constants and dimensions only, see the off-diagonal estimates of Section \ref{sec:aut} and the connection between the special exponents for $L_0$ and $L_0^*$, which we discuss in the proof of Proposition \ref{prop:ext_M}.
			\begin{definition}
				\label{def:BV} We say that $A\colon (0,\infty) \to L^\infty(\BBR^{n};\BBC^{NM\times NM})$ has bounded variation in time, denoted $A\in BV (L^\infty)$, if 
				\[\|A\|_{BV(L^\infty)}\coloneqq\sup\left \{\sum_{k=0}^\infty\|A(t_{k+1},\cdot)-A(t_k,\cdot)\|_{L^\infty}\mid (t_k)_{k\in \BBN}\subseteq [0,\infty) \mbox{ non-decreasing} \right\}<\infty.\] 
				\end{definition}
					\begin{proposition}
						\label{prop:bound BV}
						Suppose $A\in BV (L^\infty)$ admits ellipticity constants $\lambda, \, \Lambda >0$. Fix some $q\in [1,2)$ and $M\colon [2,q')\to (0,\infty)$ be such that for all bounded intervals $I\subseteq (0,\infty)$ it holds
						$$A_I(\cdot)\coloneqq\fint_I A(s,\cdot)ds \in \mathcal M (\Lambda,\lambda, q, M).$$ Then for  $p\in(\max\{1;p_c\},2)$, where $p_c=\max\{\frac{nq}{n+mq};\frac{2n}{n+mq'}\}$ as in Proposition \ref{prop:ext_M}, \ref{ubc} holds and the bound depends on the ellipticity, dimensions, the function $M$, $p, q$ and $\|A\|_{BV(L^\infty)}$.
						\end{proposition}
						\begin{proof} 
						After approximation of $A$ with piecewise constant in time matrices, the proof exploits the explicit form of energy solutions in this case. Crucial tent space estimates are derived with help of Proposition \ref{prop:ext_M}. The detailed reasoning is identical as in \cite[Theorem 6.9, Lemma 6.7]{AMP15}.
							We only give references needed to adjust the proof to the higher order setting. Consider $w=e^{-tL_0}f$ for $L_0$ autonomous and $f\in L^p(\BBR^n)\cap L^2(\BBR^n)$. Then estimate 
							\[\|\nabla^mw\|_{\tentspace}\lesssim \|f\|_{L^p}\] is true by \cite[Corollary 3.5 (ii)]{CMY16} for $p\in (q_-(L_0),q_+(L_0))$. This condition is fulfilled by the assumption on $q$, cf.\ the discussion in the proof of Proposition \ref{prop:ext_M}.
							 In this range of exponents, the vertical square function estimates remain true for higher order operators, see \cite[p. 96]{A07}.\qedhere
							\end{proof}
	\subsubsection{Small perturbations for $p<2$.}
	\label{sec:coeff small norm}
	\noindent 
	Let us recall from Corollary \ref{cor:duhamel prop} that for the considered operator $L$ and autonomous $L_0=(-1)^m \mbox{div}_m \underline A(x)\nabla^m$ it holds
	\begin{equation} \label{eq:repr1}\Gamma(t,0)h=e^{-tL_0}h
		+\int_0^t e^{-(t-s)L_0}\text{div}_m (A(s,\cdot)-\underline A)\nabla^m \Gamma(s,0)h\; ds\end{equation}
	on $L^2(\BBR^n)$. If $\|A-\underline A\|_{L^\infty(\BBR^{n+1}_+)}$ is small, then by the boundedness of the maximal integral operators from Appendix \ref{sec:integral ops}, we are able to see that the uniform $L^p(\BBR^n)$ boundedness of the family $(e^{-tL_0})_{t>0}$ is inherited by $(\Gamma(t,0))_{t>0}$.
	
	\begin{proposition} \label{pro:wp_small_p_abstandklein}
		Let $A\in L^\infty(\BBR^{n+1}_+; \BBC^{NM\times NM})$ admit ellipticity constants $\lambda,\,\Lambda>0$. Suppose there exists $\underline A\in \mathcal M (\Lambda,\lambda, q, M)$ with $q\in [1,2)$ and $M\colon [2,q')\to (0,\infty)$, such that for some $p\in(\max\{1;\frac{2n}{n+mq'}\},2)$
		\[\varepsilon \coloneqq \|A-\underline A\|_{L^\infty}<\frac{1}{\|\tilde {\mathcal M}_{L_0}\|_{\mathscr L ({\tentspace})}},\] then \ref{ubc} holds.
	\end{proposition}
	\begin{proof} This is an easy adaptation of the proof of \cite[Theorem 6.14]{AMP15}. 
	\end{proof}
	\begin{remark}
	 Similarly to Proposition \ref{pro:wp_small_p_abstandklein} we can treat the case $A\in \mathscr C([0,T]; L^\infty(\BBR^n))$ and $A(s,\cdot) \in\mathcal M (\Lambda,\lambda, q, M)$ for all $s\in [0,T]$, to derive the uniform boundedness of propagators up to time $T$. 
	\end{remark}
Indeed, we partition the interval $[0,T]$, according to the uniform continuity of $A$, and let $A(t_k,\cdot)$ play the role of $\underline A$ in Proposition \ref{pro:wp_small_p_abstandklein} on the corresponding interval $[t_k,t_{k+1})$. The details can be found in \cite[Theorem 6.15]{AMP15}.
	
	\subsection{Uniform boundedness of the propagators for $|p-2|$ small}
	\label{sec:uniform bounds}
	For a homogeneous higher order elliptic operator $L=(-1)^m \text{div}_m A(t,x)\nabla^m$, introduced in Section \ref{sec:ellipticity}, let $I_A$ be the maximal interval of exponents $p\in[1,\infty]$, such that \ref{ubc} holds, that is $\{\Gamma(t,s)| \; 0\leq s\leq t<\infty\}$ is uniformly bounded on $L^p(\BBR^n)$.
Since $2\in I_A$, $I_A\neq \emptyset$. In this section we will prove that also the interior of $I_A$ is non-empty. With methods based on the ideas from \cite{ABES18}, we show the following.
	\begin{theorem} \label{thm:uniform_bdd_p_close_to_2}
		There exists $\varepsilon >0$ depending only on ellipticity and dimensions, such that for all $p\in [1,\infty]$ with 
		$p\in(2-\varepsilon, 2+\varepsilon)$ the family  $\{\Gamma(t,s)| \; 0\leq s\leq t<\infty\}$ is uniformly bounded on $L^p(\BBR^n)$.
	\end{theorem}
	
	Comparing this result with item (ii) in Section \ref{sec:aut}, let us underline that we do not have a quantitative description of $\varepsilon$. 
	\subsubsection{Reduction to the bound on parabolic cylinders.}
	\label{sec:bound reduction}
	We show how the uniform boundedness of the propagators follows from local estimates.

	\textbf{Step 1.} Reduction to the case $p>2$. 
	
	We argue by duality. If $p\in(\max\{1 
		;2-\frac{\varepsilon}{1+\varepsilon}\},2)$, then its dual exponent, determined by $1=\frac{1}{p}+\frac{1}{p'}$, satisfies $p'\in (2,2+\varepsilon)$. Recall from Lemma \ref{lem:backwards_prop} that for $0\leq s<t<\infty$ the adjoint
	$\Gamma(t,s)^*$ equals $\tilde \Gamma(t-s,0)$ on $L^2(\BBR^n)$, where $\tilde \Gamma$ is the propagator associated to the matrix $\tilde A (s,x)$ given by $A^*(t-s,x)$ if $s\in(-\infty,t]$ and $A^*(0,x)$ otherwise. The new matrix satisfies the same ellipticity estimates as $A$, uniformly in $0<t<\infty$. Thus, by density, we obtain $$\sup_{0\leq s<t}\|\Gamma(t,s)\|_{\mathscr L (L^p)}=\sup_{0\leq s<t}\|\tilde \Gamma(t-s,0)\|_{\mathscr L (L^{p'})}<\infty$$, if the claim is true for $p'$. This bound holds uniformly in $t>0$.

	\textbf{Step 2.} Reduction to the case $s=0$.
	
	By uniqueness of energy solutions we know that $\Gamma(t,s)=\tilde \Gamma (t-s,0)$ for $\tilde \Gamma$ the propagator associated to the matrix $\tilde A(t,\cdot)=A(t+s,\cdot)$ for $t>0$. Thus it is enough to obtain bounds of $\{\tilde \Gamma(t,0)\mid 0<t<\infty\}.$

	\textbf{Step 3.} Reduction to $L^2-L^p$ off-diagonal estimates.
	
	In \cite[Proposition 2.1.ii]{BK05} Blunck and Kunstmann showed that an $L^2- L^p$ off-diagonal estimate for any linear operator $R$ on $L^2(\BBR^n)$ of the form
	\begin{equation}\label{eq:BK_off-diagonal}
	\big\|\mathbbm{1}_{B(x,\sqrt[2m]{t})}R\mathbbm{1}_{B(y,\sqrt[2m]{t})}\big\|_{L^2\to L^p} \lesssim t^{\frac{n}{2m}(\frac{1}{p}-\frac{1}{2})} g\left(\frac{d(x,y)}{\sqrt[2m]{t}}\right)
	\end{equation}
	is enough to deduce the bound $\|R\|_{\mathscr{L}(L^p)}\leq C \sum_{k=0}^\infty (k+1)^\lambda g(k)$ with some constants $C,\, \lambda>0$ independent of the operator $R$, the function $g$ and the parameter $t>0$. Here, we assume that $g\colon \BBR_{\ge 0} \to \BBR_{>0}$ is decreasing with $h=-\log g$ convex and $\liminf_{a\to\infty}h(a)/a>0$. We will apply this result with $g(a)\coloneqq e^{-ca^{2m/(2m-1)}}$ to $R=\Gamma(t,0)$. Note also that in \eqref{eq:BK_off-diagonal} we can replace $d(x,y)$ by $d({B(x,\sqrt[2m]{t})},{B(y,\sqrt[2m]{t})})$ on the expense of some additional constants.
	
	\textbf{Step 4:} Reduction to the $L^2- L^p$ bounds on parabolic cylinders.
	
	Recall from \ref{prop:off2} the $L^2$ off-diagonal estimates
	\begin{equation*}
	\big\|\mathbbm{1}_{B(x,\sqrt[2m]{t})}\Gamma(t,0)\mathbbm{1}_{B(y,\sqrt[2m]{t})}\big\|_{L^2\to L^2} \lesssim  g\left(\frac{d({B(x,\sqrt[2m]{t})},{B(y,\sqrt[2m]{t})})}{\sqrt[2m]{t}}\right),
	\end{equation*}
	with $g(a)\coloneqq e^{-ca^{2m/(2m-1)}}$ for some constant $c>0$. By interpolation, to obtain \eqref{eq:BK_off-diagonal}, we only need to show the localized $L^2-L^p$ bound
	\begin{equation}\label{eq:revisited_off-diagonal}
	\big\|\mathbbm{1}_{B(x,\sqrt[2m]{t})}\Gamma(t,0)\big\|_{L^2\to L^p} \lesssim t^{\frac{n}{2m}(\frac{1}{p}-\frac{1}{2})},
	\end{equation} with constants independent of $x\in \BBR^n$ and $t>0$. Since for  $f\in L^2(\BBR^n)$, $\|\Gamma(t,0)f\|_{L^2}\leq \|f\|_{L^2}$ holds uniformly in  $t>0$, inequality \eqref{eq:revisited_off-diagonal} follows from the following estimate
	\begin{equation*}
	\sup_{s\in (t/2,2 t)}\left(\fint_{B(x,\sqrt[2m]{t})}|\Gamma(s,0)f|^pdy\right)^{1/p} \lesssim \left(\fint_{(t/4,4 t)}\fint_{B(x,2\sqrt[2m]{t})}|\Gamma(s,0)f|^2dy\right)^{1/2}.
	\end{equation*} By scaling and translation, we only need to consider the case $t=1$ and $x=0$. Summarizing, if we denote $u(t,x)=\Gamma(t,0)f(x)$, then our goal is to prove the existence of some $\varepsilon >0$ depending on ellipticity and dimensions only, such that for all $2<p<2+\varepsilon$ and $f\in L^2(\BBR^n)$ the following bound is true
	\begin{equation}\label{eq:goal_uniform_L^p}
	\sup_{s\in (1/2,2)}\left(\fint_{B(0,1)}|u(s,y)|^pdy\right)^{1/p} \lesssim \left(\fint_{(1/4,4)}\fint_{B(0,2)}|u(s,y)|^2dy\right)^{1/2}.
	\end{equation} 
	\subsubsection{Outline of the method.}
	\label{sec:idea proof}
	By multiplying $u$ with a cut-off function $\chi \in\mathscr C_c ^\infty (\BBR^{n+1}_+)$ we obtain a function $v\in L^2(\BBR;L^2(\BBR^n))$, which now solves some inhomogeneous equation weakly on the whole space $\BBR^{n+1}$, where we extend $A(t,x)=A(0,x)$, if $t<0$. Local information about $u$ on $\text{supp}\, \chi$ carries over to $v$ and vice versa. It is therefore enough to prove the bound \eqref{eq:goal_uniform_L^p} for $v$. The global setting, that is $v\in L^2(\BBR^{n+1})$, allows us to extract valuable information from the time derivative. Formally, $$\partial_t v = D_t^{1/2}H_t D_t^{1/2}v,$$ where $D_t^{1/2}=\mathcal F^{-1}(|\tau|^{1/2}\mathcal F)$ is the half derivative in $t$ and $H_t=\mathcal F^{-1}(i\tau/|\tau|\mathcal F)$ is the one dimensional Hilbert transform. For the rest of the section, $\mathcal F$ denotes the $n+1$ dimensional Fourier transform.
	
	By standard interpolation, we realize that $v\in L^2(\BBR; H^m(\BBR^n))$ and $\partial_t v \in L^2(\BBR; H^{-m}(\BBR^n))$ imply $D_t^{1/2}v\in L^2(\BBR^{n+1})$.
	Using an abstract result of \v{S}ne\u{\i}berg (see \cite{S74} or \cite[Lemma 5.16]{A07}), we show higher integrability of $v$ and $D^{1/2}_tv$, that is $v,\, D_t^{1/2}v\in L^p(\BBR^{n+1})$ and use the Campanato characterization of the H\"older continuity to conclude. 
	\subsubsection{Abstract results.}
	\label{sec:abstract results}
	In this section we introduce the formal setting for the proof of \eqref{eq:goal_uniform_L^p}. We adapt the definition of energy spaces from \cite{ABES18} to the higher order case.
	
	\begin{definition}\label{def:energy_spaces}
		Let $p\in[1,\infty)$. Let $H^{1/2,p}=H^{1/2,p}(\BBR;L^p(\BBR^n))$ denote the potential space consisting of such $f\in L^p(\BBR^{n+1})$, for which $D_t^{1/2}f\in L^p(\BBR^{n+1})$. We equip this space with norm $$\|f\|_{H^{1/2,p}}\coloneqq(\|f\|_{L^p(\BBR^{n+1})}^p+\|D_t^{1/2}f\|_{L^p(\BBR^{n+1})}^p)^{1/p}.$$
		Further, we define the p-energy space $E_p\coloneqq L^p(\BBR;W^{m,p}(\BBR^m))\cap H^{1/2,p}(\BBR;L^p(\BBR^n))$ and equip it with norm $\|f\|_{E_p}\coloneqq (\|f\|_{L^p(\BBR;W^{m,p})}^p+\|D_t^{1/2}f\|_{H^{1/2,p}})^{1/p}$.
	\end{definition} 

	By the Gagliardo--Nirenberg inequality (Lemma \ref{lem:Gagliardo--Nirenberg inequality}) and ellipticity it holds for $w\in H^m(\BBR^n)$
	\begin{equation}
	\label{eq:coercivity}\text{Re}\int A(t,x)\nabla^m v(x) \overline{\nabla^m v(x)}dx\ge \lambda \|\nabla^m v\|^2_{L^2(\BBR^n)}\ge  \tilde\lambda \|v\|^2_{H^m(\BBR^n)}-\|v\|_{L^2(\BBR^n)}^2\quad\mbox{for a. e. } t\in \BBR,\end{equation}
	where $\tilde \lambda>0$ is some new constant dependent on the order $m$, dimension $n$ and the ellipticity constant $\lambda$. Estimate \eqref{eq:coercivity} together with the Lax-Milgram Lemma implies the following.
	\begin{lemma}
		\label{lem:lax_milgram}
		The operator $\mathscr L\coloneqq \partial_t+(-1)^m\mbox{div}_mA\nabla^m + 2$, initially defined on $\mathscr C^\infty_c(\BBR^{n+1})$ extends to a bounded, invertible operator $E_2\to E_2'$ through the pairing 
		\[\langle \mathscr L u ,v\rangle\coloneqq \int_{\BBR^{n+1}}A\nabla^m u \overline{\nabla^m v} +H_tD_t^{1/2}u\overline{D_t^{1/2}v} +2u\overline{v}dx dt.\] Furthermore, the norm of $\mathscr L$ and the norm of its inverse depend only on the ellipticity and the dimensions.
	\end{lemma}

	\begin{proof} See \cite[Lemma 3.2]{ABES18}.
	\end{proof} 
	We borrow from \cite{ABES18} a parabolic Sobolev embedding theorem.
	\begin{lemma}
		\label{lem:embed}
		Let $1<p<n+2$ and $p^*>p$ be determined by $\frac{1}{p^*}=\frac{1}{p}-\frac{1}{n+2}$. Then $E_p \hookrightarrow L^{p^*}(\BBR^{n+1})$ with 
		\[\|u\|_{ L^{p^*}(\BBR^{n+1})}\lesssim \|\nabla u \|_{ L^{p}(\BBR^{n+1})} +\|D_t^{1/2}u\|_{ L^{p}(\BBR^{n+1})}.\]
	\end{lemma}
	\begin{proof}
		Note that $E_p\hookrightarrow \tilde E_p\coloneqq L^p(\BBR;W^{1,p}(\BBR^n))\cap H^{1/2,p}(\BBR;L^p(\BBR^n))$, which is the energy space used in \cite{ABES18}. The result follows by \cite[Lemma 3.4]{ABES18}.
	\end{proof}
	\begin{corollary}
		\label{cor:embed}
		Let $1<p<n+2$ and $p^*>p$ be determined by $\frac{1}{p^*}=\frac{1}{p}-\frac{1}{n+2}$. Suppose $v\in E_p$. Then for all integer $0\leq k\leq m$ and $p_k\ge p$ given by $\frac{1}{p_k}=\frac{m-k}{m}\frac{1}{p^*}+\frac{k}{m}\frac{1}{p}$, it holds
		$\nabla^k v \in L^{p_k}(\BBR^{n+1})$ with
		\[\|\nabla^k v\|_{L^{p_k}(\BBR^{n+1})}\lesssim \|v\|_{L^{p^*}(\BBR^{n+1})}^{\frac{m-k}{m}}\|\nabla^m v \|_{L^p(\BBR^{n+1})}^\frac{k}{m}\lesssim \|v\|_{E^p}.\]
	\end{corollary}
	\begin{proof}
		This is an application of the Gagliardo--Nirenberg inequality from Lemma \ref{lem:Gagliardo--Nirenberg inequality} applied slice-wise, followed by the H\"older inequality (in time) and Lemma \ref{lem:embed}.
	\end{proof}
	The interval for $p$ in the last lemma is not optimal, but we are not bothered by this fact since the application of the \v{S}ne\u{\i}berg's Lemma does not give us a quantitative information about the intervals for $p$ we wish to work with.
	
	\begin{lemma}\label{lem:interpolation scales}
		The energy spaces $(E_p)$ and their duals $((E_p)')$ form complex interpolation scales. Precisely, let $\varepsilon>0$ and $1+\varepsilon<p_0\leq p_1<1+\varepsilon^{-1}$. For $\theta\in (0,1)$ the complex interpolation identity is true
		\[[E_{p_0},E_{p_1}]_\theta = E_{p_\theta} \text{\;\;and\;\;} [(E_{p_0})',(E_{p_1})']_\theta = (E_{p_\theta})', \text{\;\; where \;} \frac{1}{p_\theta}=\frac{1}{p_0}+\frac{1}{p_1}\] and the equivalence constants depend only of $\varepsilon$ and the dimensions.
	\end{lemma}
	
	\begin{proof}
		The proof is a simple adaptation of the reasoning from \cite[Lemma 6.1]{AEN18}, so we skip the details.
	\end{proof}
	Lemma \ref{lem:interpolation scales} allows to apply the \v{S}ne\u{\i}berg's result on bounded operators acting on complex interpolation scales, which in our setting translates into
	\begin{proposition}
		\label{prop:invertibility_around_2}
		There exists an $\varepsilon>0$, such that for all $1<p<\infty$ satisfying $|p-2|<\varepsilon$ the operator  $\mathscr L\coloneqq \partial_t+(-1)^m\mbox{div}_mA\nabla^m + 2\colon E_2\to (E_2)'$ extends to a bounded and invertible operator $E_p\to (E_{p'})'$. The inverse agrees with the one for $p=2$ on $(E_2)'\cap (E_{p'})'$. The norm of the inverse and the value of $\varepsilon$ depend only on the ellipticity and the dimensions.
	\end{proposition}
	\begin{proof}
		Given Lemmas \ref{lem:lax_milgram} and \ref{lem:interpolation scales} the proof is identical as in \cite[Lemma 7.1]{ABES18}. 
	\end{proof}
	\subsubsection{H\"older regularity of solutions.}
		Suppose $u\in L^2_{loc}(0,\infty;H^m_{loc}(\BBR^n))$ is a global weak solution to \eqref{eq:the parabolic equation}. Let $\chi \in\mathscr C_c ^\infty (\BBR^{n+1}_+)$ be supported in the cylinder $ (3/8,4)\times B(0,3/2)$, satisfy $\chi\leq 1$ and $\chi=1$ on $ (1/2,2)\times B(0,1)$. Define 
		$v(t,x)\coloneqq \chi(t,x)u(t,x)$ for $(t,x)\in \BBR^{n+1}$. The following properties for $v$ are direct consequences of the a priori energy estimates for $u$ (see Section \ref{sec:energy estimates}) 
		\begin{enumerate}[label=\upshape{(\roman{*})}]
			\item $ v\in L^2(\BBR; H^{m}(\BBR^n))$.
			\item  $\partial_t v\in L^2(\BBR; H^{-m}(\BBR^n))$.
			\item  $ v\in L^q(\BBR^{n+1})$ for $1\leq q\leq 2+\frac{4m}{n}.$
		\end{enumerate} 
		By (i) and (ii), we have $D_t^{1/2}v\in L^2(\BBR^{n+1})$, see \cite[Chap.\ XVIII, \S1, Theorem 6]{DL92}. Thus, $v\in E_2$. 

	By a repetitive use of the product rule, we deduce that $v=\chi u$ solves
	\begin{equation}\label{eq:inhom_for_v}
	\mathscr L v =  \partial_tv+(-1)^m\mbox{div}_mA\nabla^mv + 2v= 2v+ f 
	+ \mbox{div}_m F+
	\sum_{k=1}^{m-1}\sum_{|\xi|=k}\partial^\xi F_\xi,
	\end{equation}
	in the weak sense, where the involved functions are $f= u \partial_t \chi+c_{m} A\nabla^mu\nabla^m\chi$, $F=(F_\alpha)_{|\alpha|=m}$ with $F_\alpha=\sum_{|\beta|=m}a_{\alpha,\beta}\sum_{\gamma<\beta} c_{\beta,\gamma}\partial^\gamma u\partial ^{\beta-\gamma}\chi$ and the $F_\xi$ satisfy
	\[\|F_{\xi}\|_ {L^2(\BBR^{n+1})} \lesssim \| u\|_{L^2((3/8,4); H^m(B(0,3/2))}\lesssim \| u\|_{L^2((1/4,4)\times B(0,2))},\] see the local energy estimates in Proposition \ref{pro:local_energy}.
	\begin{lemma}\label{lem:some bounds}
		Assume $2<p<\min\{2+\varepsilon, 2+\frac{4}{m(n+2)-2}\}$ for $\varepsilon>0$ from Proposition \ref{prop:invertibility_around_2}. Then the right hand side of equation \eqref{eq:inhom_for_v} belongs to $(E_{p'})'$ and it holds
		\[\bigg\| 2v+ f 
		+ \mbox{div}_m F+
		\sum_{k=1}^{m-1}\sum_{|\xi|=k}\partial^\xi F_\xi\bigg\|_{(E_{p'})'}\lesssim\| u\|_{L^2((1/4,4)\times B(0,2))}\]
		with constant depending on the ellipticity and dimensions. Therefore, by Proposition \ref{prop:invertibility_around_2},
		\[\|v\|_{E_p}=\Bigg\|\mathscr L ^{-1}\left( 2v+ f 
		+ \mbox{div}_m F+
		\sum_{k=1}^{m-1}\sum_{|\xi|=k}\partial^\xi F_\xi\right)\Bigg\|_{E_p}\lesssim \| u\|_{L^2((1/4,4)\times B(0,2))}.\]
	\end{lemma}
	This lemma holds also for $p<2$ satisfying $|p-2|<\varepsilon$, but since the knowledge of $v\in E^p$ does not seem to be any helpful if $p<2$, we omit this case here.
	\begin{proof}
		Fix $2<p<n+2$. By Corollary \ref{cor:embed} we find that for $\frac{1}{p'^*}=\frac{1}{p'}-\frac{1}{n+2}$ $$(L^{p'^*})'+L^{p}(\BBR;W^{-m,p}(\BBR^m)) \hookrightarrow (E_{p'})'.$$ Clearly, the dual exponent $q$ to ${p'^*}$ satisfying $1=\frac{1}{q}+\frac{1}{p'^*}$, is $q=p_\ast$, where $\frac{1}{p_\ast}=\frac{1}{p}+\frac{1}{n+2}$. 
		
		\noindent Let us list what do we know about the integrability of the involved functions.
		\begin{enumerate}[label=\upshape{(\roman{*})}]
			\item   $f\in  L^p(\BBR^{n+1})$ for $1\leq p\leq 2$ with 
			\[\|f\|_{L^p(\BBR^{n+1})}\lesssim_p \|u\|_{L^2((1/4,4)\times B(0,2))}.\]
			This is H\"older's inequality combined with the energy estimates from Proposition \ref{pro:local_energy}. 
			\item  $ v\in L^q(\BBR^{n+1})$ for $1\leq q\leq 2+\frac{4m}{n}$ with 
			\[\|v\|_{L^q(\BBR^{n+1})}\lesssim_q \|u\|_{L^2((1/4,4)\times B(0,2))}.\]
			This is possible thanks to the reversed H\"older estimates from Corollary \ref{cor:rev_hoelder} and a covering argument.
	
			\item For each $|\alpha|=m$, the function $F_\alpha$ can be rewritten as 
			\[F_\alpha=\sum_{|\beta|=m}a_{\alpha,\beta}\sum_{\gamma<\beta}\tilde c_{\beta,\gamma}\partial^\gamma (u\partial ^{\beta-\gamma}\chi)\]
			and for any $\gamma$ appearing in the sum, $|\gamma|=k<m$, we apply Corollary \ref{cor:embed} in order to get
			\[\hspace{1.2cm} \|\partial^\gamma (u\partial ^{\beta-\gamma}\chi)\|_{L^{2_{k}}(\BBR^{n+1})}\lesssim \|u\partial ^{\beta-\gamma}\chi\|_{L^{2^*}(\BBR^{n+1})}^{\frac{m-k}{m}}\|\nabla^m (u\partial ^{\beta-\gamma}\chi) \|_{L^2(\BBR^{n+1})}^\frac{k}{m}\lesssim \|u\|_{L^2((1/4,4)\times B(0,2))},\]
			where
			$\frac{1}{2_k}=\frac{1}{2}-\frac{m-k}{m}\frac{1}{n+2}\leq \frac{1}{2}- \frac{1}{m}\frac{1}{n+2}<\frac{1}{2}$. Last estimate uses $2^*=2+\frac{4}{n}<2+\frac{4m}{n}$ to bound $\|u\partial ^{\beta-\gamma}\chi\|_{L^{2^*}(\BBR^{n+1})}$ by Proposition \ref{cor:rev_hoelder}. 
			\item Let $\phi \in E_{p'}$. As $q=p'<2$, by Corollary \ref{cor:embed}, for integer $0\leq k \leq m$ and corresponding exponents $q_k$ it holds $\|\nabla^k \phi\|_{L^{q_k}(\BBR^n)}\lesssim \|\phi\|_{E_{p'}}$. Note that $p'< q_{m-1}< \dots <q_1< p'^*$ and so for integer $0\leq k\leq m-1$ we have by H\"older's inequality $\nabla^k \phi \in L^{q_{m-1}}_{loc}(\BBR^{n+1})$ with the bound
			\[\|\nabla^k \phi \|_{ L^{q_{m-1}}([c,d]\times B(0,R))}\lesssim \|\nabla^k \phi \|_{ L^{q_{k}}([c,d]\times B(0,R))}\] with constants depending on the size of the cylinder, as well, as $p'$, $m$ and $k$.
		\end{enumerate}
		
		Keeping those considerations in mind, we will now conclude $$2v+ f 
		+ \mbox{div}_m F\in L^{p_\ast}(\BBR^{n+1})+L^{p}(\BBR;W^{-m,p}(\BBR^m)).$$
		
		First, introduce the restriction $2<p\leq2^*$, which implies $\frac{1}{p_\ast}\ge \frac{1}{2^*}+\frac{1}{n+2}=\frac{1}{2}$, that is $p_\ast\leq 2$. Then $2v+f\in L^{p_\ast}(\BBR^{n+1})$ by (i) and (ii).
		Further, let us add the assumption $2<p\leq2_{m-1}$, that is $\frac{1}{p}\ge\frac{1}{2}-\frac{1}{m}\frac{1}{n+2}$.
		By (iii), (iv) and the fact that $\chi$ is compactly supported, we have for any $\gamma< \beta$, $|\beta|=m$
		\[\|\partial^\gamma (u\partial ^{\beta-\gamma}\chi)\|_{L^{2_{m-1}}(\BBR^{n+1})}\lesssim \|u\|_{L^2((1/4,4)\times B(0,2))}.\] Thus, for each $|\alpha|=m$, the compactly supported distribution $F_\alpha$ satisfies $F_\alpha \in L^{2_{m-1}}(\BBR^{n+1})$. For $p\leq2_{m-1}$ we then have $F_\alpha \in L^{p}(\BBR^{n+1})$, and so $\mbox{div}_m F\in L^{p}(\BBR;W^{-m,p}(\BBR^m))$ with  $$\|\mbox{div}_m F\|_{ L^{p}(\BBR;W^{-m,p}(\BBR^m))}\lesssim\|u\|_{L^2((1/4,4)\times B(0,2))}.$$
		
		Finally, under the assumptions we posed on $p$, it holds $(p')_{m-1}\ge2$. For any multi-index $|\xi|\leq m-1$ and $\psi\in E_{p'}\cap \mathscr D(\BBR^{n+1})$ we then have
		\[|\langle \partial^\xi F_\xi, \psi\rangle_{\mathscr D'(\BBR^{n+1}),\mathscr D'(\BBR^{n+1})}|\lesssim \|F_\xi\|_{L^2(\BBR^{n+1})}\|\nabla^{|\xi|}\psi\|_{L^2(\text{supp}\chi)}\stackrel{(iv)}{\lesssim} \|F_\xi\|_{L^2(\BBR^{n+1})}\|\psi\|_{E_{p'}},\] as $\chi$ is compactly supported. This shows $F_\xi \in (E_{p'})'$ and  
		\[\|F_\xi\|_{(E_{p'})'}\lesssim \|u\|_{L^2((1/4,4)\times B(0,2))}.\]
		This finishes the first part of the claim. Above reasoning shows also 
		\[ \bigg\| 2v+ f 
		+ \mbox{div}_m F+
		\sum_{k=1}^{m-1}\sum_{|\xi|=k}\partial^\xi F_\xi\bigg\|_{(E_{2})'}\lesssim\| u\|_{L^2((1/4,4)\times B(0,2))}.\] Recalling that $v\in E_2$ solves \eqref{eq:inhom_for_v} weakly, we have $v=\mathscr L ^{-1}|_{(E^2)'}$. By Proposition \ref{prop:invertibility_around_2} we thus have $$v=\mathscr L ^{-1}|_{(E_{p'})'}\left( 2v+ f 
		+ \mbox{div}_m F+
		\sum_{k=1}^{m-1}\sum_{|\xi|=k}\partial^\xi F_\xi\right) \in E_p\cap E_2.$$ The norm bound follows the derived estimates.
	\end{proof}
	From here, arguing as in the proof of \cite[Theorem 8.1]{ABES18}, Proposition \ref{prop:invertibility_around_2}, a fractional Poincar\'e inequality {\cite[Lemma 6.4]{ABES18}} and the well-known embedding of Campanato spaces into the space of H\"older continuous functions lead to \eqref{eq:goal_uniform_L^p}.	

Note that we have only been using the local energy estimates on $u$ from Proposition \ref{pro:local_energy} and their consequences, so the proof applies to any global weak solution $u$. We have obtained 
	\begin{theorem} \label{thm:higher integrability of solutions}
		Let $u\in L^2_{loc}(0,\infty; H^m_{loc}(\BBR^n))$ be a global weak solution of \eqref{eq:the parabolic equation} and let $$2<p<\min\left\lbrace2+\varepsilon, 2+\frac{4}{m(n+2)-2}\right\rbrace$$ for $\varepsilon>0$ from Proposition \ref{prop:invertibility_around_2}. Then it holds $u\in L^\infty_{loc}(0,\infty;L^p_{loc}(\BBR^n))\cap \mathscr C ^\alpha_{loc}(0,\infty; L^p_{loc}(\BBR^n))$ with $\alpha=1/p-1/2$. We have the estimates
		\begin{equation*}
		\sup_{s\in (t/2,2 t)}\left(\fint_{B(x,\sqrt[2m]{t})}|u(s,y)|^pdy\right)^{1/p} \lesssim \left(\fint_{(t/4,4 t)}\fint_{B(x,2\sqrt[2m]{t})}|u(s,y)|^2dy\right)^{1/2}
		\end{equation*}
		and 
		\begin{equation*}
		\sup_{s, s'\in (t/2,2 t)}{t^\alpha}\left(\fint_{B(x,\sqrt[2m]{t})} \frac{|u(s',y)-u(s,y)|^p}{|s'-s|^{\alpha p}}\right)^{1/p} \lesssim \left(\fint_{(t/4,4 t)}\fint_{B(x,2\sqrt[2m]{t})}|u(s,y)|^2dy\right)^{1/2}
		\end{equation*} for all $(t,x) \in \BBR^{n+1}_+$. Moreover, we have the $p$-integrability of the derivatives, in particular, it holds $\nabla^m u \in L^p_{loc}(\BBR^{n+1}_+)$ with 
		\begin{equation*}
		\left(\fint_{(t/2,2 t)}\fint_{B(x,\sqrt[2m]{t})}|\sqrt{s}\nabla^mu(s,y)|^pdy\right)^{1/p} \lesssim \left(\fint_{(t/4,4 t)}\fint_{B(x,2\sqrt[2m]{t})}|u(s,y)|^2dy\right)^{1/2}.
		\end{equation*}
		All constants depend only on the ellipticity and the dimensions.
	\end{theorem} 
 \section{Appendix}
 	\setcounter{theorem}{0} \setcounter{equation}{0}
 	\subsection{Integral operators $\Mmaxreg$ and $\mathcal{R}_L$} \label{sec:integral ops}
 		Let $L$ be an autonomous elliptic operator as in Section \ref{sec:aut}. Following \cite[Proposition 2.5]{AMP15}, we see that the maximal integral operator $$\Mmaxreg f(t,\cdot)= \int_0^t \nabla^m e^{-(t-s)L}\mbox{div}_m f(s,\cdot)ds,$$ initially defined as a mapping from $L^1(0,\infty;(H^{2m}(\BBR^n))^{M})$ to $L^\infty_{loc}(L^2)$, extends to a bounded functional on $L^2(L^2)$. This is a non-trivial result, which uses the consequences of the solution to the Kato square root conjecture \eqref{eq:Kato} and de Simon's regularity result \cite{dS64}.
 	We remind that $L^2(L^2)=T^{2,2}_m$ and derive an extension of $\Mmaxreg$ to some of the tent spaces $\tentspace$ if $p\neq 2$. This is used in Section \ref{sec:bound examples}. We closely follow \cite[Proposition 2.8]{AMP15}.
 	\begin{proposition}
		\label{prop:ext_M}
		Let $q=q_+(L^*)'\in[1,2)$ be as introduced in \ref{sec:aut}, that is  $$\sup_{t>0}\|\sqrt{t} \nabla^me^{-tL^*}\|_{\mathscr L (L^s)}<\infty \quad \mbox{for all} \quad 2\leq s<q'.$$ Then $\Mmaxreg$ extends to a bounded operator on $\tentspace$ for all $p\in (p_c,\infty]$, where $$p_c=\max\left\{\frac{nq}{n+mq};\frac{2n}{n+mq'}\right\}\leq \max\left\{1; \frac{2n}{n+2m}\right\}.$$
	\end{proposition}
	\noindent We believe that the exponent $p_c$ in Proposition \ref{prop:ext_M} is not optimal and through the analogy to the second order case can possibly by taken as the Sobolev exponent $p_c=\frac{nq}{n+mq}$, if one disposes of an improved version of \cite[Theorem 3.1]{AKMP12}.
	\begin{proof}
		We rely on the work \cite{AKMP12} about the boundedness of integral operators on tent spaces, which already contains the case of higher order operators. As in \cite{AMP15}, we want to apply \cite[Theorem 3.1]{AKMP12} with $\beta=0$ for $p\leq 2$ and \cite[Proposition 4.2]{AKMP12} for $p\ge 2$. For this we need to show the $L^r-L^2$ decay of the operator-valued integral kernel
			\[\nabla^m e^{-(t-s)L}\text{div}_m\] of $\Mmaxreg$,
		 which means an off-diagonal $L^r-L^2$ bound with the decay of order $$(t-s)^{-1-\frac{n}{2m}(\frac{1}{r}-\frac{1}{2})}(1+\frac{d(E,F)^{2m}}{t-s})^{-M},$$ for some $M>0$. We first remark that the second factor decays much slower then the exponential function appearing in the off-diagonal estimates and we have $$t\nabla^me^{-tL}\text{\text{div}}_m=\sqrt t\nabla^m e^{-\frac{1}{2}tL} (\sqrt t \nabla^m e^{-\frac{1}{2}tL^*})^*,$$ where each factor has $L^2$ off-diagonal bounds by Section \ref{sec:aut}. This argument provides the necessary conditions for \cite[Proposition 4.2]{AKMP12}. For $p\leq2$ the desired decay follows then by interpolation once we have shown for all $\tilde q\in (q,2]$. \[\sup_{t>0}\|t^{1+\frac{n}{2m}(\frac{1}{\tilde q}-\frac{1}{2})}\nabla^m e^{-tL} \text{\text{div}}_m\|_{\mathscr{L}(L^{\tilde q}, L^2)}<\infty.\] 
		 
		 Let us write  $t^{1+\frac{n}{2m}(\frac{1}{\tilde q}-\frac{1}{2})}\nabla^m e^{-tL} \text{\text{div}}_m= A_tB_tC_t$ with $A_t=\sqrt t\nabla^m e^{-\frac{1}{3}tL}$ uniformly bounded on $L^2(\BBR^n)$ as just mentioned, $B_t= t^{\frac{n}{2m}(\frac{1}{\tilde q}-\frac{1}{2})} e^{-\frac{1}{3}tL}$ and $C_t=\sqrt t e^{-\frac{1}{3}tL}\text{div}_m$. Recall from Section \ref{sec:aut} the relations between the exponents $q_{\pm}(L)$ and $p_{\pm}(L) $ associated to the semigroup. From this we deduce the $L^{\tilde q}-L^2$ boundedness of $B_t$. Indeed, we know $p_+(L^*)\ge \frac{nq'}{n-mq'}$ if $q'<\frac{n}{m}$ and is infinite otherwise. This implies $p_-(L)\leq \frac{nq}{n+mq}<\frac{2n}{n+mq'}$ if $q'<\frac{n}{m}$ and $p_-(L)=1$ otherwise, in both cases $p_-(L)\leq q<\tilde q$, which suffices to conclude the boundedness. Finally $\tilde q>q$ and the assumption give the boundedness of $C_t$ on $L^{\tilde q}$. Putting all this together we obtain the claim. 
	\end{proof}
	\begin{remark}
		\label{rem:sharpness}
		A sufficient condition for the assumption of Proposition \ref{prop:ext_M} to hold with $q'=\infty$ is that the semigroup $(e^{-tL^*})_{t>0}$ has a kernel $(b_t)_{t>0}$ satisfying
		\[\|\nabla^m b_t\|_{L^\infty(\BBR^n)}\leq c t^{-\frac{1}{2}}\] for some constant $c>0$. This follows directly from the Young's inequality
		\[\|g\ast h\|_{L^p}\lesssim \|g\|_{L^p}\|h\|_{L^1}\] for $g\in L^p(\BBR^n)$ and $h\in L^1(\BBR^n)$. 
	\end{remark}
	\begin{example}\label{ex:sharpness}
		The condition from Remark \ref{rem:sharpness} is satisfied for $L=(-1)^m\Delta^m$ for any dimension $n$ and order $m\ge 1$. Indeed, the case $m=1$ is trivial, since we consider the standard Gaussian kernel. The higher order is handled by estimating oscillatory integrals. See \cite[Lemma 2.4]{KL12} for the proof in case $m=2$ and adjust the powers to get the full result. Alternatively see the references in Section \ref{sec:kernel bounds} (i).
	\end{example}
 	\noindent For $f=(f_\beta)_{|\beta|=m}$ such that $f_\beta\in L^1(0,\infty;H^{m}(\BBR^n))$ let us also consider  
	\[ R_{L}f(t,\cdot)\coloneqq\sum_{|\beta|=m} \int_0^t e^{-(t-s)L}\partial^\beta f_\beta(s,\cdot)ds.\] 
	It is immediate that $ \mathcal{R}_L$ defines a continuous map into $L^\infty_{loc}(L^2)$. At least formally we see $$\nabla^m \mathcal{R}_L = \Mmaxreg.$$ This equality can be made rigorous in $\tentspace$ for $p>p_c$. 
	\begin{proposition} \label{prop:ext_R}
		For any $p\in (0,\infty]$ the operator $ \mathcal{R}_L$ extends to a bounded linear map from $\tentspace$ to $X^p_m$ and it holds $\tilde{\mathscr M }_L=\nabla^m \mathcal{R}_L$ on $\tentspace$ if $p_c<p<\infty$.
	\end{proposition}
	\begin{proof}
		The proof is a one-to-one copy of the one of \cite[Propositions 2.12 and 2.13]{AMP15}, relying on the $L^2$ off-diagonal estimates from Section \ref{sec:aut} and the Schur's Lemma. We omit the details.
	\end{proof}
	\subsection{Comparability of $\|\nabla^m u\|_{\tentspace}$ and $\|u\|_{X^p_m}$ - proofs}
	\label{sec:comparability proofs}
	We provide here the proof of the statements in Section \ref{sec:comparability results}
	\begin{proof}[Proof of Proposition \ref{prop:tent_kenig_comp_smallp}]
		\begin{claim}
			The statement is true for $L_0=(-1)^m\Delta^m$.
		\end{claim}
		\begin{claimproof}
			Instead of attempting a direct calculation, we remind that the kernel of the heat semigroup $(e^{-tL_0})_{t\ge0}$ satisfies kernel bounds (direct calculation, see Section \ref{sec:kernel bounds} (i) with help of the Fourier transform) and so $p_+(L_0)=\infty$. We note that by a change of variables the condition $\|\nabla^m u_f \|_{\tentspace}<\infty$ for $p\in(0,\infty)$ is exactly
		\[S_{h,L_0,0}f(x)\coloneqq \left(\int_0^\infty \int_{B(x,t)}|(t\nabla)^me^{-t^{2m}L_0}f (y)|^2\frac{dydt}{t^{n+1}}\right)^{1/2} \in L^p(\BBR^n).\] 
		
		The claim follows  then for any $p\in(0,\infty)$ by the equivalent characterizations of Hardy spaces associated to autonomous homogeneous operators from \cite{CMY16}, precisely Theorems 1.8, 1.4 and Corollary 3.11 with $k=1$. For this we only need to show that the $L^p(\BBR^n)$ norm of the non-tangential function $\mathcal N_{h,L_0}f$ used in \cite{CMY16} dominates the one we used in the definition of $X^p_m$ for any $f\in L^2(\BBR^n)$. For $f\in L^2$ and $v=e^{-t(-1)^m\Delta^m}f$ it is defined for every $x\in \BBR^n$
		$$\mathcal N_{h,L_0}f(x)=\left(\sup_{s>0}\fint_{B(x, s)}|v(s^{2m},y)|^2dy\right)^{1/2}.$$ For $\beta>0$ let us also introduce the non-tangential maximal function with changed angle
		$$\mathcal N^\beta_{h,L_0}f(x)=\left(\sup_{s>0}\fint_{B(x,\beta s)}|v(s^{2m},y)|^2dy\right)^{1/2}.$$ 
		
		By Vitali's Covering Lemma we obtain 
		\[\|\mathcal N^\beta_{h,L_0}f\|_{L^p(\BBR^n)}\leq C_{n,\beta}\|\mathcal N_{h,L_0}f\|_{L^p(\BBR^n)},\]
		for $\beta\ge 1$, any $f\in L^2(\BBR^n)$ and $p\in(0,\infty)$ (we provide details of this argument in the proof of Lemma \ref{lem:wp_small}).
		If we replace $s$ in those definitions by $\sqrt[2m]{t}$, we easily estimate
		\[\nontan v(x)=\sup_{\delta>0}\left(\fint_{\delta/2}^{\delta}\fint_{B(x,\sqrt[2m]{\delta})}|v(t,y)|^2dydt\right)^{1/2}\leq \beta^n \mathcal N^{\beta}_{h,L_0}f(x)\] for all $x\in \BBR^n$ with $\beta=\sqrt[2m]{2}$.
		Thus, $\|v\|_{X^{p}_m}\lesssim \|\nabla^m v\|_{T^{p,2}_m}$ holds as desired.
	\end{claimproof}
	
	We turn towards the non-autonomous case. The representation formula from Corollary \ref{cor:duhamel prop} gives
	\begin{equation*} 
	u_f(t,\cdot)=e^{-tL_0}f(\cdot)
	+\int_0^t e^{-(t-s)L_0}\text{div}_m (A(s,\cdot)-\underline A)\nabla^m u_s(t,\cdot)\; ds\end{equation*}
	with $\underline{A}$ being a matrix generating $L_0$. Using that the integral operator $\mathcal R_{L_0}$ is $\tentspace \to X^p_m$ bounded by Proposition \ref{prop:ext_R} for any $p\in(0,\infty]$, we estimate
	\begin{align*}\|u_f\|_{X^{p}_m}&\lesssim\|v\|_{X^p_m}+\|\mathcal {R}_{L_0}\|_{\mathscr L ({\tentspace},X^{p}_m)}\|A-\underline{A}\|_{L^\infty(\BBR^{n+1}_+)}\| \nabla^{m}u_f\|_{\tentspace}\\
	&\lesssim \| \nabla^{m}v\|_{\tentspace}+\| \nabla^{m}u_f\|_{\tentspace}
	\end{align*} with constants depending on ellipticity and dimensions. On the other hand, observe that the representation formula also gives us the bound
	\[\| \nabla^{m}v\|_{\tentspace}\lesssim \|\nabla^{m}u_f\|_{\tentspace}  +\|\mathcal {\tilde{M}}_{L_0}\|_{\mathscr L ({\tentspace})}\|A-\underline{A}\|_{L^\infty}\| \nabla^{m}u_f\|_{\tentspace}.\] Thus, by boundedness of $\mathcal {\tilde{M}}_{L_0}$ on $\tentspace$ when $p\in (\frac{n}{n+m},\infty)$, we conclude $\|u\|_{X^p_m}\lesssim\|\nabla^m u \|_{\tentspace}$. 
	\end{proof}
\begin{proof}[Proof of Proposition \ref{prop:comp_tent_kenig_smallp}]
	Let $p\in [1,2)$ and suppose that the operator $ L$ satisfies the strong ellipticity bounds \eqref{eq:strong lower ellipticity bounds}. Further, let $u\in X^p_m$ be a global weak solution to \eqref{eq:the parabolic equation}. We will show that $\nabla^m u \in \tentspace$ and 
			\[ \|\nabla^m u \|_{\tentspace}\lesssim \|u\|_{X^p_m}.\]
	\noindent We follow the main idea of \cite[Proposition 3.9]{CMY16} with appropriate adjustments, which are basically the same as those met in case $m=1$ by the authors of \cite{AMP15}. Thus, we work with a different non-tangential function as in \cite{CMY16}, construct special cut-off functions and rely on the a priori energy bounds from Proposition \ref{pro:local_energy}. 
	Let $\beta>0$ be a parameter to be determined later and 
	\[\nontanbeta u(x)\coloneqq \sup_{\delta>0}\left(\fint_{\delta^{2m}}^{\beta^{2m}\delta^{2m}}\fint_{B(x,\beta\delta)}|u(t,y)|^2dydt\right)^{1/2}.\] 
	
	By a covering argument we see that $\|\nontan u\|_{L^p}\sim_\beta \|\nontanbeta u\|_{L^p}$ (see the proof of Lemma \ref{lem:wp_small}). Let $\sigma>0$. We will denote
	\[E\coloneqq \left\{ x\in \BBR^n\mid\nontanbeta \leq \sigma \right\}\] and introduce \[E^*\coloneqq \left\{x\in E\mid |B(x,r)\cap E| \ge \frac{1}{2}|B(x,r)| \text{ for all } r>0 \right\}.\] Set also $B\coloneqq \BBR^n \setminus E$ and $B^* \coloneqq \BBR^n \setminus E^*$. For $0<\varepsilon < R <\infty$ we consider the truncated cones 
	\[\Gamma^{\varepsilon, R, \alpha}(x)\coloneqq \{(t,y)\in \BBR^{n+1}_+\mid t\in (\varepsilon, R) \text{ and } |y-x|<\alpha t \}\] 
	and define the saw-tooth region based at $E^*$ by setting $\mathcal R^{\varepsilon,R,\alpha}(E^*)= \cup_{x\in E^*} \Gamma^{\varepsilon, R, \alpha}(x)$. Note that this set is unbounded.
	By a change of variables we know that 
	\[\int_{E^*}\int_0^\infty\fint_{B(x,\frac{\sqrt[2m]{t}}{2})} 
	|\nabla^m u(t,y)|^2dydtdx\sim \int_{E^*}\int_0^\infty\fint_{B(x,\frac{s}{2})} 
	|s^m\nabla^m u(s^{2m},y)|^2\frac{dyds}{s}dx. \] 
	To estimate the second integral we note that by Fubini's Theorem and $\int_{E^*\cap B(y,r)}1/ r^ndx\lesssim_n 1$ we have
	\[\int_{E^*}\int_{2\varepsilon}^R\fint_{B(x,\frac{s}{2})} 
	|s^m\nabla^m u(s^{2m},y)|^2\frac{dyds}{s}dx\lesssim \int_{\mathcal R^{2\varepsilon,R,1/2}(E^*)} 
	|s^m\nabla^m u(s^{2m},y)|^2\frac{dyds}{s}\] for every $\varepsilon>0$ and $R>0$. We will estimate the last expression and let $\varepsilon \to 0$ and $R\to \infty$ eventually. The strategy is as follows. With help of a cut-off function supported in a slightly bigger saw-tooth region, $R^{\varepsilon,2R,1}(E^*)$, we will replace the domain of integration by the entire space $\BBR^{n+1}_+$. In the next step (and this is the only point where we use the stronger assumption) we will use the strong ellipticity condition \eqref{eq:strong lower ellipticity bounds} and the fact that $u$ is a global weak solution to be able to integrate by parts. We will be left with terms containing the time derivative of compactly supported functions, which will either force the corresponding integral to vanish or will give us sufficient decay in time to control the remaining term. The other integrals we will need to deal with will contain derivatives of $u$ of different orders integrated over the difference set $R^{\varepsilon,2R,1}(E^*)\setminus R^{2\varepsilon,R,1/2}(E^*)$. We handle those terms using the local energy estimates from Proposition \ref{pro:local_energy}.
	
	Consider the function $\chi\colon \BBR^{n+1}_+\to [0,1]$ given by 
	$$\chi(t,y)=\left(1-\eta\left(\frac{8 d_t (y,E^*)}{t}\right)\right)\eta\left(\frac{7}{2}\frac{t}{\varepsilon}\right)\left(1-\eta\left(\frac{7}{2}\frac{t}{R}\right)\right),$$ where $\eta\in \mathscr C^\infty(\BBR,[0,1])$ satisfies $\eta\equiv 0 $ on $[0,5]$ and $\eta\equiv 1 $ on $[7,\infty)$. Also, we denoted by $d_t (y,E^*)=\nu_{t/8}\ast  d(\cdot,E^*) (y)$ a smooth modification of the Euclidean distance function $d(y,E^*)$ with some standard mollifier $\nu_\varepsilon (\cdot)=\nu (\cdot/\varepsilon)$. We defined $\chi$ such that \begin{enumerate}[label=\upshape(\roman{*})]
		\item $\text{supp} \chi \in R^{\varepsilon,2R,1}(E^*)$ and $0\leq \chi \leq 1$, 
		\item $\chi \equiv 1$ on $\mathcal R^{2\varepsilon,R,1/2}(E^*),$
		\item $\chi \in \mathscr C^\infty(\BBR^{n+1}_+)$ with 
		\[|\partial_t \chi (t,y)|\lesssim \frac{1}{t} \quad \mbox{and} \quad |\nabla^k \chi (t,y) | \lesssim \frac{1}{t^k} \quad \mbox{for all} \quad k=0,\dots, m.\]
	\end{enumerate}
	\noindent Property (iii) follows from simple calculations and the properties of the support of $\chi$ and $\eta$. We only point out that for every $(t,y) \in R^{\varepsilon,2R,1}(E^*)$ there is $x\in E^*$ with $|x-y|<t$ and so
	\[|\nabla d_t (y,E^*)|\lesssim_\nu \frac{t+t/8}{t/8}\lesssim 1.\]
	Similarly for any multi-index $\alpha\in \BBN^n$, $|\partial^\alpha d_t(y,E^*)|\lesssim t^{1-|\alpha|}$. By the Dominated Convergence Theorem we also calculate 
	\begin{align*}\bigg|\frac{d}{dt} d_t (y,E^*)\bigg|& =\bigg|n \frac{1}{t}d_t (y,E^*) + \frac{8}{t^2}\frac{8^n}{t^n}\int \nabla \nu \left(\frac{8x}{t}\right)\cdot y\, d(y-x,E^*)dy\bigg|\\
	&\lesssim_n 1.\end{align*}
	\begin{claim}
	For $0<\varepsilon \ll R$ it holds $(t,y) \mapsto u(t^{2m},y)\chi^2 (t,y) \in L^2(\varepsilon, 2R; H^m(\BBR^n))$.
	\end{claim}
	
	Let us first observe the following. We constructed $\chi$ such that it separates $\mathcal R^{2\varepsilon,R,1/2}(E^*)$ from $\BBR^n\setminus\mathcal R^{\varepsilon,2R,1}(E^*)$, but changing the scalar factors in the definition we obtain the same claim for $\tilde \chi $ separating $\mathcal R^{\varepsilon,2R,1}(E^*)$ and $\BBR^n\setminus\mathcal R^{\varepsilon/2,4R,2}(E^*)$. Thus $\nabla^mu(t^{2m},x)\in L^2(\mathcal R^{\varepsilon,2R,1}(E^*))$. 
	
	By density, $u(t^{2m},y)\chi^2 (t,y)$ can be then used as a test function. Clearly,
		\begin{align*}
		I&=\int_{\mathcal R^{2\varepsilon,R,1/2}(E^*)} 
		|s^m\nabla^m u(s^{2m},y)|^2\frac{dyds}{s}  \leq \int_{\BBR^{n+1}_+} 
		|\chi (s,y) s^m\nabla^m u(s^{2m},y)|^2\frac{dyds}{s},
		\end{align*}
		hence by the strong ellipticity assumption \eqref{eq:strong lower ellipticity bounds}
	\begin{align*}
	I&\lesssim \frac{1}{\lambda} \text{Re} \int_{\BBR^{n+1}_+} s^{2m}
	\chi^2 (s,y) A(s^{2m},y) \nabla^m u(s^{2m},y) \cdot \overline {\nabla^m u(s^{2m},y)}\frac{dyds}{s}.
	\end{align*}
	
	We now use the product rule and then the equation to obtain 
	\begin{align*}
	I & \lesssim_{\lambda,m}  \text{Re} \int_{\BBR^{n+1}_+} 
	A(t,y) \nabla^m u(t,y) \cdot \overline {\nabla^m (\chi^2 (\sqrt[2m]{t},y) u(t,y))}dydt\\		     
	&- \text{Re}\sum_{|\alpha|=|\beta|=m} \sum_{\gamma < \alpha} c_{\gamma, \alpha} \int_{\BBR^{n+1}_+} a_{\alpha,\beta}(t,y) \partial^\beta u(t,y) \;\partial^{\alpha-\gamma}(\chi^2 (\sqrt[2m]{t},y))\;\overline
	{ \partial ^\gamma u(t,y)} dydt  \\  
	&\lesssim_{\Lambda, m, n}  A + B + \sum_{k=0}^{m-1} C_k,
	\end{align*}
	where 
	\begin{align*}
	A&= \bigg|\int_{\BBR_+} \left\langle 
	\partial _t(\chi (t,y)  u(t^{2m},y)), \chi (t,y) u(t^{2m},y)\right\rangle_{H^{-m}(\BBR^n),\sob}dt\bigg|,
	\\
	B&=\bigg|\int_{\BBR^{n+1}_+} 
	\partial _t(\chi (t,y))  u(t^{2m},y) \overline {\chi (t,y) u(t^{2m},y)}dydt\bigg|	,
	\\
	C_k&=\sum_{|\alpha|=|\beta|=m} \sum_{\gamma<\alpha,|\gamma|=k} \int_{\BBR^{n+1}_+} t^{2m} | \partial^\beta u(t^{2m},y)| |\partial^{\alpha-\gamma}(\chi^2 (t,y))|| \partial ^\gamma u(t^{2m},y)| \frac{dydt}{t}.
	\end{align*}
	
	Since $\chi$ is compactly supported in time, we obtain 
	\[ A\sim \int_0^\infty \partial_t \|\chi(t,\cdot) u(t^{2m},\cdot)\|_{L_2}^2=0.\]
	
	By properties (i)-(iii) of $\chi$, $B$ is bounded by
	\[B\lesssim \int_{R^{\varepsilon,2R,1}(E^*)\setminus R^{2\varepsilon,R,1/2}(E^*)}|u(t^{2m},y)|^2 \frac{dydt}{t}.\]
	
	Thus, we need to carefully estimate the integrals close to the boundary of the truncated cones. We have 
	\[R^{\varepsilon,2R,1}(E^*)\setminus R^{2\varepsilon,R,1/2}(E^*)\subseteq \tilde { \mathcal{B}} ^{\varepsilon, R}(E^*)=\tilde { \mathcal{B}} ^{\varepsilon}(E^*)\cup\tilde { \mathcal{B}} ^{ R}(E^*)\cup\tilde { \mathcal{B'}} (E^*),\] where we denote
	 $\tilde { \mathcal{B}} ^{\omega}(E^*)\coloneqq \{ (t,y)\in [0,\infty)\times \BBR^n\mid t\in (\omega,2\omega) \text{ and } d(y,E^*)<t\}$ for $\omega>0$, and  
	 $\tilde { \mathcal{B'}} (E^*)\coloneqq \{ (t,y)\in[0,\infty)\times \BBR^n\mid t\in (\varepsilon,2R) \text{ and } t/2\leq d(y,E^*)<t\}$.
	Note also that H\"older's inequality gives 
		\[ C_0 \lesssim \left(\int_{\tilde { \mathcal{B}} ^{\varepsilon, R}(E^*)}|t^m\nabla^mu(t^{2m},y)|^2 \frac{dydt}{t}\right)^{1/2} \left(\int_{\tilde { \mathcal{B}} ^{\varepsilon, R}(E^*)}|u(t^{2m},y)|^2 \frac{dydt}{t}\right)^{1/2},\] as well as, for $k=1,\dots, m-1$,
		\[ C_k \lesssim \left(\int_{\tilde { \mathcal{B}} ^{\varepsilon, R}(E^*)}|t^m\nabla^mu(t^{2m},y)|^2 \frac{dydt}{t}\right)^{1/2} \left(\int_{\tilde { \mathcal{B}} ^{\varepsilon, R}(E^*)}|t^k \nabla^ku(t^{2m},y)|^2 \frac{dydt}{t}\right)^{1/2}.\]
	
	Because of their similar structures, we estimate $B$ and $C_k$ simultaneously using local energy estimates. First, for any $(t,y)\in \tilde { \mathcal{B}} ^{\omega}(E^*)$ there is an $x\in E^*$ with $|x-y|<t$ and, by definition, $|E\cap B(y,2t)|\ge|E\cap B(x,t)|\ge \frac{1}{2}|B(x,t)| =\frac{1}{2} \omega_n t^n.$
	By Fubini's Theorem and for $\omega = \varepsilon$ or $\omega =R$ we estimate \begin{align*}
	\int_{\tilde { \mathcal{B}} ^{\omega}(E^*)}|u(t^{2m},y)|^2 \frac{dydt}{t}&\lesssim 
	\int_{\tilde { \mathcal{B}} ^{\omega}(E^*)}\left(\int_{E\cap B(y,2t)} t^{-n}dx\right)|u(t^{2m},y)|^2 \frac{dydt}{t}\\
	&\lesssim
	\int_{\omega}^{2\omega}\int_E \fint_{B(x,2t)} |u(t^{2m},y)|^2 \frac{dydxdt}{t}\\	
	&\lesssim
	\int_E \int_{\omega}^{4\omega}\fint_{B(x,4\omega)} |u(t^{2m},y)|^2 \frac{dydt}{t}dx\\	
	&\lesssim
	\int_E\int_{\omega^{2m}}^{4^{2m}\omega^{2m}} \fint_{B(x,4\omega)} |u(s,y)|^2 \frac{dyds}{s}dx\\
	&\lesssim
	\int_E\sup_{\omega >0}\fint_{\omega^{2m}}^{4^{2m}\omega^{2m}} \fint_{B(x,4\omega)} |u(s,y)|^2 dydsdx=\int_E | {\mathcal N}_{m,4}(x)|^2 dx.
	\end{align*} 
	
	Similarly, applying the local energy estimates from Proposition \ref{pro:local_energy} we obtain
	\begin{align*}
	\int_{\tilde { \mathcal{B}} ^{\omega}(E^*)}|t^k\nabla^ku(t^{2m},y)|^2 \frac{dydt}{t}
	&\lesssim
	\int_E\int_{\omega^{2m}}^{2^{2m}\omega^{2m}} \fint_{B(x,4\omega)} |s^{\frac{k}{2m}}\nabla^ku(s,y)|^2 \frac{dyds}{s}dx\\
	&\lesssim \int_E\frac{\omega^{2k}}{\omega^{4m}}\omega^{2m-2k}\int_{\omega^{2m}/2}^{2^{2m}\omega^{2m}} \fint_{B(x,8\omega)} |u(s,y)|^2 dydsdx\\
	&\lesssim
	\int_E\sup_{\omega >0}\fint_{\omega^{2m}/2}^{8^{2m}\omega^{2m}} \fint_{B(x,8\omega)} |u(s,y)|^2 dydsdx\\
	&\lesssim\int_E | {\mathcal N}_{m,8}(x)|^2 dx,
	\end{align*} 
	for any $k=1, \dots, m$. In the last step we used again a covering argument. To estimate the integrands on $\tilde { \mathcal{B'}} (E^*)$, let us consider the Whitney decomposition $\{B(x_k,r_k)\}_k$ of $B^*$. This covering has the properties
	\begin{enumerate}[label=\upshape{(\roman*)}]
		\item $B^*= \cup_k B(x_k,r_k)$;
		\item There are constants $c_1,\; c_2 \in (0,1)$ such that, for all $k$,
		\[c_1d(x_k,E^*)\leq r_k\leq c_2d(x_k,E^*).\]
		\item There is a constant $c_3>0$ such that, for all $x\in B^*$, 
		$\sum_k \mathbbm{1}_{B(x_k,r_k)}(x)\leq c_3.$
	\end{enumerate}
	
	We then estimate, performing a change of variables
	\begin{align*}
	\int_{\tilde { \mathcal{B'}} (E^*)}|u(t^{2m},y)|^2 \frac{dydt}{t}&\lesssim \sum_k \int_{r_k(\frac{1}{c_2}-1)}^{2r_k(\frac{1}{c_1}+1)}\int_{B(x_k,r_k)} |u(t^{2m},y)|^2 \frac{dydt}{t}\\
	&\lesssim \sum_k r_k^n \fint_{r_k^{2m}(\frac{1}{c_2}-1)^{2m}}^{2^{2m}r_k^{2m}(\frac{1}{c_1}+1)^{2m}}\fint_{B(x_k,\frac{c_2}{1-c_2}\sqrt[2m]{s})} |u(s,y)|^2 dyds.
	\end{align*}
	Since $E^*\subseteq E$, we have $d(x_k,E)\leq d(x_k,E^*)\leq \frac{r_k}{c_1}\leq \frac{c_2}{c_1(1-c_2)}\sqrt[2m]{s}$, for any $s \ge r_k^{2m}(\frac{1}{c_2}-1)^{2m}$. Thus there is an $x_k'\in E$ with $B(x_k, \frac{c_2}{(1-c_2)}\sqrt[2m]{s})\subseteq B(x_k',  \frac{c_2}{(1-c_2)}(\frac{1}{c_1}+1)\sqrt[2m]{s})$ and we can deduce 
	\begin{align*}
	\int_{\tilde { \mathcal{B'}} (E^*)}|u(t^{2m},y)|^2 \frac{dydt}{t}
	&\lesssim \sum_k r_k^n \fint_{r_k^{2m}(\frac{1}{c_2}-1)^{2m}}^{2^{2m}r_k^{2m}(\frac{1}{c_1}+1)^{2m}}\fint_{B(x_k',\frac{c_2}{(1-c_2)}(\frac{1}{c_1}+1)\sqrt[2m]{s})} |u(s,y)|^2 dyds \\
	&\lesssim \sum_k r_k^n \sup_{x\in E} |\nontanbeta u(x)|^2 \lesssim |B^*|\sup_{x\in E} |\nontanbeta u(x)|^2
	\end{align*} for some $\beta\ge 8 $ big enough, depending on $m,\,c_1$ and $c_2$ only. To estimate 
	\[\int_{\tilde { \mathcal{B'}} (E^*)}|t^k\nabla^ku(t^{2m},y)|^2 \frac{dydt}{t}\] we only need to apply Proposition \ref{pro:local_energy} and due to the proper scaling we arrive at the same bound. Thus, there is a $\beta>0$, depending on $\sigma$ and in particular independent of $\varepsilon>0$ and $R>0$, so that
	\[\int_{\mathcal R^{2\varepsilon,R,1/2}(E^*)} 
	|s^m\nabla^m u(s^{2m},y)|^2\frac{dyds}{s} \lesssim \int_E |\nontanbeta u(x)|^2 dx + |B^*|\sup_{y\in E} |\nontanbeta u(y)|^2. \] 
	
	We now complete the proof, following \cite[Theorem 7.3]{AMP15}. Taking limits $\varepsilon\to 0 $ and $R\to \infty$ we get
	\[\int_{E^*}\int_0^\infty\fint_{B(x,\frac{\sqrt[2m]{t}}{2})} 
	|\nabla^m u(t,y)|^2dydtdx \lesssim \int_E |\nontanbeta u(x)|^2 dx + |B^*|\sigma^2. \]
	Denoting $g_N(\sigma)=\big|\{x\in \BBR^n\mid\nontanbeta u (x) >\sigma\}\big| $ and $$g_S(\sigma)=\Bigg|\left\{ x\in \BBR^n\bigg| \,\left(\int_0^\infty \fint_{B(x,\frac{\sqrt[2m]{s}}{2})} |\nabla^m u (t,y) |^2 dydt\right)^{1/2}>\sigma\right\}\Bigg|$$ we have $|B^*| \lesssim |B|= g_N(\sigma)$ and $\int_E |\nontanbeta u(x)|^2 dx \leq 2 \int_0^\sigma t g_N(t) dt$. Putting this together,
	\begin{align*}g_S(\sigma)& \lesssim |B^*| + \frac{1}{\sigma^2}\int_{E^*}\int_0^\infty \fint_{B(x,\frac{\sqrt[2m]{t}}{2})}|\nabla^m u(t,y)|^2dydtdx\\
	& \lesssim|B^*| + \frac{1}{\sigma^2}\int_E |\nontanbeta u(x)|^2 dx
	\lesssim g_N(\sigma) + \frac{1}{\sigma^2} \int_0^\sigma t g_N(t) dt.
	\end{align*}
	Not that we have not used $p<2$ yet. Let us conclude
	\[\int_0^\infty \sigma^{p-1} g_S(\sigma) d\sigma \lesssim  \int_0^\infty \sigma^{p-1} g_N(\sigma) d\sigma+\int_0^\infty \sigma^{p-3}\int_0^{\sigma} t g_N(t) dt d\sigma \lesssim \int_0^\infty \sigma^{p-1} g_N(\sigma) d\sigma.
	\] This gives $\|\nabla^m u\|_{\tentspace} \lesssim \|\nontanbeta u\|_{L^p}\lesssim \|u\|_{X^p_m}.$
	We are left with proving the claim from above.
	\begin{claim}
			For $0<\varepsilon \ll R$ it holds $(t,y) \mapsto u(t^{2m},y)\chi^2 (t,y) \in L^2(\varepsilon, 2R; H^m(\BBR^n))$.
	\end{claim}
	\begin{claimproof}
	First, partition $[\varepsilon,2R]$ into finitely many intervals $[\delta^{2m},\beta^{2m}\delta^{2m}]$. Then, Fubini's Theorem and a priori energy estimates are applied as in the estimate of $$\int_{\tilde { \mathcal{B}} ^{\omega}(E^*)}|t^k\nabla^ku(t^{2m},y)|^2 \frac{dydt}{t},$$ $k=0,\dots, m$, to see $(t,y) \mapsto u(t^{2m},y)\chi^2 (t,y) \in L^2(\delta^{2m},\beta^{2m}\delta^{2m}; H^m(\BBR^n))$. For the $L^2(L^2)$ norm, as $p<2$, the obtained bound reads $$\int_E |\nontanbeta u (x)|^2 dx=   \int_E |\nontanbeta u (x)|^p  |\nontanbeta u (x)|^{2-p} dx \leq \sigma^{2-p} \int_E |\nontanbeta u (x)|^p  dx<\infty.\vspace{-0.82cm}$$
	\end{claimproof}
	
	\end{proof}
	\begin{proof}[Proof of Proposition \ref{prop:comp_tent_kenig_bigp}]
	Suppose $2\leq p\leq \infty$ and $u\in X^p_m$ is a global weak solution to \eqref{eq:the parabolic equation} with $ \|u\|_{X^p_m}<\infty$. By Remark \ref{rem:eqv norm tent space}, it holds $\|f\|_{\tentspace}\lesssim_p \| C(|f|)\|_{L^p}$, so it suffices to show 
	\begin{equation}\label{eq:goal} C(|\nabla^m u|)(y)=\sup_{B\ni y}\left(\int_0^{r^{2m}}\fint_B|\nabla^m u(t,x)|^2dtdx\right)^{1/2}\lesssim \left(\mathcal M_{HL}(\nontan u)^2\right)^{1/2}(y)\end{equation}  for $y\in \BBR^n$. To achieve this, we show for any ball $B(x_0,R)$
		\begin{align*}
		\int_{0}^{R^{2m}} \int_{B(x_0,R)} |\nabla^m u(t,x)|^2 dxdt \lesssim\int_{B(x_0,6R)} |\nontan (u)(x)|^2dx. 
		\end{align*} This estimate also proves the Proposition in case $p=2$.
	Let us first underline the fact that $u$ is assumed to be a global weak solution, whence it satisfies $u\in L^2_{loc}(0,\infty;H^m_{loc}(\BBR^n))$.
	
	Let $x_0\in \BBR^n$, $\varepsilon>0 $ and $R>0$. We choose $\chi$ to be a smooth cut-off function in time with $\chi\equiv 1$ on $[2\varepsilon,R^{2m}]$, supported in $[\varepsilon, (2R)^{2m} ]$ and satisfying
	\[|\partial_t \chi(t)|\lesssim \frac{1}{\varepsilon} \quad\mbox{if}\quad t\in [\varepsilon, 2\varepsilon]\] and 
	\[|\partial_t \chi(t)|\lesssim  \frac{1}{R^{2m}}\quad\mbox{if}\quad t\in [R^{2m}, (2R)^{2m}].\] Let also $\phi \in \testfunctions$ be such that  $\phi\equiv1 $ on $B(x_0,R)$ and $\phi\equiv 0$ on $\BBR^n\setminus B(x_0,2R)$. We require 
	$\|\partial^\alpha \phi\|_{L^\infty}\lesssim R^{-|\alpha|}$ for all $|\alpha|\leq m$.
	Then $\psi \coloneqq u (\chi\phi)^{2m}\in L^2(\varepsilon,(2R)^{2m};H^m_0(B(x_0,2R))$ is an admissible test function and an application of Lemma \ref{lem:absolute continuity} leads to
	\begin{align*}
	0=&2\text{Re} \int_{\varepsilon}^{(2R)^{2m}}\int_{B(x_0,2R)}  \partial_t(\chi^{2m})\phi^{4m}\chi^{2m} |u|^2dxdt
		\\
		&-2\text{Re} \int_{\varepsilon}^{(2R)^{2m}}\int_{B(x_0,2R)} A\nabla^m u \nabla^m (u\chi^{4m}\phi^{4m})dxdt.
	\end{align*}
	 
	 Here, we used $\partial_t \psi =(\chi\phi)^{2m}\partial_tu +\phi^{2m}u\partial_t\chi^{2m}$ in $L^2(0,\infty;H^{-m}(B(x_0,2R)))$. Let us denote the first summand from above by $J$. The same calculation as done for \eqref{claim1} in Proposition \ref{pro:local_energy} (notice that $\phi$ has proper decay) leads to  
	 \begin{align*}
	 	I&\coloneqq \int_0^\infty \chi^{4m} \int_{\BBR^n} |\nabla^m(u \phi^{2m})|^2  dx dt\\&\lesssim_{\lambda,\Lambda,m,n} |J|+ \sum_{k=0}^{m-1}  \int_0^\infty \frac{\chi^{4m}}{R^{2m-2k}} \int_{B(x_0,2R)\setminus B(x_0,R)}|\nabla^k u|^2 dx dt \\
	 	& =: |J| + \sum_{k=0}^{m-1} I_k,
	 \end{align*} 
	where clearly $I_k=  \int_0^\infty \chi^{4m}/R^{2m-2k} \int_{\BBR^n}|\nabla^k u|^2 dx dt$. 
	
	We first bound $|J|$. Up to constants depending on $m$, it holds by our assumptions on $\chi$
	\begin{align*}
	|J|&\lesssim \bigg|\int_0^\infty \int_{\BBR^n}\partial_t \chi(t) \mathbbm{1}_{B(x_0,2R)}(x)|u(t,x)|^2dxdt\bigg|\\
	&\lesssim \frac{1}{\varepsilon}\int_{\varepsilon}^{2\varepsilon}\int_{\BBR^n}|u(t,x)|^2\mathbbm{1}_{B(x_0,2R)}(x) dx dt +  \frac{1}{R^{2m}}\int_{R^{2m}}^{(2R)^{2m}}\int_{\BBR^n}|u(t,x)|^2\mathbbm{1}_{B(x_0,2R)}(x)dxdt\\
	&= \int_{\BBR^n}\fint_{\varepsilon}^{2\varepsilon}\fint_{B(y,\sqrt[2m]{2\varepsilon})}|u(t,x)|^2\mathbbm{1}_{B(x_0,2R)}(x)dxdtdy \\
	&+  \int_{\BBR^n}\fint_{R^m}^{(2R)^m}\fint_{B(y,2R)}|u(t,x)|^2\mathbbm{1}_{B(x_0,2R)}(x)dxdtdy,
	\end{align*} where the last equality follows by Fubini's Theorem. We realize that, up to a constant,  
	\begin{align*}
	|J|\lesssim \int_{\BBR^n} |\nontan((t,y)\mapsto u(t,y)\mathbbm{1}_{[0,(2R)^{2m}]}(t)\mathbbm{1}_{B(x_0,2R)}(y))(x)|^2dx.
	\end{align*} 
	
	If $p=2$ this is implies $|J|\lesssim \|u\|_{X^2_m}$. If $p\in(2,\infty]$, we make use of the triangular inequality to state that $B(x,\sqrt[2m]{\delta})\cap B(x_0,2R)\neq \emptyset$ for some $\delta<2 (2R)^{2m}$ implies $x\in B(x_0, 6R)$, thus
	\begin{equation} \label{eq:estimate_norm_of_non_tan}
	||\nontan((t,y)\mapsto u(t,y)\mathbbm{1}_{[0,(2R)^{2m}]}(t)\mathbbm{1}_{B(x_0,2R)}(y))||_{L^2(\BBR^n)}^2\lesssim \int_{B(x_0,6R)} |\nontan (u)(x)|^2dx.
	\end{equation} 
	
	We turn to our main estimate. Let us observe that we could have also chosen te support of $\phi$ to be contained in some $B(x_0,\overline{r})$ for $\overline{r} <2R $ and require that $\phi =1$ on $B(x_0,\underline{r})$ for some $\underline{r} <\overline{r}$. This would change the factor $R^{-(2m-2k)} $ in $I_k$ to $(\overline{r}-\underline{r})^{-(2m-2k)} $. We thus have for all $0<\underline{r} <\overline{r}\leq 2R$ the following inequality
	\begin{align*}
	\int_0^\infty  \chi^{4m} \int_{B(x_0,\underline{r})} |\nabla^m u|^2 dxdt
	& \lesssim \sum_{k=0}^{m-1} \frac{1}{(\overline{r}-\underline{r})^{2m-2k}}
	\int_0^\infty \chi^{4m} \int_{B(x_0,\overline{r})\setminus  B(x_0,\underline{r})} |\nabla^k u|^2 dxdt \notag\\&
	+\int_{B(x_0,6R)} |\nontan (u)|^2dx. \notag
	\end{align*}
	
	Regarding the second summand as a constant, we repeat the iteration procedure from Proposition \ref{pro:local_energy} (see the Claim therein) based on the technique by Barton (cf.\ \cite{B16}). As a result we end up with 
	\begin{align}
	\int_0^\infty  \chi^{4m} \int_{B(x_0,R)} |\nabla^m u|^2 dxdt & \lesssim \frac{1}{R^{2m}}
	\int_0^\infty \chi^{4m} \int_{B(x_0,2R)} |u|^2 dxdt
	+\int_{B(x_0,6R)} |\nontan (u)|^2dx, \label{eq:claim10}
	\end{align}
	where the constants depend only on ellipticity and dimensions. Thus, we are left with estimating $\tilde I_0 =\frac{1}{R^{2m}}
	\int_\varepsilon ^{(2R)^{2m}} \int_{B(x_0,2R)} |u|^2 dxdt$. We choose $K=K_\varepsilon \in \BBN$ such that $ 2^{K-1}\varepsilon<(2R)^{2m}\leq 2^{K}\varepsilon$ and average in space to obtain
	\begin{align*}
	\tilde I_0& \leq \frac{1}{R^{2m}} \sum_{k=1}^{K} \int_{2^{k-1}\varepsilon}^{2^{k}\varepsilon} \int_{\BBR^n}  |u(t,x)\mathbbm{1}_{B(x_0,2R)}(x)|^2 dxdt 
	\\&= \frac{1}{R^{2m}} \int_{\BBR^n}\sum_{k=1}^{K} \int_{2^{k-1}\varepsilon}^{2^{k}\varepsilon} \fint_{B(y,\sqrt[2m]{2^k\varepsilon})}  |u(t,x)\mathbbm{1}_{B(x_0,2R)}(x)|^2 dxdtdy  \\
	&\lesssim  \frac{1}{R^{2m}} \int_{\BBR^n}\sum_{k=1}^{K} 2^k\varepsilon \fint_{2^{k-1}\varepsilon}^{2^{k}\varepsilon} \fint_{B(y,\sqrt[2m]{2^k\varepsilon})}  |u(t,x)\mathbbm{1}_{B(x_0,2R)}(x)|^2 dxdtdy 
	\\
	&\lesssim  \frac{1}{R^{2m}}\left(\sum_{k=1}^{K} 2^k\varepsilon\right) \int_{\BBR^n}
	|\nontan((t,y)\mapsto u(t,y)\mathbbm{1}_{[0,(2R)^{2m}]}(t)\mathbbm{1}_{B(x_0,2R)}(y))(x)|^2 dx.
	\end{align*}
	Using \eqref{eq:estimate_norm_of_non_tan} and $$\sum_{k=1}^{K} 2^k\varepsilon\leq 2^{K+1}\varepsilon < 4 (2R)^{2m}$$ 
	we obtain from \eqref{eq:claim10} 
	\begin{align}\label{eq:claim11}
	\int_{2\varepsilon}^{R^{2m}} \int_{B(x_0,R)} |\nabla^m u(t,x)|^2 dxdt \lesssim\int_{B(x_0,6R)} |\nontan (u)(x)|^2dx 
	\end{align} 
	with constants independent on $\varepsilon>0$ and $R>0$. We first let $\varepsilon \to 0$. If $p=2$ then $R\to\infty$ finishes the proof. In other case, we divide both sides of \eqref{eq:claim11} by $R^n$ and obtain to obtain the claimed estimate \eqref{eq:goal}, which finishes the proof by the boundedness of the Hardy--Littlewood maximal function. 
	\end{proof}
\section*{Acknowledgments} I would like to thank my Master's thesis supervisor Herbert Koch, for his support, sharing with me his intuition and the fruitful discussions about the tent space well-posedness. My gratitude belongs also to Pierre Portal, who brought this project to my attention and whom I could always contact with questions. Finally, I owe my gratitude to Pascal Auscher, Moritz Egert, Xavier Ros-Oton, Olli Saari and Christoph Thiele for the conversations we had and their comments. 
	Finally, I would like to acknowledge helpful suggestions made by the anonymous referee.

\end{document}